\documentclass[12pt]{article}
\usepackage[T1]{fontenc}
\usepackage[utf8]{inputenc}
\usepackage{lmodern}
\usepackage{amsmath,amssymb,amsthm,mathtools,bm}
\usepackage{array,booktabs,float}
\usepackage{enumitem}
\usepackage{graphicx}
\graphicspath{{fig/}{Paper1:Meanfunction/fig/}}
\usepackage{microtype}
\usepackage{xcolor}
\usepackage{xurl}
\usepackage[round,authoryear]{natbib}
\usepackage[hidelinks]{hyperref}
\usepackage[nameinlink,capitalize,noabbrev]{cleveref}
\newcommand{\SuppCref}[1]{\Cref*{#1} in the supplementary material}
\newcommand{\SuppLemmaCref}[1]{Lemma~\ref*{#1} in the supplementary material}

\newcommand{\SuppPropositionCref}[1]{Proposition~\ref*{#1} in the supplementary material}

\newtheorem{theorem}{Theorem}[section]
\newtheorem{proposition}[theorem]{Proposition}
\newtheorem{lemma}[theorem]{Lemma}
\newtheorem{corollary}[theorem]{Corollary}
\theoremstyle{definition}

\newtheorem{assumption}[theorem]{Assumption}
\newtheorem{remark}[theorem]{Remark}
\numberwithin{equation}{section}

\crefname{theorem}{Theorem}{Theorems}
\Crefname{theorem}{Theorem}{Theorems}
\crefname{proposition}{Proposition}{Propositions}
\Crefname{proposition}{Proposition}{Propositions}
\crefname{lemma}{Lemma}{Lemmas}
\Crefname{lemma}{Lemma}{Lemmas}
\crefname{corollary}{Corollary}{Corollaries}
\Crefname{corollary}{Corollary}{Corollaries}
\crefname{assumption}{Assumption}{Assumptions}
\Crefname{assumption}{Assumption}{Assumptions}
\crefname{definition}{Definition}{Definitions}
\Crefname{definition}{Definition}{Definitions}
\crefname{remark}{Remark}{Remarks}
\Crefname{remark}{Remark}{Remarks}
\crefname{example}{Example}{Examples}
\Crefname{example}{Example}{Examples}

\newcommand{\E}{\mathbb{E}}
\newcommand{\argmin}{\operatorname*{arg\,min}}
\newcommand{\R}{\mathbb{R}}
\newcommand{\cS}{\mathcal{S}}
\newcommand{\cF}{\mathcal{F}}
\newcommand{\cI}{\mathcal{I}}
\newcommand{\veps}{\varepsilon}
\newcommand{\btheta}{\bm{\theta}}
\newcommand{\bbeta}{\bm{\beta}}
\newcommand{\bA}{\bm{A}}
\newcommand{\bB}{\bm{B}}
\newcommand{\bZ}{\bm{Z}}

\newcommand{\MainCref}[1]{\Cref*{#1}}
\newcommand{\MainEqref}[1]{equation~(\ref*{#1})}
\newcommand{\PP}{\mathbb{P}}
\newcommand{\Var}{\operatorname{Var}}
\newcommand{\Cov}{\operatorname{Cov}}
\newcommand{\norm}[1]{\left\|#1\right\|}
\newcommand{\mnorm}[1]{\left\|#1\right\|_{\max}}
\newcommand{\oinf}[1]{\left\|#1\right\|_{\infty}}

\allowdisplaybreaks

        \date{}
\begin{document}

\def\spacingset#1{\renewcommand{\baselinestretch}%
		{#1}\small\normalsize} \spacingset{1}


		\title{\bf   Mean estimation in spatial functional data 

        }
        \author{Holger Dette \thanks{Department of Mathematic, Ruhr University Bochum Email: 
holger.dette@rub.de.}
  \and
  Qirui Hu\thanks{School of Statistics and Data Science, Shanghai University of Finance and Economics\\ Department of Mathematic, Ruhr University Bochum. Email: huqirui@mail.sufe.edu.cn}}
        \maketitle

	\bigskip

\begin{abstract}
Functional mean estimation exhibits a sparse-to-dense phase transition as the number of measurements per curve increases. For spatially indexed functional data, however, dependence between nearby curves and increasing spatial sampling density fundamentally alter this transition. We investigate this phenomenon under a unified framework allowing both increasing-domain and infill asymptotics, with irregular temporal observations. Our analysis identifies two distinct sources of estimation error: an observation-level component governed by the total number of measurements and a curve-level component governed by a spatial dependence-adjusted effective sample size. The resulting uniform convergence rates reveal how spatial dependence shifts the temporal phase-transition boundary and how spatial infill can improve estimation through additional measurements without a corresponding increase in effective curve-level information. We establish these results for three spline-based estimators: a pooled estimator of a stationary mean curve, a local-kernel estimator, and a tensor-product estimator of a spatially varying mean surface. The effects of spatial localization on the phase transition are characterized explicitly. Simulations illustrate the interplay between spatial dependence and temporal sampling density, and environmental data applications demonstrate the proposed estimation methods.

\end{abstract}

\noindent%
\vfill

\newpage
\spacingset{1.1}

\section{Introduction}\label{sec:intro}

The accuracy of functional mean estimation depends critically on the number of measurements available within each curve. With sparse temporal observations, estimation is governed by nonparametric smoothing, whereas sufficiently dense observations allow curve-level variation to become the dominant source of uncertainty. For spatially indexed functional data, this transition is further affected by dependence between curves observed at different locations. In particular, increasing the spatial sampling density may substantially increase the total number of measurements without providing a comparable increase in independent information about the underlying functional process. This raises a fundamental question: {\it How do spatial dependence and spatial sampling density affect the sparse-to-dense transition in functional mean estimation?} The problem arises naturally in environmental monitoring, where temperature, pollution, and water-quality profiles are observed at spatially dependent locations and often at irregular time points.

For independent functional data, the effect of within-curve sampling is well understood.  Sparse longitudinal methods were developed by \citet{YaoMullerWang2005}, and \citet{HallMullerWang2006} studied the role of the sampling plan in functional principal component analysis.  For functional  mean estimation, \citet{CaiYuan2011} established minimax $L^2$-rates and identified the transition to the root-sample-size rate available from complete trajectories.  \citet{ZhangWang2016} provided a unified local-smoothing theory across sparse, semi-dense, and ultra-dense designs.  Related developments include simultaneous inference for dense functional means \citep{CaoYangTodem2012}, optimal sup-norm rates under synchronous designs \citep{BergerHermannHolzmann2025}, non-asymptotic high-dimensional phase transitions \citep{GuoLiQiaoWang2025}, and simultaneous inference across sampling regimes \citep{CaiHu2025mean}.  Together these results show how the number and placement of measurements within a curve determine whether the estimate of mean behaves as a nonparametric or parametric object.

Dependence across curves introduces a second source of information loss.  Functional time-series methods describe linear and weakly dependent processes \citep{Bosq2000,HoermannKokoszka2010}, mean inference \citep{HorvathKokoszka2012,HorvathKokoszkaReeder2013,ChenSong2015,LiYang2023}, frequency-domain structure \citep{PanaretosTavakoli2013}, high-dimensional panels \citep{ZhouDette2023}, and locally stationary processes \citep{vanDelftEichler2018,DetteWu2026}.  Sparse-to-dense transitions have also been studied for structural-break detection and relevant-hypothesis testing in discretely observed functional time series \citep{CaiHu2024FTSBreaks,CaiHu2025RelevantFTS}.  Spatial functional data, however, are indexed by pairing with  multidimensional geometry rather than an ordered temporal lag.  Functional kriging, spatial functional statistics, and spatially dependent principal components have been developed by \citet{DelicadoGiraldoComasMateu2010}, \citet{GiraldoDelicadoMateu2011}, \citet{MartinezHernandezGenton2020}, and \citet{ZhangLi2022}; see also the broader second-generation framework of \citet{KonerStaicu2023}.

While the sparse-to-dense transition is well understood for independent functional observations, much less is known about how it changes under spatial dependence and simultaneous infill and increasing-domain sampling. A central difficulty is that spatial replication generates two different information scales. Additional nearby sites contribute new temporal observations, but their underlying functional trajectories may contain largely redundant information. Consequently, the total number of measurements and the effective number of spatial curves can grow at substantially different rates. As we demonstrate below, this distinction changes the temporal sampling density at which the leading estimation error switches from an observation-limited to a curve-limited regime. Understanding and quantifying this shift is the principal objective of the present paper.

At sites $S_i = (S_{i,1},\dots, S_{i,d})^\top \in\R^d$, $1\le i\le m_n$, and times $T_{ij}\in\left[0,1\right]$, we observe
\begin{equation}\label{eq:model}
Y_{ij}=\eta\left(S_i,T_{ij}\right)+\sigma\left(S_i,T_{ij}\right)\veps_{ij},
\qquad 1\le i\le m_n,\quad 1\le j\le N_i,
\end{equation}
where $\eta\left(s,t\right)=m\left(s,t\right)+X\left(s,t\right)$ and $\E\left\{X\left(s,t\right)\right\}=0$.  The random field $X$ represents spatially dependent curve-to-curve variation and $\veps_{ij}$ is an independent measurement error.  We consider both a stationary mean $m(t)=m(s,t)$, shared by all sites, and a locally stationary mean surface $M(u,t)$, where $M(u,t)$ approximates the spatially
varying mean $m(s,t)$ after appropriately rescaling the expanding spatial domain to the cube $[0,1]^d$; see \cite{Kurisu2022}.
 This decomposition separates
spatially dependent curve-level variation from observation-level variation generated
by random sampling times and measurement errors, and this distinction is reflected
explicitly in the convergence rates.

Our main contribution is to characterize how spatial dependence and the spatial sampling density affect the sparse-to-dense transition in functional mean estimation. We establish dependence-adjusted uniform convergence rates that reveal two distinct information scales. The observation-level component is governed by the nominal number of sampled measurements, whereas the curve-level component is governed by a dependence-adjusted effective sample size \(m_{\mathrm{eff},n}=m_n/\Theta_n\), where \(\Theta_n\) quantifies the inflation caused by spatial dependence. The balance between these two components determines the temporal sampling density at which the dominant source of estimation error changes. Under increasing-domain sparse-space sampling, the effective sample size is of the same order as the nominal number of curves, recovering the familiar sparse-to-dense transition. Under spatial infill with persistent local dependence, however, the effective sample size is governed by the spatial domain volume and can therefore be substantially smaller than the nominal number of curves. Consequently, additional spatial observations may continue to improve the observation-level component while contributing comparatively little independent curve-level information. This shifts the temporal phase-transition boundary and reveals an interaction between spatial and temporal sampling that cannot be captured by simply replacing the nominal sample size with an effective sample size. For spatially varying means, spatial localization additionally affects the transition through the kernel-window volume or the spatial-spline resolution.

To develop this theory for different mean structures, we consider three spline-based
estimators. A pooled least-squares estimator recovers the stationary mean, a local-kernel
estimator recovers the mean curve at an interior spatial point, and a tensor-product
spatial spline estimates the full mean surface. Their linear representations reduce
the stochastic analysis to empirical Gram matrices and spatially dependent site-level
scores and allow the same distinction between observation-level and curve-level
information to be traced across the three estimation problems.

We complement the theoretical results by a simulation study that varies sample size,
spatial geometry and  temporal sampling density, together with
additional distributional robustness experiments. Three environmental applications
illustrate the estimators under regular-grid, irregular-network, and highly time-sparse
sampling designs.

\section{Statistical framework and methodology}\label{sec:problem}\label{sec:methodology}

We consider the model \eqref{eq:model} and assume that the spatial locations $S_i$ are fixed and the within-site sample sizes $N_i$  are treated as fixed whereas the observation times $T_{ij}$  and the measurement errors $\varepsilon_{ij}$ are random. { For each fixed site, $(\omega,t)\mapsto X(s,t;\omega)$ is jointly measurable with respect to the $\sigma$-field generated by that curve and the Borel sets of $[0,1]$.}
More specifically, we make the following assumptions regarding the measurement locations and errors in model \eqref{eq:model}.

\begin{assumption} [Time design]
\label[assumption]{ass:time}
The $N_i$ define the number of measurements of each location.  The variables $\left\{T_{ij}:1\le i\le m_n,1\le j\le N_i\right\}$  are i.i.d.\ with density $f_T$ on $\left[0,1\right]$ satisfying $0<c_f\le f_T\left(t\right)\le C_f<\infty$ for all $t\in\left[0,1\right]$, and the entire collection of time points is independent of the latent random field $X$.  The errors $\veps_{ij}$ are i.i.d., centered, variance one, and independent of the latent field $X$ and of the $\{T_{ij}\}$.
\end{assumption}

\begin{assumption}[Infill-increasing spatial design]\label[assumption]{ass:design}
For each $n$, let $L_{1,n}, \dots , L_{d,n}>0$ and  $\cS_n=\prod_{r=1}^d\left[0,L_{r,n}\right]$, the spatial domain with volume, $V_n=|\cS_n|=\prod_{r=1}^dL_{r,n}$, and $L_n=V_n^{1/d}$.  For $\delta_n>0$, partition $\R^d$ into cubes $\Delta_{n,\ell}=\prod_{r=1}^d\left[\ell_r\delta_n,\left(\ell_r+1\right)\delta_n\right)$, $\ell = (\ell_1,\dots, \ell_d) \in\mathbb Z^d$, and set $\cI_n=\left\{\ell\in\mathbb Z^d:\Delta_{n,\ell}\cap\cS_n\ne\varnothing\right\}$.  Each $\ell\in\cI_n$ contains exactly one spatial location $S_\ell\in\Delta_{n,\ell}\cap\cS_n$.  Thus $\left\{S_i\right\}_{i=1}^{m_n}=\left\{S_\ell:\ell\in\cI_n\right\}$ and $m_n=|\cI_n|\asymp V_n\delta_n^{-d}$.  
\end{assumption}

Note that $V_n$ measures the size of the observation domain, whereas $\delta_n$ controls the spatial sampling resolution.
Throughout the asymptotic theory, we consider  an increasing-domain, that is
$V_n\to\infty$, and distinguish two spatial regimes:   
(i) if $\delta_n\asymp1$  we refer to this as an 
    {\it  sparse space} regime; 
(ii) if $\delta_n\to0$ we refer to this as a {\it dense space} regime.

\begin{assumption}[Full-dimensional domain growth]\label[assumption]{ass:full-dimensional}
There exist constants $0<c_L<C_L<\infty$ such that
$c_L\le L_{r,n}/L_n\le C_L$ for every $1\le r\le d$ and all sufficiently large $n$.
\end{assumption}

Let $M_{1,n}=\sum_{i=1}^{m_n}N_i$ denote the number of all observations and $M_{2,n}=\sum_{i=1}^{m_n}N_i\left(N_i-1\right)$, and set $\bar N_{n}=M_{1,n}/m_n$, and  $\bar N_{2,n}=M_{2,n}/m_n$.
Thus $\bar N_n$ is the average number of observations per curve; its growth relative to the dependence-adjusted effective size determines the time-sampling regime.
For the estimation of the mean function $m : \mathbb R^d \times [0,1] \to \mathbb R$, defined by  $m(s,t) = \mathbb E [ \eta (s,t)] $,   we will distinguish the cases of (spatial) {\it stationarity} , that is 
 $m(s,t) =m(t)$ (for all $s$) 
 and a {\it locally stationary mean} that can be approximated, after rescaling the spatial domain, by a smooth surface $M(u,t)$, see \cite{Kurisu2022}.

Our approach is  based on  a B-spline estimator. More specifically,  we fix an integer   $p\ge2$, and let $\bB_J\left(t\right)=\left(B_{1,J}\left(t\right),\ldots,B_{J+p,J}\left(t\right)\right)^\top$ denote the normalized B-spline basis of order $p$ on a quasi-uniform partition of the interval $\left[0,1\right]$ with $J$ interior knots.   
Throughout this paper, $J_m=J_{m_n}$, depends on $m_n$, denotes the time-direction spline resolution, and we use the notation $\bB_m\left(t\right)=\bB_{J_m}\left(t\right)$ where the dimension of this vector is $K_m = J_m +p$. For further details on B-spline theory and its use in regression, see, for example, \citet{deBoor2001}.

\subsection{Stationary mean}\label{sec:method-stationary}
In this section, we consider the model (\ref{eq:model}) with a stationary mean, that is $m\left(s,t\right)=m\left(t\right)$, for all $s \in \R^d$. The spatial locations do not enter the target, so observations from all sites can be pooled to estimate a single temporal curve.

  Accordingly, pooling all observations gives the least-squares projection onto the time-spline space, that is 
\begin{equation}\label{eq:mean-est}
\hat m_{J_m}\left(t\right)=\bB_m\left(t\right)^\top \hat\bbeta_{m,J_m},
\end{equation}
where 
\begin{equation*}
\hat\bbeta_{m,J_m}:=\argmin_{\bbeta\in\mathbb R^{K_m}}
\sum_{i=1}^{m_n}\sum_{j=1}^{N_i}\left\{Y_{ij}-\bB_m\left(T_{ij}\right)^\top\bbeta\right\}^2.
\end{equation*}
Equivalently, $\hat\bbeta_{m,J_m}=\hat Q_{m,J_m}^{-1}\hat r_{m,J_m}$, where the  empirical Gram  matrix $\hat Q_{m,J_m}$ and the vector  $\hat r_{m,J_m}$ are given by
\begin{equation*}
\hat Q_{m,J_m}:=\frac1{M_{1,n}}\sum_{i=1}^{m_n}\sum_{j=1}^{N_i}\bB_m\left(T_{ij}\right)\bB_m\left(T_{ij}\right)^\top,
\qquad
\hat r_{m,J_m}:=\frac1{M_{1,n}}\sum_{i=1}^{m_n}\sum_{j=1}^{N_i}\bB_m\left(T_{ij}\right)Y_{ij}.
\end{equation*}
The corresponding population projection is
given by $
m_{J_m}\left(t\right):=\bB_m\left(t\right)^\top\bbeta_{m,J_m},$  
where
\begin{align*}
\bbeta_{m,J_m}:=Q_{m,J_m}^{-1}r_{m,J_m}~,~~
Q_{m,J_m}:=\E\left\{\bB_m\left(T\right)\bB_m\left(T\right)^\top\right\}~,~~
r_{m,J_m}:=\E\left\{\bB_m\left(T\right)m\left(T\right)\right\},
\end{align*}
and $T$ has the density $f_T$ specified in \Cref{ass:time}. 
This projection satisfies the orthogonality condition
\begin{equation}\label{eq:mean-orthogonality}
\E\left[\bB_m\left(T\right)\left\{m\left(T\right)-m_{J_m}\left(T\right)\right\}\right]=0.
\end{equation}
Note that the normal equations give an exact centering for the stochastic expansion.  More precisely, with the notation
\begin{equation*}
\xi_{i,J_m}:=\sum_{j=1}^{N_i}\bB_m\left(T_{ij}\right)\left\{Y_{ij}-m_{J_m}\left(T_{ij}\right)\right\}\in\mathbb R^{K_m}
\end{equation*}
we have, by Assumption \ref{ass:time} and (\ref{eq:mean-orthogonality}), $\E \xi_{i,J_m}= 0$, and 
\begin{equation}\label{eq:mean-score-avg}
\hat r_{m,J_m}-\hat Q_{m,J_m}\bbeta_{m,J_m}
= M_{1,n}^{-1}\sum_{i=1}^{m_n}\xi_{i,J_m}.
\end{equation}

The estimator \eqref{eq:mean-est} pools data from all sites and smooths the mean over time $t$.  The number $J_m$ (of interior knots) controls bias-variance trade-off: a larger value of $J_m$ value of reduces the bias but increases the variance.   More observations within a curve reduce observation-level variation. Once the curves are sampled sufficiently densely in time, however, the dominant uncertainty arises from the latent curve-level variation, for which spatial dependence reduces the effective amount of replication. Section \ref{sec:theory} makes this distinction more precise.

\subsection{Locally stationary mean}\label{sec:local-stationarity-assumptions}
In this  section, we consider the problem of estimating the spatially varying mean surface $M(u,t)$ without assuming spatial stationarity of the mean. 
When the mean changes gradually over space, pooling all sites would obscure spatial heterogeneity.  We instead represent the mean on a rescaled domain by a  smooth function $M:\left[0,1\right]^d\times\left[0,1\right]\to\R$, that is
\begin{align} \label{hol5}
  \sup_{s\in\cS_n,t\in\left[0,1\right]}|m\left(s,t\right)-M\left(U\left(s\right),t\right)|\le CL_n^{-1} , 
\end{align}
where $U\left(s\right)=\left(s_1/L_{1,n},\ldots,s_d/L_{d,n}\right)^\top$.
This approximation is in the spirit of local-stationarity conditions for spatial regression and random fields \citep{Kurisu2022,LahiriMcElroyWeinberg2025}. It  allows the mean function to depend on the spatial domain and hence implicitly on \(n\). In particular, it  describes a sequence of spatially varying mean surfaces that becomes asymptotically representable by a fixed function \(M\) on the rescaled domain.
For a detailed formulation of the assumption (which also describes the condition on the temporal smoothness precisely) we refer to  \cref{ass:local-mean} in the following section.

The stationary and local settings share the difficulties caused by irregular
observations over time and dependence across sites. Estimation of a spatially
varying mean additionally requires smoothing in the spatial direction. We consider
two approaches for this purpose. The (spatial) local-kernel estimator introduced in \cref{sec:method-local} uses observations from
sites in a neighborhood of a specified spatial location, with a bandwidth $a_n$
controlling the degree of localization. The spatial-spline estimator discussed in \cref{sec:method-spatial-spline} instead represents
the entire mean surface by a tensor-product spline, with the spatial resolution $J_s$
controlling its flexibility. Accordingly, the convergence rates derived in the following Section~\ref{sec:theory} reflect the different spatial resolutions of the two procedures: localization through the bandwidth $a_n$
for the kernel estimator and through  $J_s$ for the
spatial-spline estimator.


\subsubsection{Local-kernel estimator}\label{sec:method-local}
We write $U_i=U\left(S_i\right)=\left(S_{i1}/L_{1,n},\ldots,S_{id}/L_{d,n}\right)^\top \in [0,1]^d $ and fix an interior point $u_0\in (0,1)^d$.  Let $K:\mathbb R^d\to\left[0,\infty\right)$ be a bounded Lipschitz kernel supported on $\left[-1,1\right]^d$, let $a_n\downarrow0$ denote a bandwidth, and define the weights $W_i\left(u_0\right):=K\left\{\left(U_i-u_0\right)/a_n\right\}$.  Weighting sites within an $a_n$-neighborhood of $u_0$ gives  the kernel estimator
\begin{align}
    \label{hol3} 
    \hat M_{J_m}\left(u_0,t\right)=\bB_m\left(t\right)^\top\hat\bbeta_{m,J_m}\left(u_0\right),
\end{align}
 where
\begin{equation} \label{hol4}
\hat\bbeta_{m,J_m}\left(u_0\right)=\argmin_{\bbeta\in\mathbb R^{K_m}}
\sum_{i=1}^{m_n}W_i\left(u_0\right)\sum_{j=1}^{N_i}
\left\{Y_{ij}-\bB_m\left(T_{ij}\right)^\top\bbeta\right\}^2.
\end{equation}
Let $M_{1,n}\left(u_0\right):=\sum_{i=1}^{m_n}W_i\left(u_0\right)N_i$.  Then $\hat\bbeta_{m,J_m}\left(u_0\right)=\hat Q_{m,J_m}\left(u_0\right)^{-1}\hat r_{m,J_m}\left(u_0\right)$, where
\begin{align*}
\hat Q_{m,J_m}\left(u_0\right)&:=\frac{1}{M_{1,n}\left(u_0\right)}
\sum_{i=1}^{m_n}W_i\left(u_0\right)\sum_{j=1}^{N_i}\bB_m\left(T_{ij}\right)\bB_m\left(T_{ij}\right)^\top,\\
\hat r_{m,J_m}\left(u_0\right)&:=\frac{1}{M_{1,n}\left(u_0\right)}
\sum_{i=1}^{m_n}W_i\left(u_0\right)\sum_{j=1}^{N_i}\bB_m\left(T_{ij}\right)Y_{ij}.
\end{align*}
Thus, the local estimator has the same temporal spline structure as the stationary
estimator (\ref{eq:mean-est}), but only observations from sites receiving positive kernel weight contribute
to the fit. The bandwidth $a_n$ therefore determines both the spatial smoothing bias
and the amount of spatial information available for estimating $M(u_0,\cdot)$.

\subsubsection{Tensor-product spatial-spline estimator}\label{sec:method-spatial-spline}
When the full spatial surface is of interest, a global tensor-product spline borrows information smoothly across all sites and avoids fitting a separate local model at each location.  Let $\bA_s\left(u\right)$ be a normalized tensor-product B-spline basis constructed from quasi-uniform partitions in the $d$ spatial coordinates, with $J_s = J_{s_n}$ subintervals in each coordinate and fixed order $p_s$.  Its dimension is $K_s\asymp J_s^d$.  With $\bZ_{J_s,J_m}\left(u,t\right):=\bA_s\left(u\right)\otimes\bB_m\left(t\right)\in\mathbb R^{K_sK_m}$, define
\begin{align}
    \label{hol6}
\hat M^{\mathrm{sp}}_{J_s,J_m}\left(u,t\right)=\bZ_{J_s,J_m}\left(u,t\right)^\top\hat\btheta_{M,J_s,J_m},
\end{align}
where
\begin{equation*}
\hat\btheta_{M,J_s,J_m}:=\argmin_{\btheta\in\mathbb R^{K_sK_m}}
\sum_{i=1}^{m_n}\sum_{j=1}^{N_i}
\left\{Y_{ij}-\bZ_{J_s,J_m}\left(U_i,T_{ij}\right)^\top\btheta\right\}^2.
\end{equation*}
The empirical Gram matrix and empirical moment vector are given by  
\begin{align*}
\hat Q_{M,J_s,J_m}&:=\frac1{M_{1,n}}\sum_{i=1}^{m_n}\sum_{j=1}^{N_i}
\bZ_{J_s,J_m}\left(U_i,T_{ij}\right)\bZ_{J_s,J_m}\left(U_i,T_{ij}\right)^\top,
\\
\hat r_{M,J_s,J_m}&:=\frac1{M_{1,n}}\sum_{i=1}^{m_n}\sum_{j=1}^{N_i}
\bZ_{J_s,J_m}\left(U_i,T_{ij}\right)Y_{ij},
\end{align*}
respectively, 
and the corresponding population quantities are given by
\begin{align*}
Q_{M,J_s,J_m}&:=\frac1{M_{1,n}}\sum_{i=1}^{m_n}N_i
\int_0^1\bZ_{J_s,J_m}\left(U_i,x\right)\bZ_{J_s,J_m}\left(U_i,x\right)^\top f_T\left(x\right)\,dx,\\
r_{M,J_s,J_m}&:=\frac1{M_{1,n}}\sum_{i=1}^{m_n}N_i
\int_0^1\bZ_{J_s,J_m}\left(U_i,x\right)m\left(S_i,x\right)f_T\left(x\right)\,dx,
\end{align*}
with
\begin{equation*}
\btheta_{M,J_s,J_m}^{\circ}:=Q_{M,J_s,J_m}^{-1}r_{M,J_s,J_m},
\qquad
M_{J_s,J_m}^{\circ}\left(u,t\right):=\bZ_{J_s,J_m}\left(u,t\right)^\top\btheta_{M,J_s,J_m}^{\circ}.
\end{equation*}
Note that the population projection $M^\circ_{J_s,J_m}$ is defined in terms of the
actual mean values $m(S_i,\cdot)$ at the sampled sites, rather than the approximating
surface $M(U_i,\cdot)$. This choice yields an exact centering of the sample normal
equations. The difference between $m(S_i,t)$ and $M(U_i,t)$ is treated separately
as the deterministic local-stationarity approximation error.

The defining population normal equation is given by 
\begin{equation*}
\sum_{i=1}^{m_n}N_i\int_0^1
\bZ_{J_s,J_m}\left(U_i,x\right)
\left\{m\left(S_i,x\right)-M_{J_s,J_m}^{\circ}\left(U_i,x\right)\right\}f_T\left(x\right)\,dx=0, 
\end{equation*}
and the centered site score for the joint fit is defined as 
\begin{equation*}
\omega_{i,J_s,J_m}:=\sum_{j=1}^{N_i}\bZ_{J_s,J_m}\left(U_i,T_{ij}\right)\left\{Y_{ij}-M_{J_s,J_m}^{\circ}\left(U_i,T_{ij}\right)\right\}-\E_i,
\end{equation*}
where $ \mathbb E_i
= \mathbb{E}\big [
\sum_{j=1}^{N_i}
\bZ_{J_s,J_m}(U_i,T_{ij}) \big \{
Y_{ij}-M^\circ_{J_s,J_m}(U_i,T_{ij})
\big \} \big ].
$
The population normal equation and \Cref{ass:time} give $\sum_i\E_i=0$, matching the aggregate centered score to the sample normal equations.

\section{Theoretical guarantees}\label{sec:theory}

In this section we demonstrate that the methodology developed in
Section~\ref{sec:methodology} provides consistent estimators in the different
scenarios. We are particularly interested in uniform convergence rates and the effects of sparse and dense sampling in the spatial and temporal
domain. For this purpose, we first introduce several conditions on the mean
functions and the sampling framework that are required to derive uniform
convergence rates for all three estimators introduced in Section 2. Assumptions specific to the
estimation of a stationary mean or to a locally stationary setting are stated
in the corresponding subsections below.

The rates presented in this section have a common structure. In all three
estimation problems, the stochastic error separates into an observation-level
component, governed by the nominal number of measurements, and a curve-level
component, governed by the dependence-adjusted effective sample size
$m_{\mathrm{eff},n}$ defined in~\eqref{meff} below. The assumptions introduced next
provide the dependence, tail, and smoothness conditions needed to establish
this decomposition.

For an integer $q\ge0$ and $\nu\in\left(0,1\right]$, let $C^{q,\nu}\left[0,1\right]$ denote the H\"older class with smoothness index $q+\nu$ and norm
\begin{align*}
\|f\|_{C^{q,\nu}}=\sum_{\ell=0}^q\|f^{\left(\ell\right)}\|_\infty+\sup_{x\ne y}{ |f^{\left(q\right)}\left(x\right)-f^{\left(q\right)}\left(y\right)| \over |x-y|^\nu } .
\end{align*}
For our theoretical analysis, we  combine a deterministic infill-increasing design with polynomial beta-mixing.  The design is a fixed-site analogue of the increasing-domain and infill regimes in spatial statistics considered, for example, in \citet{Lahiri2003} or \citet{JenishPrucha2009}; the mixing condition supports the covariance bounds \citep{Rio1993,Bradley2005}  used in the proofs of our main results.

For a Borel set $A\subset\mathbb R^d$, let $\cF_X\left(A\right)$ denote the $\sigma$-field generated by $\left\{X\left(s,\cdot\right):s\in A\right\}$. {The local model may depend on $n$, with this index suppressed and all constants uniform in $n$. For $r\ge0$, we define the pairwise spatial beta-mixing coefficients by
\begin{align*}
\beta_X\left(r\right)&:=\sup_n\sup_{s\ne s'}\left\{
\beta\left\{\cF_X\left(\{s\}\right),\cF_X\left(\{s'\}\right)\right\}:
\|s-s'\|_2\ge r
\right\},
\end{align*}}
where, for sub-$\sigma$-fields $\mathcal A,\mathcal B$, the coefficient $\beta\left(\mathcal A,\mathcal B\right)$  is defined by 
\begin{align*}
\beta\left(\mathcal A,\mathcal B\right)
&:=\frac12\sup_{\{A_u\},\{B_v\}}
\sum_{u,v}\left|\Pr\left(A_u\cap B_v\right)-\Pr\left(A_u\right)\Pr\left(B_v\right)\right|
\end{align*}
and the supremum is taken  over all finite partitions $\{A_u\}$ and $\{B_v\}$ of the underlying sample space with $A_u\in\mathcal A$ and $B_v\in\mathcal B$. Our next assumption refers to the decay rate of the spatial beta-mixing coefficients.

\begin{assumption}[Polynomial beta-mixing]\label[assumption]{ass:beta}
There exist constants $C_\beta>0$ and $\alpha_\beta>2d$ such that $\beta_X\left(r\right)\le C_\beta\left(1+r\right)^{-\alpha_\beta}$ for $r\ge0$.
\end{assumption}
To measure the effect of spatial dependence  we introduce the coefficient 
\begin{equation}\label{eq:Theta}
\Theta_n:=1+\sup_{1\le i\le m_n}\sum_{k\ne i}\beta_X\left(\|S_i-S_k\|_2\right)^{1/2},
\end{equation}
and call 
\begin{equation}\label{meff}
m_{\mathrm{eff},n}:=m_n/\Theta_n
\end{equation}
the {\it (dependence-adjusted) effective sample size}.
Spatial dependence reduces the amount of independent information contained in the
$m_n$ sampled curves. The factor $\Theta_n$ quantifies this redundancy, so that
$m_{\mathrm{eff},n}=m_n/\Theta_n$
provides a dependence-adjusted measure of the effective number of curves for the purpose of our convergence-rate analysis; see \SuppLemmaCref{lem:cov-sum} for more details. Thus, adding nearby sites may
contribute much less information than adding the same number of well-separated
sites. By contrast, independently sampled observation times and measurement errors
still average over the actual observations, so this part of the estimation error
benefits from the nominal sampling scale. This distinction leads to the two
information scales appearing in the convergence rates presented below. Under \Cref{ass:design,ass:beta}, \SuppLemmaCref{lem:Theta-order} gives $m_{\mathrm{eff},n}\gtrsim V_n$.  In the case of a sparse space regime, we have $m_n\asymp m_{\mathrm{eff},n} \asymp V_n$, whereas in the dense-space infill regime, we have $m_{\mathrm{eff},n}\asymp V_n$ under \Cref{ass:full-dimensional} and a short-range nonindependence condition (see  \SuppPropositionCref{prop:Theta-lower} for details)

The next assumptions describe irregular time sampling, tails, and smoothness, paralleling the independent-curve framework of \citet{CaiYuan2011} and \citet{ZhangWang2016}.

\begin{assumption}[Uniform tails]\label[assumption]{ass:tails}
There exists a constant $K_0>0$ such that
\begin{align*}
\sup_{s\in\R^d,t\in\left[0,1\right]}\|X\left(s,t\right)\|_{\psi_2}&\le K_0,&
\sup_{i,j}\|\veps_{ij}\|_{\psi_2}&\le K_0,&
\sup_{s,t}|\sigma\left(s,t\right)|&\le K_0.
\end{align*}
Moreover, we assume that the finite-dimensional distributions of the centered field are uniformly sub-Gaussian relative to their variance: for every $r\ge1$, every deterministic collection $(s_\ell,t_\ell)_{\ell=1}^r$, and every vector $(a_1,\ldots,a_r)^\top\in\R^r$,
\begin{equation}\label{eq:joint-subGaussian-field}
\left\|\sum_{\ell=1}^r a_\ell X(s_\ell,t_\ell)\right\|_{\psi_2}
\le K_0
\left[\operatorname{Var}\left\{\sum_{\ell=1}^r a_\ell X(s_\ell,t_\ell)\right\}\right]^{1/2}.
\end{equation}

\end{assumption}

\begin{assumption}[Balanced within-site counts]\label[assumption]{ass:balanced}
There exist constants $0<c_N<C_N<\infty$ such that $c_N\bar N_n\le N_i\le C_N\bar N_n$ for every $1\le i\le m_n$, where $\bar N_n=M_{1,n}/m_n$.
\end{assumption}

\begin{assumption}[Temporal mean smoothness]\label[assumption]{ass:mean-smooth}
In the stationary setting, the mean function $m$ is an element of $C^{q_m,\nu_m}\left[0,1\right]$; in the locally stationary setting, the function $t\mapsto M\left(u,t\right)$ is an element of $C^{q_m,\nu_m}\left[0,1\right]$ for all   $u\in\left[0,1\right]^d$.  The spline order satisfies $p\ge q_m^\ast+1$, where $q_m^\ast=q_m+\nu_m>0$.
\end{assumption}

\Cref{ass:time,ass:tails,ass:balanced,ass:mean-smooth} collect familiar regularity conditions for irregularly sampled functional data.  Random exogenous observation times with a density bounded above and away from zero, comparable within-curve sample sizes, and smooth mean functions have close counterparts in sparse-to-dense functional mean estimation \citep{YaoMullerWang2005,CaiYuan2011,ZhangWang2016}.  The covariance-relative finite-dimensional condition in \eqref{eq:joint-subGaussian-field} is stronger than the preceding pointwise $\psi_2$ bound because it controls every deterministic linear combination.  It holds automatically for Gaussian fields and parallels covariance-scaled sub-Gaussian conditions for Hilbert-space-valued functional observations \citep{GuoLiQiaoWang2025}. The balance condition is used to express the phase-transition rates through $\bar N_n$ and to simplify the variance terms for the local-kernel and spatial-spline estimators.

\subsection{Stationary mean}\label{sec:theory-stationary}
We first establish uniform convergence for the pooled stationary-mean estimator  \eqref{eq:mean-est} in the situation, when the mean function does not depend on the spatial locations. For this purpose we require a further assumption.

\begin{assumption}[Mean leverage condition]\label[assumption]{ass:mean-leverage}
The site counts satisfy $\max_{1\le i\le m_n}N_i=O\left(M_{1,n}/m_n\right)$.
\end{assumption}

\Cref{ass:mean-leverage} is equivalent to $\max_{i=1}^{m_n} N_i/\bar N_n=O\left(1\right)$ and  prevents individual sites from contributing a disproportionately large fraction of the total observations. Unlike Assumption \ref{ass:balanced}, it does not impose a uniform lower bound on the within-site sample sizes. The condition is not essential for obtaining a convergence rate, but it simplifies the representation of the rates.
For a family of $p_n$ score coordinates, write $b_n=(e+p_n+m_{\mathrm{eff},n})^{2/\alpha_\beta}$.

\begin{theorem}\label[theorem]{thm:mean-rate}
Suppose that $m(s,t)\equiv m(t)$ and that Assumptions \ref{ass:time}, \ref{ass:design}, \ref{ass:beta},  \ref{ass:tails}, \ref{ass:mean-smooth} and \ref{ass:mean-leverage} are satisfied.  Take $p_n=K_m$ in the definition of $b_n$.  Let $J_m\to\infty$ and suppose $J_m[\{\log K_m/(M_{1,n}J_m)\}^{1/2}+\log K_m/M_{1,n}]\to0$.  Then estimator (\ref{eq:mean-est}) satisfies that
\begin{align}
\sup_{t\in\left[0,1\right]}|\hat m_{J_m}\left(t\right)-m\left(t\right)|
&=
O_p\left(
J_m^{-q_m^\ast}+
\sqrt{\log K_m\left[
\frac{J_m}{M_{1,n}}+
\frac{M_{2,n}+(\Theta_n-1)(M_{1,n}+M_{2,n})}{M_{1,n}^2}
\right]}
\right.
\nonumber\\
&\left.\hspace{7em}+
\frac{J_m\delta_n^{-d} b_n^d\log K_m}{m_n}
\right).
\label{eq:mean-main-rate}
\end{align}
\end{theorem}

The first term on the right-hand side of \eqref{eq:mean-main-rate} is the spline-approximation error, while the second term corresponds to the stochastic part. Under the balanced-count condition this term can be further simplified.

\begin{corollary}\label[corollary]{cor:mean-phase}
 Suppose that the conditions of \Cref{thm:mean-rate} and \Cref{ass:balanced} are satisfied. If, additionally,  
$\delta_n^{-d}b_n^dJ_m\{\log K_m/[m_n(\Theta_n+J_m/\bar N_n)]\}^{1/2}\to0,
$
then (\ref{eq:mean-main-rate}) turns into
\begin{equation}\label{eq:mean-phase-rate}
\sup_{t\in\left[0,1\right]}|\hat m_{J_m}\left(t\right)-m\left(t\right)|
=
O_p\left(
J_m^{-q_m^\ast}+
\sqrt{\log K_m\left(
\frac{J_m}{m_n\bar N_n}+
\frac{1}{m_{\mathrm{eff},n}}
\right)}
\right).
\end{equation}

\end{corollary}
The stochastic part in \eqref{eq:mean-phase-rate}  consists of two components with different information scales. The term \(J_m/(m_n\bar N_n)\) corresponds to observation-level variation, whereas \(m_{\mathrm{eff},n}^{-1}\) represents the curve-level component. Independently sampled observation times and measurement errors contribute to the former, while variation of the latent random field and spatial dependence contribute to the latter. Increasing the number of observations per curve reduces the observation-level component until the curve-level component becomes dominant.
\\
Ignoring logarithmic factors, the balance between these two components yields the temporal phase-transition boundary
\begin{equation}\label{eq:mean-phase-boundary}
\bar N_n\asymp \frac{m_{\mathrm{eff},n}^{1/(2q_m^\ast)}}{\Theta_n}.
\end{equation}
Below this boundary, balancing the temporal spline bias with the observation-level component gives \(J_m\asymp(m_n\bar N_n)^{1/(2q_m^*+1)}\) and the rate
\[
\sup_{t\in [0,1]}|\hat m_{J_m}(t)-m(t)|
=O_p\!\left\{(m_n\bar N_n)^{-q_m^\ast/(2q_m^\ast+1)}\right\},
\]
up to logarithmic factors. At or above the boundary, choosing \(J_m\asymp m_{\mathrm{eff},n}^{1/(2q_m^*)}\) yields the rate \(m_{\mathrm{eff},n}^{-1/2}\), again up to logarithmic factors. We refer to temporal sampling below the boundary as sparse and to sampling at or above it as dense.
\\
An important consequence of \eqref{eq:mean-phase-boundary} is that the temporal phase-transition boundary depends not only on the effective number of spatial curves, but also on the spatial dependence inflation factor \(\Theta_n\). In the sparse-space regime, \(\Theta_n\asymp1\) and \(m_{\mathrm{eff},n}\asymp m_n\), so that\eqref{eq:mean-phase-boundary} reduces to the familiar boundary \(\bar N_n\asymp m_n^{1/(2q_m^*)}\). By contrast, under full-dimensional spatial infill with persistent local dependence, we have \(m_{\mathrm{eff},n}\asymp V_n\) and \(\Theta_n\asymp\delta_n^{-d}\), and the boundary becomes $
\bar N_n\asymp V_n^{1/(2q_m^*)}\delta_n^d.
$
Thus, increasing spatial sampling density shifts the temporal phase transition towards fewer observations per curve. Depending on the relative rates of domain expansion and spatial infill, this boundary may even converge to zero. Since \(\bar N_n\ge1\), all feasible temporal sampling densities then lie asymptotically in the dense-time regime of the leading convergence-rate bound. In this situation, additional spatial observations may still reduce the observation-level variation, but the leading curve-level error remains of order \(V_n^{-1/2}\). This illustrates how spatial infill can alter the temporal phase transition rather than merely reducing the effective number of independent curves.

\begin{remark}\label[remark]{rem:mean-special-cases}
If the data are independent, we have $\Theta_n=1$ and \Cref{cor:mean-phase} recovers the sparse-to-dense  structure in the work  of \citet{CaiYuan2011} and \citet{ZhangWang2016}. For $q_m^\ast=2$, the phase transition boundary is $\bar N_n\asymp m_n^{1/4}$ after logarithmic factors are suppressed.  The same algebra applies to a stationary functional time series observed on the fixed-spacing lattice.  Indeed, \Cref{ass:beta} with $d=1$ and $\alpha_\beta>2$ implies $\sum_{h\ge1}\beta_X\left(\Delta h\right)^{1/2}<\infty$, hence $\Theta_n=O\left(1\right)$ and $m_{\mathrm{eff},n}\asymp m_n$.  The resulting rate is of  order $(m_n\bar N_n )^{-q_m^\ast/\left(2q_m^\ast+1\right)}$ below the phase transition boundary and $m_n^{-1/2}$ above it.  The latter agrees with the root-$m_n$ mean behavior in \citet{HorvathKokoszkaReeder2013}, \citet{ChenSong2015}, and \citet{LiYang2023}, with serial dependence entering their long-run covariance.  This specialization complements sup-norm theory for independent curves \citep{BergerHermannHolzmann2025} by making the spatial covariance inflation explicit. The rates established here are upper bounds for the proposed estimators; establishing
matching minimax lower bounds and simultaneous inference is left for future work.
\end{remark}

\subsection{Locally stationary mean}

We next derive uniform convergence rates for the (spatial) kernel estimator in \eqref{hol3}, with coefficients obtained from \eqref{hol4}, and for the spatial spline estimator \eqref{hol6}.  Both methods allow the mean to vary smoothly over space.  We first state a condition used to describe this setting.

\begin{assumption}[Locally stationary mean]\label[assumption]{ass:local-mean}
There exists a function $M:\left[0,1\right]^d\times\left[0,1\right]\to\R$ such that: (i) for some $\gamma_u\in\left(0,1\right]$, $\sup_{t\in\left[0,1\right]}|M\left(u,t\right)-M\left(v,t\right)|\le C\|u-v\|_2^{\gamma_u}$ for all $u,v\in\left[0,1\right]^d$; (ii) the temporal smoothness requirement in \Cref{ass:mean-smooth} holds uniformly in $u$; (iii) the actual mean satisfies \eqref{hol5};

and (iv) the family $\left\{m\left(s,\cdot\right):s\in\cS_n\right\}$ is uniformly bounded in $C^{q_m,\nu_m}\left[0,1\right]$.
\end{assumption}

\subsubsection{Local-kernel estimator}\label{sec:theory-local}

For the theoretical analysis of the kernel estimator \eqref{hol3}, we make the following  assumption regarding the weights in the loss function \eqref{hol4}.

\begin{assumption} 
\label[assumption]{ass:local-mass}
At the fixed point $u_0 \in (0,1)^d$, we have  $\sum_{i=1}^{m_n}W_i\left(u_0\right)\asymp m_na_n^d$ and $\sum_{i=1}^{m_n}W_i\left(u_0\right)^2\asymp m_na_n^d$.

\end{assumption}
Recall  the nation $M_{1,n}\left(u_0\right)=\sum_{i=1}^{m_n}W_i\left(u_0\right)N_i$ and define $
m_{u_0,\mathrm{eff},n}:=\left\{M_{1,n}\left(u_0\right)\right\}^2/\\\left\{\Theta_n\sum_{i=1}^{m_n}W_i\left(u_0\right)^2N_i^2\right\}.
$ Note that under \Cref{ass:balanced,ass:local-mass} we have   $m_{u_0,\mathrm{eff},n}\asymp m_{\mathrm{eff},n}a_n^d$. Thus, spatial localization reduces both the nominal and the effective
curve-level sample sizes by the window-volume factor $a_n^d$.

The deterministic error of the local kernel estimator \eqref{hol3} consists of the temporal spline bias, the spatial smoothing bias $a_n^{\gamma_u}$, and the $L_n^{-1}$ approximation error from local stationarity.

\begin{theorem}\label[theorem]{thm:local-mean-rate}
Let Assumptions \ref{ass:time}, \ref{ass:design}, \ref{ass:beta}-\ref{ass:balanced}, \ref{ass:local-mean}, \ref{ass:local-mass} be satisfied and let  $J_m\to\infty$, $J_m[\{\log K_m/(M_{1,n}(u_0)J_m)\}^{1/2}+\log K_m/M_{1,n}(u_0)]\to0$. Take $p_n=K_m$ in the definition of $b_n$.  Then, for every fixed interior point $u_0$,
\begin{align}
&\sup_{t\in\left[0,1\right]}|\hat M_{J_m}\left(u_0,t\right)-M\left(u_0,t\right)|\nonumber
\\&=
O_p\left(
J_m^{-q_m^\ast}+a_n^{\gamma_u}+L_n^{-1}
+
\sqrt{\log K_m\left(
\frac{J_m}{m_na_n^d\bar N_n}+
\frac{1}{m_{\mathrm{eff},n}a_n^d}
\right)}
+
\frac{J_m\delta_n^{-d} b_n^d\log K_m}{m_n a_n^d}
\right).
\label{eq:local-mean-rate}
\end{align}

\end{theorem}

Within the spatial window, the observation-level information is of order $m_na_n^d\bar N_n$, whereas the curve-level effective size is $m_{\mathrm{eff},n}a_n^d\asymp m_{u_0,\mathrm{eff},n}$.  The balance between these two information scales determines the local time-sampling transition.

\begin{corollary}[Local phase transition]\label[corollary]{cor:local-mean-phase}
If  the assumptions of \Cref{thm:local-mean-rate} and additionally the condition   $\delta_n^{-d}b_n^dJ_m\left\{\log K_m/\left[a_n^{d}m_n\left(\Theta_n+J_m/\bar N_n\right)\right]\right\}^{1/2}\to0$ are satisfied, the rate in \eqref{eq:local-mean-rate} simplifies to
\begin{equation}\label{eq:local-mean-rate-simple}
\sup_{t\in\left[0,1\right]}|\hat M_{J_m}\left(u_0,t\right)-M\left(u_0,t\right)|
=
O_p\left(
J_m^{-q_m^\ast}+a_n^{\gamma_u}+L_n^{-1}
+
\sqrt{\log K_m\left(
\frac{J_m}{m_na_n^d\bar N_n}+
\frac{1}{m_{\mathrm{eff},n}a_n^d}
\right)}
\right).
\end{equation}
\end{corollary}
Ignoring logarithmic factors and $L_n^{-1}$, the time-direction phase transition boundary is  
\begin{equation}\label{eq:local-mean-phase-boundary}
\bar N_n\asymp
\frac{\left(m_{\mathrm{eff},n}a_n^d\right)^{1/(2q_m^\ast)}}{\Theta_n}.
\end{equation}
Below this phase transition boundary, balancing the temporal  spline bias with the observation-level component gives $J_m\asymp(m_na_n^d\bar N_n)^{1/(2q_m^\ast+1)}$ and the time-smoothing contribution $(m_na_n^d\bar N_n)^{-q_m^\ast/(2q_m^\ast+1)}$ in (\ref{eq:local-mean-rate-simple}), up to logarithmic factors. At or above the phase transition boundary, the curve-level component determines the order of the stochastic error. Taking $J_m\asymp(m_{\mathrm{eff},n}a_n^d)^{1/(2q_m^\ast)}$ gives the order $(m_{\mathrm{eff},n}a_n^d)^{-1/2}$. In both regimes, the complete error bound in \eqref{eq:local-mean-rate-simple} also contains the spatial terms $a_n^{\gamma_u}$ and $L_n^{-1}$.

The preceding results concern estimation at a fixed spatial location. We next strengthen Theorem~\ref{thm:local-mean-rate} to obtain uniformity over interior spatial points.

\begin{theorem}[Uniformity over interior spatial points]\label[theorem]{thm:local-mean-uniform}
Let $\mathcal U_\circ=\left[c_0,1-c_0\right]^d$ for some fixed constant $c_0\in\left(0,1/2\right)$, and let \Cref{ass:time}, \ref{ass:design}, \ref{ass:beta}-\ref{ass:balanced} and \ref{ass:local-mean} be satisfied. Further, assume that the kernel $K$ is bounded, Lipschitz continuous, and compactly supported and that, for $r=1,2$,
\[
c m_na_n^d\le \inf_{u\in\mathcal U_\circ}\sum_{i=1}^{m_n}W_i(u)^r
\le \sup_{u\in\mathcal U_\circ}\sum_{i=1}^{m_n}W_i(u)^r\le C m_na_n^d.
\]
Let $\mathcal U_n\subset\mathcal U_\circ$ be a regular $(a_n^{2d+2}/J_m )$-net, {take $p_n=K_m|\mathcal U_n|$ in the definition of $b_n$}, and define $\ell_n:=\log K_m+\log\left(a_n^{-1}\right)$.  
 If  $J_m\to\infty$, $J_m\big [\big \{\ell_n/ (m_na_n^d\bar N_nJ_m)\big \}^{1/2}+\ell_n/(m_na_n^d\bar N_n )\big ]\to0$,
we have 
\begin{align}
&\sup_{u\in\mathcal U_\circ}\sup_{t\in\left[0,1\right]}|\hat M_{J_m}\left(u,t\right)-M\left(u,t\right)|
\nonumber\\
&=
O_p\left(
J_m^{-q_m^\ast}+a_n^{\gamma_u}+L_n^{-1}
+
\sqrt{\ell_n\left(
\frac{J_m}{m_na_n^d\bar N_n}+
\frac{1}{m_{\mathrm{eff},n}a_n^d}
\right)}
+
\frac{J_m\delta_n^{-d} b_n^d\ell_n}{m_n a_n^d}
\right)
\label{eq:local-mean-uniform}
\end{align}

\end{theorem}

\subsubsection{Spatial-spline estimator}\label{sec:theory-spatial-spline}
The tensor-product estimator \eqref{hol6} recovers the complete surface $M$ in one global fit.  In place of a local window, it allocates the effective sample across $J_s^d$ spatial degrees of freedom. For its theoretical analysis we require the following assumptions.

\begin{assumption}[Spatial smoothness of the local mean]
\label[assumption]{ass:spatial-spline-mean}
Let $q_u\in\mathbb N_0$, $\nu_u\in(0,1]$, and write
$q_u^*=q_u+\nu_u$. For every multi-index
$\alpha\in\mathbb N_0^d$ with $|\alpha|\le q_u$ and every $0\le \ell\le q_m$, the derivative
$\partial_u^\alpha \partial_t^\ell M(u,t)$ 
exists and is uniformly bounded on $[0,1]^d\times[0,1]$. Moreover,
there exists a constant $C>0$ such that
\[
\sup_{t\in[0,1]}
\left|
\partial_u^\alpha \partial_t^\ell M(u,t)
-
\partial_u^\alpha \partial_t^\ell M(v,t)
\right|
\le
C\|u-v\|_2^{\nu_u}
\]
for all $u,v\in[0,1]^d$, $|\alpha|=q_u$, and
$0\le \ell\le q_m$, and
\[
\sup_{u\in[0,1]^d}
\left|
\partial_u^\alpha \partial_t^{q_m} M(u,t)
-
\partial_u^\alpha \partial_t^{q_m} M(u,s)
\right|
\le
C|t-s|^{\nu_m}
\]
for all $s,t\in[0,1]$ and $|\alpha|\le q_u$.
The spline orders satisfy $
p_s\ge q_u^*+1, ~
p\ge q_m^*+1.
$
\end{assumption}

\begin{assumption}[Spatial design stability]\label[assumption]{ass:spatial-spline-design}
The aspect ratios $L_{r,n}/L_n$ are bounded above and bounded away from zero. For the rescaled sites $U_i=U(S_i)\in[0,1]^d$, define  
\[
h_{U,n}:=\sup_{u\in[0,1]^d}\min_{1\le i\le m_n}\|u-U_i\|_2,
\qquad
q_{U,n}:=\frac12\min_{i\ne k}\|U_i-U_k\|_2,
\]
as the fill distance and separation radius. Assume that $h_{U,n}/q_{U,n}\le C$ and $h_{U,n}\to0$.  Thus the rescaled sites are quasi-uniform and, in particular, $h_{U,n}\asymp q_{U,n}\asymp m_n^{-1/d}$.  The spatial spline resolution satisfies $J_sh_{U,n}\to0$.
\end{assumption}

\begin{theorem}[Spatial-spline local mean: global supnorm rate]\label[theorem]{thm:spatial-spline-mean-rate}
Let Assumptions \ref{ass:time}, \ref{ass:design}, \ref{ass:beta}- \ref{ass:balanced}, \ref{ass:local-mean}, \ref{ass:spatial-spline-mean}, \ref{ass:spatial-spline-design} be satisfied.  {Take $p_n=K_sK_m$ in the definition of $b_n$.}  If
$
J_s^dJ_m\left\{
\sqrt{J_s^{-d}\log\left(K_sK_m\right)/(M_{1,n}J_m)}
+ \log\left(K_sK_m\right)/M_{1,n}
\right\}\to0,$ 
then
\begin{align}
&\sup_{u\in\left[0,1\right]^d}\sup_{t\in\left[0,1\right]}
|\hat M^{\mathrm{sp}}_{J_s,J_m}\left(u,t\right)-M\left(u,t\right)|
\label{eq:spatial-spline-mean-rate}\\
&\quad =
O_p\left(
J_s^{-q_u^\ast}+J_m^{-q_m^\ast}+L_n^{-1}
+
\sqrt{
J_s^d\log\left(K_sK_m\right)
\left(
\frac{J_m}{m_n\bar N_n}+
\frac{1}{m_{\mathrm{eff},n}}
\right)}
+
\frac{J_s^dJ_m\delta_n^{-d} b_n^d\log\left(K_sK_m\right)}{m_n}
\right).
\nonumber
\end{align}

\end{theorem}

Note that the approximation errors have the different order $J_s^{-q_u^\ast}$ and $J_m^{-q_m^\ast}$ in the space and time direction.  The observation-level component allocates $m_n\bar N_n$ measurements across $J_s^d$ spatial degrees of freedom, whereas the curve-level component allocates the dependence-adjusted size $m_{\mathrm{eff},n}$ across the same spatial resolution.

\begin{corollary}[Spatial phase transition]\label[corollary]{cor:spatial-spline-mean-phase}
Let the assumptions of \Cref{thm:spatial-spline-mean-rate} be  satisfied and assume  that   $\delta_n^{-d}b_n^dJ_m\big \{J_s^d\log (K_sK_m )/\big [m_n (\Theta_n+J_m/\bar N_n\  )\big]\big \}^{1/2}\to0$.  Then for fixed spatial resolution $J_s$ and with logarithmic factors suppressed, the time-direction phase transition boundary is
\[
\bar N_n\asymp
\frac{\left(m_{\mathrm{eff},n}/J_s^d\right)^{1/(2q_m^\ast)}}{\Theta_n}.
\]
Below this phase transition boundary, we have $J_m\asymp\left(m_n\bar N_n/J_s^d\right)^{1/(2q_m^\ast+1)}$ and
\[
R_{m,n}:=\sup_{u\in\left[0,1\right]^d}\sup_{t\in\left[0,1\right]}
|\hat M^{\mathrm{sp}}_{J_s,J_m}\left(u,t\right)-M\left(u,t\right)|
=O_p\left\{J_s^{-q_u^\ast}+L_n^{-1}+(m_n\bar N_n/J_s^d)^{-q_m^\ast/(2q_m^\ast+1)}\right\}.
\]
At or above it,  we have $J_m\asymp (m_{\mathrm{eff},n}/J_s^d )^{1/(2q_m^\ast)}$ and $R_{m,n}=O_p\big \{J_s^{-q_u^\ast}+L_n^{-1}+ (J_s^d/m_{\mathrm{eff},n} )^{1/2}\big \}$.  If $L_n^{-1}$ is of  lower order, then the joint minimization in the sparse-time regime gives
\[
R_{m,n}=O_p\left\{ (m_n\bar N_n )^{-q_u^\ast q_m^\ast/\left(2q_u^\ast q_m^\ast+dq_m^\ast+q_u^\ast\right)}\right\},
\]
whereas in the dense-time spatial regime the  rate is given by  $m_{\mathrm{eff},n}^{-q_u^\ast/\left(2q_u^\ast+d\right)}$.

\end{corollary}

The two local-mean procedures have the same leading stochastic structure when their spatial resolutions are matched.  Setting $a_n\asymp J_s^{-1}$ makes $m_{\mathrm{eff},n}a_n^d$ and $m_{\mathrm{eff},n}/J_s^d$ equivalent, and hence aligns their time-direction phase-transition boundaries up to logarithmic factors.  Their distinction lies in approximation: a local-constant kernel has spatial bias $a_n^{\gamma_u}$, whereas a tensor-product spline achieves $J_s^{-q_u^\ast}$ under the stronger anisotropic smoothness condition.

\section{Simulation study}\label{sec:simulation}

In this section, we investigate the finite sample properties of the proposed estimators.  
All reported Monte Carlo averages are based on $200$ replications.  We use the values  $m_n=400$, $900$, $1600$, $2500$, $3600$, $6400$, and $10000$ for the number of spatial locations.  We first consider the problem of estimating a stationary mean and then retain the same spatial and temporal sampling designs for the problem of estimating a locally stationary mean.

\subsection{Stationary mean}\label{sec:sim-stationary}

In dimension $d=2$, we generated the data from model (\ref{eq:model}) with  $m\left(s, t\right)=2\sin\left(\pi t\right)$, $\sigma(s,t) = 0.3$, $T_{ij}\stackrel{\mathrm{iid}}{\sim}\operatorname{Unif}\left[0,1\right]$, and  independent  standard normal distributed errors $\epsilon_{ij}$.  The time points and measurement errors are independent of the latent process.  At site $S_i$, we define the random field  by
$X\left(S_i,t\right)=\sum_{k=1}^4\xi_{ik}\phi_k\left(t\right)$, where
\begin{align*}
\left(\phi_1\left(t\right),\phi_2\left(t\right),\phi_3\left(t\right),\phi_4\left(t\right)\right)^\top
=\sqrt2\left(\sin\left(2\pi t\right),\cos\left(2\pi t\right),\sin\left(4\pi t\right),\cos\left(4\pi t\right)\right)^\top, 
\end{align*}
and, for each $k$, $\left(\xi_{1k},\ldots,\xi_{m_n,k}\right)^\top$ is a centered Gaussian vector with 
$\operatorname{Cov}\left(\xi_{ik},\xi_{\ell k}\right)=2^{-k}\rho_{r,k}\left(S_i-S_\ell\right)$ for  some function $\rho_{i,k}$. The four score vectors are mutually $i$ independent.   Writing $h=S_i-S_\ell$, the five choices for the function  $\rho_{i,k}$ are
\begin{align*}
\rho_{0,k}\left(h\right)&=\mathbf 1\left(h=0\right),  \tag{\text{D}0}\\
\rho_{1,k}\left(h\right)&=\exp\left(-\|h\|/\alpha_k\right),  \tag{\text{D}1}\\
\rho_{2,k}\left(h\right)&=\exp\left(-\|h\|^2/\alpha_k^2\right), \tag{\text{D}2}\\
\rho_{3,k}\left(h\right)&=\left(1+\sqrt3\|h\|/\alpha_k\right)\exp\left(-\sqrt3\|h\|/\alpha_k\right), \tag{\text{D}3}\\
\rho_{4,k}\left(h\right)&=\exp\!\left[-\left\{\left(h_1/a_{1k}\right)^2+\left(h_2/a_{2k}\right)^2\right\}^{1/2}\right]. \tag{\text{D}4}
\end{align*}
Here $\left(\alpha_1,\ldots,\alpha_4\right)^\top=\left(0.4,0.6,0.8,1.0\right)^\top$, while
$\left(a_{1k},a_{2k}\right)^\top$ equals $\left(0.3,0.8\right)^\top$, $\left(0.4,1.0\right)^\top$, $\left(0.5,1.2\right)^\top$, and $\left(0.6,1.4\right)^\top$ for $k=1,\ldots,4$.  Thus the scenario (D0) gives scores with independent components, and (D1) - (D4) use exponential, Gaussian, Mat\'ern-$3/2$, and anisotropic exponential covariance functions.

For each $m_n$, the sites form a $\sqrt{m_n}\times\sqrt{m_n}$ grid.  In the sparse space regime, the nearest-neighbor distance  of the spatial locations  is one and the spatial area is of order $m_n$.  In the dense space regime, the spatial locations are  given by $ \big (a/m_n^{1/4},b/m_n^{1/4} \big )^\top$, so the nearest-neighbor distance is $m_n^{-1/4}$ and the volume of the spatial domain is of order $m_n^{1/2}$.  We therefore have $m_{\mathrm{eff},n}=m_n$ under (D0) and under sparse-space scenarios (D1)--(D4), and $m_{\mathrm{eff},n}=m_n^{1/2}$ under dense-space scenarios (D1)--(D4).

For the sparse time regime, the numbers of time points at each locations are sampled uniformly from the integers between $2m_n^{1/5}$ and $4m_n^{1/5}$, while in the dense time regime are  sampled uniformly from the integers between $1.5m_n^{1/2}$ and $2.5m_n^{1/2}$. They are kept fixed for each simulated scenario, and the number of time points is the same for the sparse and the dense space regime.   The additional simulation experiment in \SuppCref{app:additional-simulations-meff} replaces $m_n$ by $m_{\mathrm{eff},n}$ in these two ranges.

The stationary estimator \eqref{eq:mean-est} uses cubic B-splines with dimension proportional to $ (m_{\mathrm{eff},n}\bar N_n)^{1/5}$ in the sparse time and to $m_{\mathrm{eff},n}^{1/4}$ in the dense time regime. The sup-norm of $ \hat m_{J_m} -m $ is calculated as a maximum over 101 equally spaced points.

\begin{figure}[!h]
\centering
\includegraphics[width=0.82\textwidth]{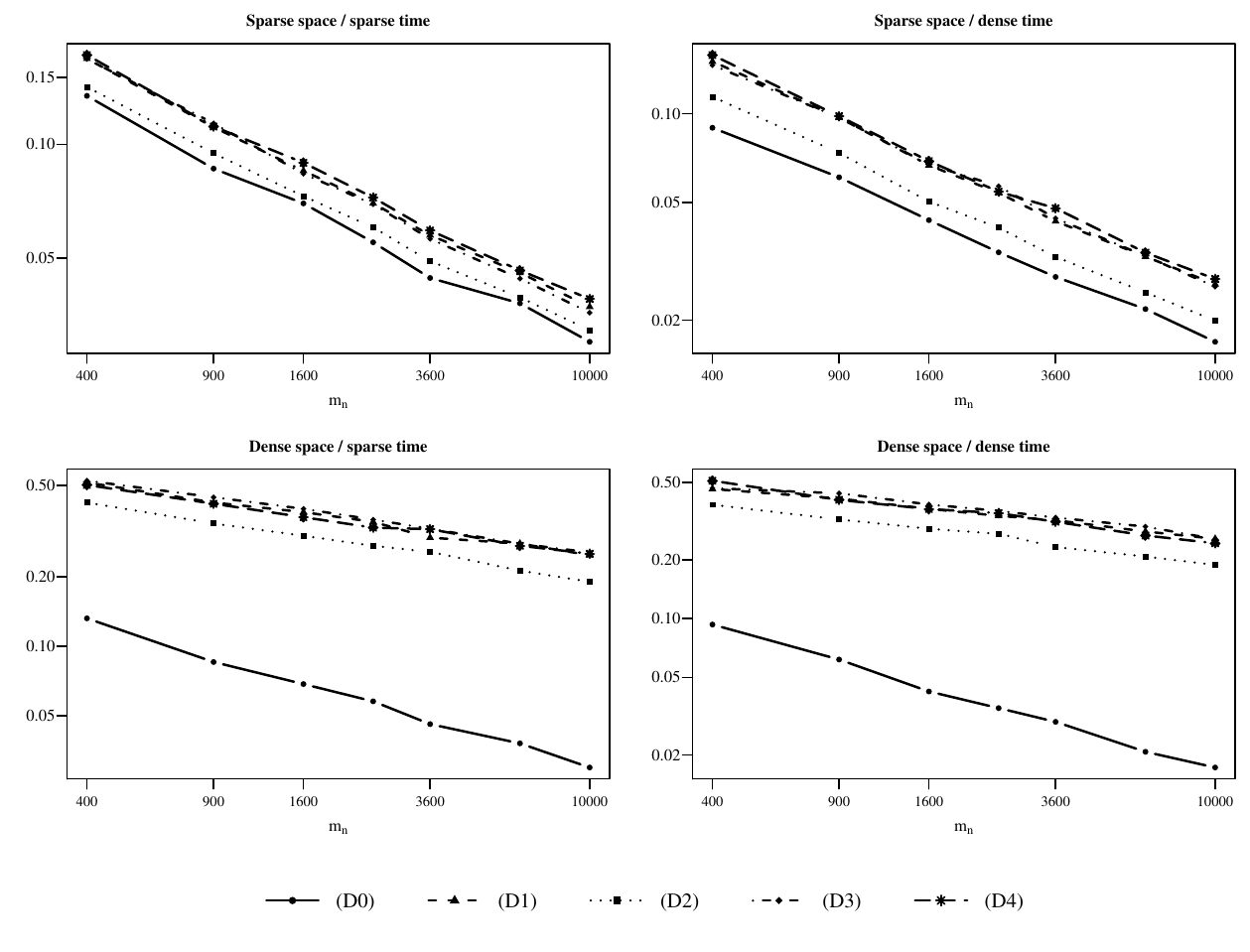}
\vspace{0pt}
\caption{\it Simulated sup-norm error for the  estimator \eqref{eq:mean-est} for different values of $m_n$ under scenario (D0)--(D4).}
\label{fig:stationary-mean-rates-(D0)-(D4)}
\end{figure}

In Figure \ref{fig:stationary-mean-rates-(D0)-(D4)} we display the simulated sup-norm error for the different cases. The detailed numbers  are shown in Table S.1 of supplementary materials. We observe  that the errors are decreasing with increasing  $m_n$ for all five dependence choices.  In scenario (D0), changing the distance between sites has little effect because the scores are independent.  Under (D1)--(D4), the errors remain close to the independent case in sparse-space regime, whereas the differences become larger and decrease more slowly in dense-space regime because nearby sites are more strongly correlated and $m_{\mathrm{eff},n}$ grows as $m_n^{1/2}$.  Dense time sampling decreases the error in the sparse space regime and under (D0).  In the dense space regime (lower panels in Figure \ref{fig:stationary-mean-rates-(D0)-(D4)}),  the sup-norm errors in models (D1)--(D4) are close and occasionally reverse because additional observations within a curve reduce the measurement error but not the spatially dependent part of the process.  These patterns agree with spatial functional results in which mean estimation depends on both spatial dependence and the distances between sites \citep{HoermannKokoszka2013spatial,ZhangLi2022}.

\subsection{Locally stationary mean}\label{sec:sim-local}

In this section, we investigate the performance of the estimator of a locally stationary mean using the same sites, observation times, measurement errors, component  and dependence functions  of scenario (D0)--(D4), and we set
$U_i=((a-1)/(\sqrt{m_n}-1),(b-1)/(\sqrt{m_n}-1))^\top$. The rescaled mean function is given by 
$M\left(u,t\right)=2\sin\left(\pi t\right)+0.75\left(u_1-1/2\right)\cos\left(2\pi t\right)$, and the variance of marginal scores $\xi_{ik}$ is changed from $2^{-k}$ to
$\lambda_k\left(U_i\right)=2^{-k}\left\{1+0.4\left(U_{i1}+U_{i2}-1\right)\right\}$, which gives 
$\left\{\lambda_k\left(U_i\right)\lambda_k\left(U_\ell\right)\right\}^{1/2}
\rho_{r,k}\left(S_i-S_\ell\right) $ for $
\operatorname{Cov}\left(\xi_{ik},\xi_{\ell k}\right)$.

For the local-kernel estimator \eqref{hol3}, we evaluate the sup-norm of 
$\hat M_{J_m}\left(u_0,\cdot \right)-M\left(u_0,\cdot \right)$ at point $u_0=\left(0.5,0.5\right)^\top$ over the same 101 equally spaced  time points.  At this location, $M\left(u_0,t\right)=m\left(t\right)$, so the stationary and local targets have temporal shape.  The bandwidth is proportional to $(m_{\mathrm{eff},n}\bar N_n)^{-2/9}$ in the sparse time and $m_{\mathrm{eff},n}^{-1/4}$ in the dense time regime.  The time dimension follows the stationary rule after replacing $m_{\mathrm{eff},n}$ by $m_{\mathrm{eff},n}a_n^2$.

For comparison over a spatial region, we evaluate the local-kernel estimator \eqref{hol3} and spatial-spline estimator \eqref{hol6} of the grid $\left\{0.2,0.35,0.5,0.65,0.8\right\}^2$ and 
$101$ equally spaced time points, and take the maximum error over all $25\times101$ points.  For the spatial spline, $K_s$ is $\max\left\{4,\left\lceil0.5m_{\mathrm{eff},n}^{1/4}\right\rceil\right\}^2$, and its time dimension follows the stationary rule with $m_{\mathrm{eff},n}$ replaced by $m_{\mathrm{eff},n}/K_s$.

Figure \ref{fig:local-mean-rates-(D0)-(D4)} shows the sup-norm error of the estimator (\ref{hol3}) and the detailed results are shown in Table S.2 in the supplementary materials. We observe the same effect of spatial dependence as in the stationary experiment.  Errors decrease rapidly in the sparse space regime but more slowly in dense space regime under (D1)--(D4).  Dense time sampling gives a clearer reduction of the error because the estimator uses only sites near the point $u_0$.

\begin{figure}[!h]
\centering
\includegraphics[width=0.82\textwidth]{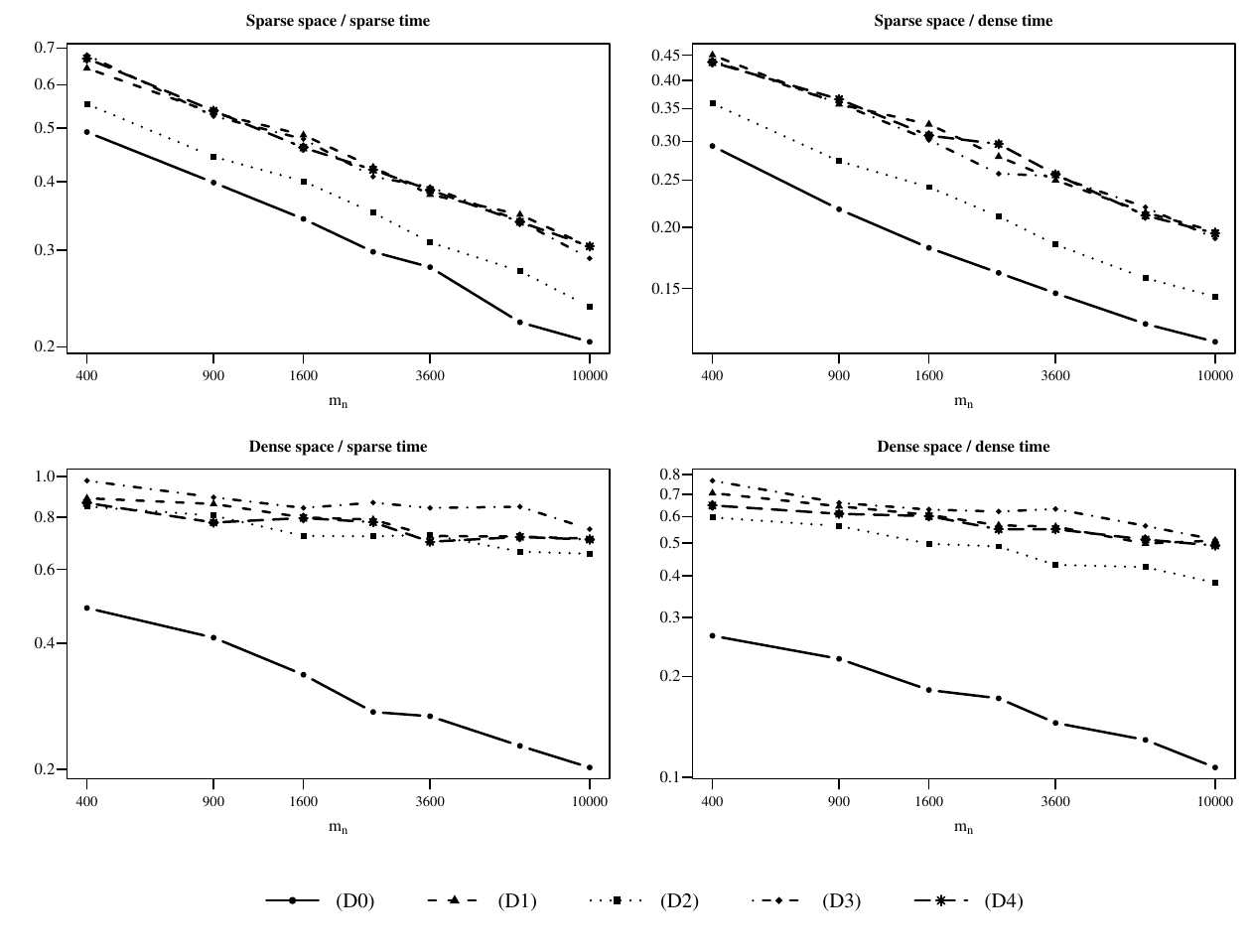}
\vspace{0pt}
\caption{\it Simulated  sup-norm errors for the local-kernel estimator \eqref{hol3} over all seven values of $m_n$ under scenarios (D0)--(D4).}
\label{fig:local-mean-rates-(D0)-(D4)}
\end{figure}

In \SuppCref{tab:mean-sim-local-methods}, we display the simulated sup-norm error of the estimators (\ref{hol3}) and (\ref{hol6}). We observe that the errors also decrease substantially in the sparse space regime.  The maximal error of the spatial spline estimator decreases faster than that of the local-kernel estimator because it fits the complete surface using all sites.  In the dense space regime, both methods improve more slowly.  In the sparse time regime, the spatial spline estimators becomes more accurate as $m_n$ increases, whereas under dense time regime the local-kernel estimator remains more accurate over the reported range.

Finally, we study the robustness of the estimators with respect to the score  distribution. More specifically, we repeated the experiment for scenario (D1) replacing either the Gaussian scores or the Gaussian measurement errors by standardized uniform, Laplace, or $t_5$ variables. 
Note that the Laplace and \(t_5\) distributions do not satisfy the sub-Gaussian assumptions of the theory. For the measurement-error experiment, each distribution was considered with $\sigma\left(s,t\right)=0.3$ and with $\sigma\left(s,t\right)=0.3\left\{1+0.5\sin\left(\pi t\right)\right\}$.  The results for the four score distributions are displayed in \SuppCref{tab:mean-sim-robustness-latent}  for the stationary  and the locally stationary mean function, where   $m_n=10000$. We do not observe substantial differences for the four score distributions.  Corresponding results for alternative distributions of  the measurement-error  are given in \SuppCref{tab:mean-sim-robustness-error-hom,tab:mean-sim-robustness-error-het} and show a similar picture. Therefore, the main differences in these experiments are due to spatial sampling and dependence rather than to the particular score or error distribution.

\section{Applications}\label{sec:applications}

We illustrate the application of the estimators in three public environmental data sets: North Atlantic sea-surface temperature data (NOAA OISST v2.1) representing spatially and temporally dense data, California ground-level
ozone  data  (EPA AQS ozone)  representing spatially sparse and temporally dense data;    California stream
nitrate  data  (WQP/USGS nitrate) representing representing spatially and temporally sparse data. The latter example is discussed in Section S.1.2 of the supplementary materials.  The spatial and temporal resolutions of the three examples are summarized in \Cref{tab:application-data-summary}.

\begin{table}[!ht]
\centering

\small
\setlength{\tabcolsep}{3.4pt}
\resizebox{\textwidth}{!}{%
\begin{tabular}{lllrc rr}
\toprule
Data set & Period & Spatial layout & $m_n$ & $N_i$: min/med/max & $M_{1,n}$ & Coverage (\%) \\
\midrule
NOAA OISST v2.1 & 2023 & Regular $0.5^\circ$ ocean grid & 2565 & 364/364/364 & 933660 & 99.7 \\
EPA AQS ozone & 2023 & Irregular monitors & 131 & 304/355/363 & 45925 & 96.0 \\
WQP/USGS nitrate & 2021--2023 & Irregular stream sites & 122 & 1/3/28 & 439 & 0.3 \\
\bottomrule
\end{tabular}
}
\caption{\it Sampling structure for three data sets:   Coverage is the percentage of possible calendar-day records after the stated site and quality filters.}
\label{tab:application-data-summary}
\end{table}

\subsection{North Atlantic sea-surface temperature}

The NOAA OISST Version 2.1 Climate Data Record contains North Atlantic sea-surface temperatures and 
serves as  an example of spatially and temporally dense data. It 
combines bias-adjusted satellite, ship, buoy, and Argo observations on a complete $0.25^\circ$ global grid, available daily since September 1981 \citep{NOAAOISSTv21,HuangEtAl2021OISST}.  We analyzed 2023 SST over $30.125^\circ$--$49.625^\circ$N and $79.875^\circ$--$40.375^\circ$W, retained every second grid point, and required at least 360 days per ocean cell.  All 2565 retained cells had 364 observations, yielding a regular $0.5^\circ$ grid with 99.7\% calendar-day coverage.

The day of year was rescaled to the interval $\left[0,1\right]$.  The stationary estimator \eqref{eq:mean-est} used cubic B-splines with 14 interior knots.  The local-kernel estimator \eqref{hol3} used weights $\{1-\left(r/225\right)^2\}^2\mathbf 1\left\{r<225\right\}$, where $r$ is great-circle distance in kilometers.  We compared three targets in the warmer Sargasso Sea region (SS1--SS3) and three near the colder Labrador Current (LC1--LC3), retaining the Gulf Stream target (GS) as a further comparison.

In \Cref{fig:application-oisst} we display the surface of the annual mean temperatures (upper left panel) and the estimators of the mean curves at different spatial locations.  We observe  a clear annual cycle with substantial regional differences.  The Gulf Stream and Sargasso targets (GS and SS1) remain warmer than the Labrador target (LC1) throughout the year.  The annual-mean SS1--LC1 contrast is $18.0^\circ$C, while the seasonal amplitude is about $15.7^\circ$C at LC1 and $9^\circ$C at GS and SS1.

Within the Sargasso region, the three curves differ by at most $1.03^\circ$C and peak within four days in August (lower left panel).  The Labrador curves peak later and differ by up to $3.18^\circ$C, with the southern target generally warmer (lower right panel).  A common mean therefore describes the selected Sargasso locations more closely than the Labrador locations.  Across regions, the differences in level and seasonality favor local rather than basin-wide anomaly baselines.

\begin{figure}[!h]
\centering
\includegraphics[width=0.82\textwidth]{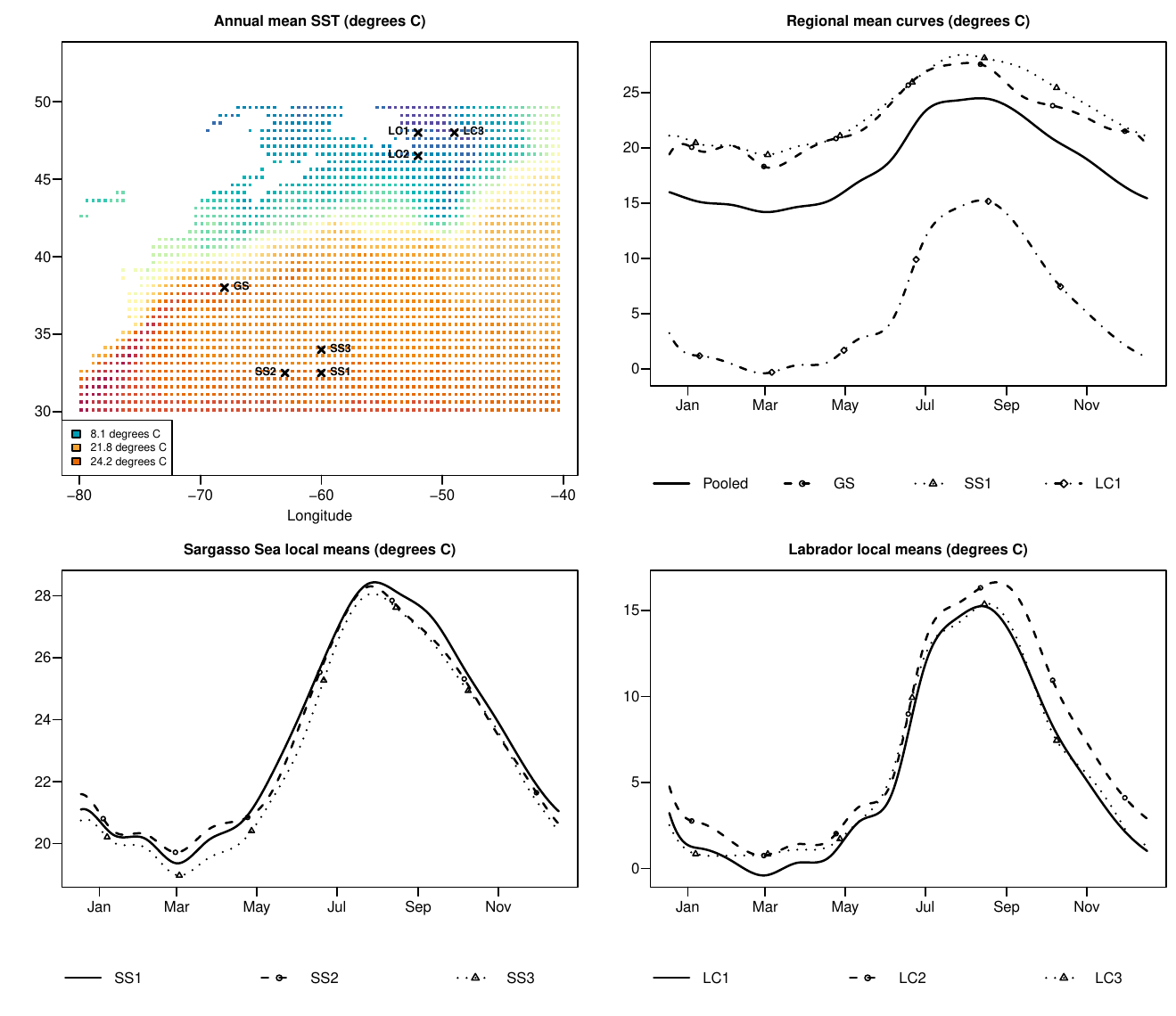}
\vspace{0pt}
\caption{\it North Atlantic SST.  Upper left: annual mean SST on the $0.5^\circ$ grid and seven target locations.  Upper right: pooled mean estimate under stationary assumption  and local mean estimate at location GS, SS1, and LC1.  Lower panels: local mean estimate at Sargasso and Labrador locations, with separate temperature scales.}
\label{fig:application-oisst}
\end{figure}

\subsection{California ground-level ozone}

The U.S. Environmental Protection Agency AQS distributes quality-assured ambient pollution data through AirData \citep{USEPAAirData2026}.  To illustrate the estimators for  spatially sparse and temporally dense data, we used California ozone records (parameter 44201) from the 2023 daily file.  Ozone forms through photochemical reactions involving nitrogen oxides and volatile organic compounds; hot, sunny conditions often raise concentrations, while transport and meteorology produce regional differences \citep{USEPAOzoneBasics2026}.  We retained ``8-HR RUN AVG BEGIN HOUR'' records for the 2015 standard, no exceptional-event flag, and at least 75\% daily completeness.  The response was the daily maximum 8-hour average in ppb.  Duplicate site--day pollutant-occurrence codes were reduced to their maximum, and each site was required to have at least $300$ valid days.  The final sample contains 131 irregular monitors with values $N_i$ ranging from $304$ to $ 363$ and a median of $355$.
The day of year was rescaled to the interval $\left[0,1\right]$.  The stationary estimator \eqref{eq:mean-est} used cubic B-splines with 14 interior knots.  The local-kernel estimator \eqref{hol3} used weights $\{1-(r/175 )^2\}^2\mathbf 1 \{r<175 \}$, where $r$ is great-circle distance in kilometers.  To examine variation within the higher- and lower-ozone regions, we compared targets near Fresno, Bakersfield, and Merced in the Central Valley (CV1--CV3), and San Francisco, San Jose, and Santa Rosa in the Bay Area (SF1--SF3).  We also included a Los Angeles Basin target (LA).

The statewide mean displayed in \Cref{fig:application-aqs} rises into summer and declines in autumn, while local means reveal distinct regional patterns.  Los Angeles peaks near 68 ppb in summer.  The Central Valley curves peak in July at 55.9--65.3 ppb, whereas the Bay Area curves remain lower and peak in April.  The Bay Area curves are also closer to one another: their pairwise root-mean-square differences are 0.64--1.16 ppb, compared with 8.29--13.83 ppb between Valley and Bay targets.  These patterns favor a common mean for the selected Bay locations but reveal appreciable variation within the Valley.  Local means thus distinguish the stronger summer elevation inland from the earlier coastal peak, which the statewide curve obscures.

\begin{figure}[!h]
\centering
\includegraphics[width=0.82\textwidth]{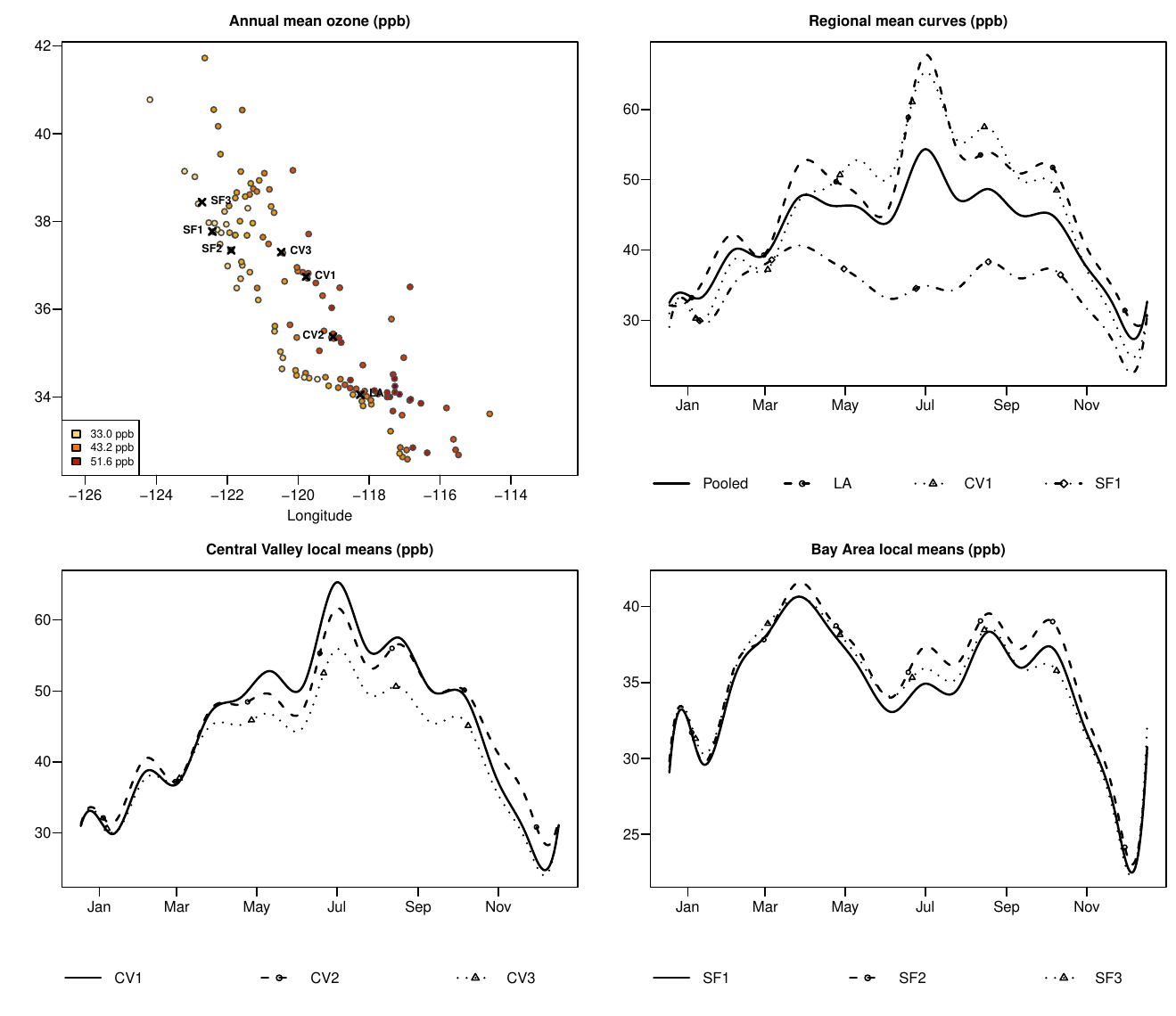}
\vspace{0pt}
\caption{\it California ozone.  Upper left: monitors colored by annual mean daily maximum 8-hour ozone and seven target locations.  Upper right: pooled mean estimate under stationarity assumption and local mean estimates at locations LA, CV1, and SF1.  Lower panels: local mean estimate Central Valley and Bay Area locations, with separate concentration scales.}
\label{fig:application-aqs}
\end{figure}

\section{Concluding remarks}\label{sec:discussion}

This paper develops a unified mean-estimation framework for functional observations whose spatial and temporal sampling densities may vary simultaneously.  Across the stationary, local-kernel, and spatial-spline estimators, independent observation-level variation accumulates according to the nominal number of sampled measurements, whereas curve-level spatial variation is governed by $m_{\mathrm{eff},n}$.  Localization changes these two information scales to $m_na_n^d\bar N_n$ and $m_{\mathrm{eff},n}a_n^d$, while a global spatial spline distributes them across $J_s^d$ spatial degrees of freedom.  Their balance determines the time-sampling phase transition boundary.  The simulations support the resulting information-scale ordering, and the applications illustrate how the estimators recover common temporal structure and regional variation under distinct spatial and temporal sampling designs.

The analysis establishes uniform convergence rate under explicit spatial mixing and design conditions.  Matching lower bounds would clarify the optimality of the dependence-adjusted phase boundaries, while simultaneous confidence bands and spatially dependent bootstrap methods would extend the framework from estimation to inference.  Other useful directions include data-driven spatial resolution, adaptive smoothing under heterogeneous site density, and joint estimation of the mean and covariance structure.

\section*{Declaration of AI use}
For this work, generative AI tools were used to edit language, search the literature, develop code, and help check the authors' mathematical arguments. Before incorporating any AI-assisted material into the manuscript or code, the authors verified it. They take full responsibility for the work.

\section*{Data availability statement}
The sea-surface temperature data are publicly available from \url{https://doi.org/10.25921/RE9P-PT57}.  The ozone data are available from the U.S. Environmental Protection Agency's \url{https://aqs.epa.gov/aqsweb/airdata/download_files.html}, and the nitrate data are available from the \url{https://www.waterqualitydata.us/}. 
\section*{Acknowledgements}
The work of the authors has been partially supported by the Deutsche Forschungsgemeinschaft (DFG): TRR 391 \textit{Spatio-temporal Statistics for the Transition of Energy and Transport} (520388526), and Research Unit 5381 \textit{Mathematical Statistics in the Information Age} (460867398).

\bibliographystyle{abbrvnat}
\spacingset{1.1}
\bibliography{ref}

\clearpage
\setcounter{section}{0}
\setcounter{figure}{0}
\setcounter{table}{0}
\setcounter{equation}{0}
\setcounter{theorem}{0}
\setcounter{page}{1}
\renewcommand{\thesection}{S.\arabic{section}}
\renewcommand{\thefigure}{S.\arabic{figure}}
\renewcommand{\thetable}{S.\arabic{table}}
\renewcommand{\thepage}{S.\arabic{page}}
\renewcommand{\theHsection}{supp.\arabic{section}}
\renewcommand{\theHfigure}{supp.\arabic{figure}}
\renewcommand{\theHtable}{supp.\arabic{table}}
\renewcommand{\theHequation}{supp.\arabic{section}.\arabic{equation}}
\renewcommand{\theHtheorem}{supp.\arabic{section}.\arabic{theorem}}
\section*{Supplementary materials}
\section{Numerical results}
\subsection{Simulation results in Section 4}

The tables in this section give the numerical results discussed in \MainCref{sec:simulation}.  \Cref{tab:mean-sim-primary,tab:mean-sim-local-pointwise} report selected values of the errors plotted in \MainCref{fig:stationary-mean-rates-(D0)-(D4),fig:local-mean-rates-(D0)-(D4)} for the stationary and local-kernel estimators, respectively.  \Cref{tab:mean-sim-local-methods} compares the local-kernel and spatial-spline estimators over the spatial region, and \Cref{tab:mean-sim-robustness-latent} examines different score distributions. 

\begin{table}[!ht]
\centering

\scriptsize
\setlength{\tabcolsep}{2.5pt}

\resizebox{\linewidth}{!}{%
\begin{tabular}{ll*{14}{c|}c}
\hline
\multicolumn{2}{c}{Design}
& \multicolumn{15}{c}{$m_n$} \\
\hline
Space & Time
& \multicolumn{5}{c|}{$400$}
& \multicolumn{5}{c|}{$3600$}
& \multicolumn{5}{c}{$10000$} \\
\hline
& &
(D0) & (D1) & (D2) & (D3) & (D4) &
(D0) & (D1) & (D2) & (D3) & (D4) &
(D0) & (D1) & (D2) & (D3) & (D4) \\
\hline
Sparse & Sparse
& 0.129 & 0.171 & 0.152 & 0.171 & 0.173
& 0.044 & 0.056 & 0.047 & 0.056 & 0.058
& 0.031 & 0.035 & 0.031 & 0.037 & 0.039 \\
Sparse & Dense
& 0.088 & 0.158 & 0.106 & 0.139 & 0.143
& 0.027 & 0.044 & 0.032 & 0.042 & 0.048
& 0.018 & 0.026 & 0.020 & 0.027 & 0.028 \\
Dense & Sparse
& 0.128 & 0.507 & 0.422 & 0.521 & 0.499
& 0.047 & 0.302 & 0.257 & 0.324 & 0.323
& 0.030 & 0.247 & 0.191 & 0.256 & 0.252 \\
Dense & Dense
& 0.091 & 0.522 & 0.384 & 0.467 & 0.509
& 0.029 & 0.302 & 0.233 & 0.329 & 0.313
& 0.017 & 0.226 & 0.188 & 0.256 & 0.243 \\
\hline
\end{tabular}%
}
\caption{\it Simulated sup-norm error for the estimator \eqref{eq:mean-est}.
For each $m_n$, the five subcolumns correspond to the different scenarios (D0)--(D4).}
\label{tab:mean-sim-primary}
\end{table}

\begin{table}[!ht]
\centering

\scriptsize
\setlength{\tabcolsep}{2.5pt}

\resizebox{\linewidth}{!}{%
\begin{tabular}{ll*{14}{c|}c}
\hline
\multicolumn{2}{c}{Design}
& \multicolumn{15}{c}{$m_n$} \\
\hline
Space & Time
& \multicolumn{5}{c|}{$400$}
& \multicolumn{5}{c|}{$3600$}
& \multicolumn{5}{c}{$10000$} \\
\hline
& &
(D0) & (D1) & (D2) & (D3) & (D4) &
(D0) & (D1) & (D2) & (D3) & (D4) &
(D0) & (D1) & (D2) & (D3) & (D4) \\
\hline
Sparse & Sparse
& 0.499 & 0.654 & 0.591 & 0.612 & 0.673
& 0.257 & 0.398 & 0.309 & 0.378 & 0.400
& 0.204 & 0.297 & 0.245 & 0.313 & 0.309 \\
Sparse & Dense
& 0.278 & 0.405 & 0.355 & 0.443 & 0.456
& 0.151 & 0.255 & 0.184 & 0.254 & 0.263
& 0.109 & 0.199 & 0.146 & 0.195 & 0.199 \\
Dense & Sparse
& 0.482 & 0.893 & 0.850 & 0.975 & 0.867
& 0.272 & 0.771 & 0.726 & 0.841 & 0.700
& 0.203 & 0.697 & 0.655 & 0.749 & 0.706 \\
Dense & Dense
& 0.283 & 0.672 & 0.597 & 0.768 & 0.646
& 0.149 & 0.553 & 0.430 & 0.632 & 0.550
& 0.111 & 0.522 & 0.381 & 0.512 & 0.492 \\
\hline
\end{tabular}%
}
\caption{\it Simulated sup-norm error for the local-kernel estimator \eqref{hol3} at $u_0=(0.5,0.5)^\top$.
For each $m_n$, the five subcolumns correspond to the different scenarios (D0)--(D4).}
\label{tab:mean-sim-local-pointwise}
\end{table}

\begin{table}[!ht]
\centering
\setlength{\tabcolsep}{2.4pt}
\begin{tabular}{l|l|r|r|r|r|r|r|r}
\toprule
Space--time & Method & $400$ & $900$ & $1600$ & $2500$ & $3600$ & $6400$ & $10000$ \\
\midrule
Sparse--Sparse & Local kernel & 1.341 & 1.107 & 0.984 & 0.861 & 0.817 & 0.702 & 0.642 \\
Sparse--Sparse & Spatial spline & 0.837 & 0.553 & 0.439 & 0.329 & 0.272 & 0.287 & 0.231 \\
Sparse--Dense & Local kernel & 0.836 & 0.713 & 0.631 & 0.587 & 0.516 & 0.455 & 0.413 \\
Sparse--Dense & Spatial spline & 0.732 & 0.508 & 0.387 & 0.306 & 0.258 & 0.291 & 0.233 \\
Dense--Sparse & Local kernel & 1.584 & 1.484 & 1.457 & 1.478 & 1.436 & 1.430 & 1.395 \\
Dense--Sparse & Spatial spline & 1.795 & 1.605 & 1.479 & 1.424 & 1.339 & 1.261 & 1.156 \\
Dense--Dense & Local kernel & 1.178 & 1.129 & 1.052 & 1.035 & 1.002 & 0.984 & 0.957 \\
Dense--Dense & Spatial spline & 1.693 & 1.584 & 1.445 & 1.406 & 1.301 & 1.211 & 1.146 \\
\bottomrule
\end{tabular}
\caption{\it Simulated uniform sup-norm errors for the local-kernel estimator \eqref{hol3} and spatial-spline estimator \eqref{hol6} under scenario (D1).}
\label{tab:mean-sim-local-methods}
\end{table}

\begin{table}[!ht]

\centering
\setlength{\tabcolsep}{4.2pt}
\begin{tabular}{l|l|r|r|r|r}
\toprule
Space--time & Target & Gaussian & Uniform & Laplace & $t_5$ \\
\midrule
Sparse--Sparse & Stationary & 0.037 & 0.035 & 0.037 & 0.036 \\
Sparse--Sparse & Local & 0.314 & 0.308 & 0.308 & 0.319 \\
Sparse--Dense & Stationary & 0.026 & 0.026 & 0.027 & 0.027 \\
Sparse--Dense & Local & 0.200 & 0.194 & 0.196 & 0.192 \\
Dense--Sparse & Stationary & 0.253 & 0.255 & 0.255 & 0.250 \\
Dense--Sparse & Local & 0.734 & 0.693 & 0.752 & 0.725 \\
Dense--Dense & Stationary & 0.249 & 0.257 & 0.253 & 0.253 \\
Dense--Dense & Local & 0.476 & 0.524 & 0.503 & 0.499 \\
\bottomrule
\end{tabular}
\caption{\it Simulated sup-norm errors for the stationary estimator \eqref{eq:mean-est} and pointwise local-kernel estimator \eqref{hol3} for four different score distributions in  scenario (D1) ($m_n=10000$).}
\label{tab:mean-sim-robustness-latent}
\end{table}
\clearpage
\subsection{California stream nitrate}

The Water Quality Portal (WQP) integrates publicly available discrete water-quality data from the United States Geological Survey (USGS), the Environmental Protection Agency (EPA), and other monitoring organizations \citep{WQP2021}.  We queried the National Water Information System (NWIS) provider only for California stream sites with nitrate results during $2021$--$2023$  representing a spatially sparse and temporally sparse data example.   To keep the measurement definition uniform, we retained accepted dissolved-nitrate results reported under USGS parameter 00618 in mg/L as N, and averaged replicate determinations from the same site and day.  This yielded $122$ sites and $439$ distinct site--day observations.  {Each site had $1-28$ observations over three years, with a median of three, giving a temporally sparse design.}  Nitrate is a mobile form of nitrogen in streams, and excessive nutrient enrichment can contribute to eutrophication \citep{USGSNutrients2019}.

For the right-skewed nitrate concentrations, we fitted $\log\left(1+Y_{ij}\right)$ using four cubic basis functions with no interior knot.  The curves display $\exp\{\text{fitted log-scale curve}\}-1$, a back-transformed summary of the mean on the log scale.  This low-dimensional fit captures broad seasonal variation in the sparse observations.  We compared the stationary statewide curve with a local curve centered at $\left(34.05^\circ\mathrm{N},118.25^\circ\mathrm{W}\right)$ in Southern California, using a 500-km biweight bandwidth.  The local window contains 25 sites and 83 observations over 30 year--month combinations.

\Cref{fig:application-wqp} shows that the observations are concentrated around the Sacramento--San Joaquin Delta.  The statewide fit declines from spring to early autumn before rising toward winter; the Southern California fit instead has a marked spring elevation and approaches the statewide level later in the year.  {The fitted seasonal patterns thus differ between the two spatial scales.}

\begin{figure}[!h]
\centering
\includegraphics[width=0.92\textwidth]{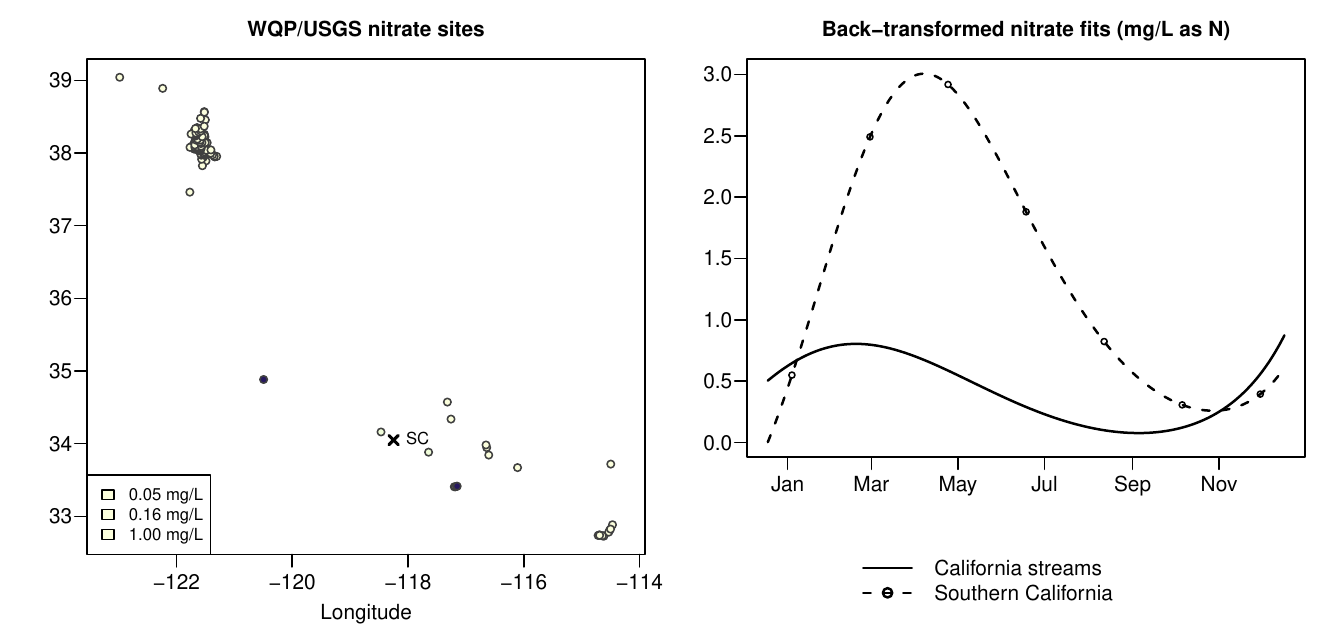}
\caption{\it California WQP/USGS dissolved-nitrate application.  The left panel shows the 122 retained stream sites colored by their back-transformed sitewise average of $\log\left(1+Y_{ij}\right)$; the cross marks the Southern California (SC) target.  The right panel shows the back-transformed stationary and SC local log-scale fits.}
\label{fig:application-wqp}
\end{figure}

\clearpage
\subsection{Additional simulation results}\label{app:additional-simulations}

The stationary and locally stationary models, spatial grids, and tuning rules are those in \MainCref{sec:sim-stationary,sec:sim-local}. All reported errors are averaged over $200$ replications. In the tables, Space--time lists the spatial sampling regime first and the temporal sampling regime second. Under scenarios (D1)--(D4) in the dense space regime, the score vectors $\left(\xi_{1k},\ldots,\xi_{m_n,k}\right)^\top$ were generated on the regular grid by exact block-circulant embedding of the stated covariance functions.

For the measurement-error experiments under (D1), we retain Gaussian score vectors and replace the measurement errors by standardized Gaussian, uniform, Laplace, or $t_5$ variables. We use $m_n=10000$ and the ranges for $N_i$ specified in \MainCref{sec:sim-stationary}. \Cref{tab:mean-sim-robustness-error-hom,tab:mean-sim-robustness-error-het} use $\sigma(s,t)=0.3$ and $\sigma(s,t)=0.3\{1+0.5\sin(\pi t)\}$, respectively. Stationary errors are computed for estimator~(\ref*{eq:mean-est}) over the $101$ time points used in the main article. Local errors use estimator~(\ref*{hol3}) at $u_0=(0.5,0.5)^\top$, taking the maximum over the same time points.

\begin{table}[H]
\centering
\caption{Sup-norm errors for the stationary estimator~(\ref*{eq:mean-est}) and local-kernel estimator~(\ref*{hol3}) at $u_0=(0.5,0.5)^\top$ under (D1) with four measurement-error distributions and $\sigma(s,t)=0.3$, at $m_n=10000$, with $N_i$ determined by $m_n$.}
\label{tab:mean-sim-robustness-error-hom}
\small
\setlength{\tabcolsep}{4.2pt}
\begin{tabular}{l|l|r|r|r|r}
\toprule
Space--time & Target & Gaussian & Uniform & Laplace & $t_5$ \\
\midrule
Sparse--Sparse & Stationary & 0.038 & 0.038 & 0.037 & 0.037 \\
Sparse--Sparse & Local & 0.316 & 0.306 & 0.305 & 0.318 \\
Sparse--Dense & Stationary & 0.028 & 0.027 & 0.028 & 0.027 \\
Sparse--Dense & Local & 0.198 & 0.200 & 0.208 & 0.196 \\
Dense--Sparse & Stationary & 0.249 & 0.253 & 0.246 & 0.261 \\
Dense--Sparse & Local & 0.687 & 0.718 & 0.674 & 0.690 \\
Dense--Dense & Stationary & 0.249 & 0.244 & 0.256 & 0.248 \\
Dense--Dense & Local & 0.485 & 0.497 & 0.529 & 0.512 \\
\bottomrule
\end{tabular}
\end{table}

\begin{table}[H]
\centering
\caption{Sup-norm errors for the stationary estimator~(\ref*{eq:mean-est}) and local-kernel estimator~(\ref*{hol3}) at $u_0=(0.5,0.5)^\top$ under (D1) with four measurement-error distributions and $\sigma(s,t)=0.3\{1+0.5\sin(\pi t)\}$, at $m_n=10000$, with $N_i$ determined by $m_n$.}
\label{tab:mean-sim-robustness-error-het}
\small
\setlength{\tabcolsep}{4.2pt}
\begin{tabular}{l|l|r|r|r|r}
\toprule
Space--time & Target & Gaussian & Uniform & Laplace & $t_5$ \\
\midrule
Sparse--Sparse & Stationary & 0.036 & 0.035 & 0.037 & 0.039 \\
Sparse--Sparse & Local & 0.317 & 0.303 & 0.314 & 0.319 \\
Sparse--Dense & Stationary & 0.027 & 0.026 & 0.026 & 0.028 \\
Sparse--Dense & Local & 0.195 & 0.199 & 0.200 & 0.199 \\
Dense--Sparse & Stationary & 0.266 & 0.246 & 0.252 & 0.249 \\
Dense--Sparse & Local & 0.723 & 0.695 & 0.733 & 0.688 \\
Dense--Dense & Stationary & 0.251 & 0.226 & 0.249 & 0.249 \\
Dense--Dense & Local & 0.516 & 0.522 & 0.477 & 0.485 \\
\bottomrule
\end{tabular}
\end{table}

\Cref{tab:mean-sim-robustness-latent} gives the corresponding experiment that changes the score innovation distribution while retaining Gaussian errors with $\sigma(s,t)=0.3$. The score vectors are obtained by multiplying standardized innovations by the same covariance square root. The three experiments were run separately, so their Gaussian columns provide separate comparisons. Across these experiments, the effects of changing the score innovation or error distribution are small relative to the differences between spatial sampling regimes. The Laplace and $t_5$ settings provide finite-sample evidence beyond the sub-Gaussian conditions used in the theory.

\label{app:additional-simulations-meff}
In the main experiment, the two ranges for $N_i$ depend on $m_n$ and are therefore the same across spatial designs at a given $m_n$. Here we replace $m_n$ by $m_{\mathrm{eff},n}$ in those ranges: $N_i$ is sampled uniformly from the integers between $2m_{\mathrm{eff},n}^{1/5}$ and $4m_{\mathrm{eff},n}^{1/5}$ in the sparse time regime, and between $1.5m_{\mathrm{eff},n}^{1/2}$ and $2.5m_{\mathrm{eff},n}^{1/2}$ in the dense time regime, with endpoints rounded down. We use the same values of $m_{\mathrm{eff},n}$ as in \MainCref{sec:sim-stationary}.  This comparison examines whether the effect of spatial dependence persists when the number of observations per site also varies with the spatial design. The models, tuning rules, and error evaluation grids remain unchanged.

\begin{table}[H]
\centering
\caption{Sup-norm errors for the stationary estimator~(\ref*{eq:mean-est})
with $N_i$ determined by $m_{\mathrm{eff},n}$.
For each $m_n$, the five subcolumns correspond to (D0)--(D4).}
\label{tab:mean-sim-primary-meff}
\scriptsize
\setlength{\tabcolsep}{2.5pt}

\resizebox{\linewidth}{!}{%
\begin{tabular}{ll*{14}{c|}c}
\hline
\multicolumn{2}{c}{Design}
& \multicolumn{15}{c}{$m_n$} \\
\hline
Space & Time
& \multicolumn{5}{c|}{$400$}
& \multicolumn{5}{c|}{$3600$}
& \multicolumn{5}{c}{$10000$} \\
\hline
& &
(D0) & (D1) & (D2) & (D3) & (D4) &
(D0) & (D1) & (D2) & (D3) & (D4) &
(D0) & (D1) & (D2) & (D3) & (D4) \\
\hline
Sparse & Sparse
& 0.130 & 0.166 & 0.145 & 0.159 & 0.172
& 0.045 & 0.058 & 0.048 & 0.057 & 0.060
& 0.029 & 0.036 & 0.032 & 0.036 & 0.036 \\
Sparse & Dense
& 0.086 & 0.149 & 0.110 & 0.141 & 0.154
& 0.029 & 0.044 & 0.034 & 0.044 & 0.047
& 0.018 & 0.025 & 0.020 & 0.026 & 0.026 \\
Dense & Sparse
& 0.127 & 0.495 & 0.386 & 0.536 & 0.503
& 0.047 & 0.304 & 0.243 & 0.325 & 0.305
& 0.028 & 0.262 & 0.192 & 0.267 & 0.250 \\
Dense & Dense
& 0.096 & 0.508 & 0.413 & 0.540 & 0.497
& 0.030 & 0.330 & 0.248 & 0.342 & 0.307
& 0.017 & 0.254 & 0.185 & 0.245 & 0.261 \\
\hline
\end{tabular}%
}
\end{table}

\Cref{tab:mean-sim-primary-meff} shows that the stationary errors decrease as $m_n$ increases. The decrease remains slower in the dense space regime under (D1)--(D4), even though the numbers of observations per site are now determined by $m_{\mathrm{eff},n}$.

The comparison at $u_0$ uses the locally stationary model in \MainCref{sec:sim-local}, including the spatially varying mean and score variances. Errors are evaluated at $u_0=(0.5,0.5)^\top$ over $101$ time points. For the regional comparison, both estimators are evaluated on the same $25$ spatial points and $101$ time points as in the main article, taking the maximum over all $25\times101$ combinations.

\begin{table}[H]
\centering
\caption{Simulated sup-norm errors for the local-kernel estimator~(\ref*{hol3})
at $u_0=(0.5,0.5)^\top$, with $N_i$ determined by $m_{\mathrm{eff},n}$.
For each $m_n$, the five subcolumns correspond to (D0)--(D4).}
\label{tab:mean-sim-local-pointwise-meff}
\scriptsize
\setlength{\tabcolsep}{2.5pt}

\resizebox{\linewidth}{!}{%
\begin{tabular}{ll*{14}{c|}c}
\hline
\multicolumn{2}{c}{Design}
& \multicolumn{15}{c}{$m_n$} \\
\hline
Space & Time
& \multicolumn{5}{c|}{$400$}
& \multicolumn{5}{c|}{$3600$}
& \multicolumn{5}{c}{$10000$} \\
\hline
& &
(D0) & (D1) & (D2) & (D3) & (D4) &
(D0) & (D1) & (D2) & (D3) & (D4) &
(D0) & (D1) & (D2) & (D3) & (D4) \\
\hline
Sparse & Sparse
& 0.501 & 0.643 & 0.560 & 0.647 & 0.667
& 0.255 & 0.363 & 0.312 & 0.382 & 0.376
& 0.199 & 0.318 & 0.245 & 0.291 & 0.312 \\
Sparse & Dense
& 0.281 & 0.462 & 0.353 & 0.446 & 0.430
& 0.147 & 0.245 & 0.182 & 0.244 & 0.255
& 0.116 & 0.197 & 0.149 & 0.190 & 0.199 \\
Dense & Sparse
& 0.492 & 0.815 & 0.800 & 0.940 & 0.800
& 0.251 & 0.665 & 0.587 & 0.736 & 0.660
& 0.209 & 0.637 & 0.526 & 0.684 & 0.600 \\
Dense & Dense
& 0.282 & 0.706 & 0.617 & 0.783 & 0.669
& 0.149 & 0.575 & 0.480 & 0.564 & 0.561
& 0.110 & 0.488 & 0.407 & 0.545 & 0.477 \\
\hline
\end{tabular}%
}
\end{table}

\begin{table}[H]
\centering
\caption{Regional sup-norm errors for the local-kernel estimator~(\ref*{hol3}) and spatial-spline estimator~(\ref*{hol6}) under (D1), with $N_i$ determined by $m_{\mathrm{eff},n}$. }
\label{tab:mean-sim-local-methods-meff}
\scriptsize
\setlength{\tabcolsep}{2.4pt}
\begin{tabular}{l|l|r|r|r|r|r|r|r}
\toprule
Space--time & Method$\backslash m_n$ & $400$ & $900$ & $1600$ & $2500$ & $3600$ & $6400$ & $10000$ \\
\midrule
Sparse--Sparse & Local kernel & 1.319 & 1.104 & 0.985 & 0.872 & 0.801 & 0.696 & 0.639 \\
Sparse--Sparse & Spatial spline & 0.831 & 0.564 & 0.423 & 0.329 & 0.268 & 0.301 & 0.234 \\
Sparse--Dense & Local kernel & 0.866 & 0.725 & 0.626 & 0.574 & 0.524 & 0.449 & 0.405 \\
Sparse--Dense & Spatial spline & 0.752 & 0.520 & 0.387 & 0.307 & 0.259 & 0.292 & 0.229 \\
Dense--Sparse & Local kernel & 1.464 & 1.340 & 1.302 & 1.317 & 1.254 & 1.204 & 1.223 \\
Dense--Sparse & Spatial spline & 1.841 & 1.607 & 1.473 & 1.447 & 1.335 & 1.190 & 1.142 \\
Dense--Dense & Local kernel & 1.184 & 1.126 & 1.071 & 1.030 & 1.035 & 0.993 & 0.957 \\
Dense--Dense & Spatial spline & 1.786 & 1.581 & 1.457 & 1.374 & 1.355 & 1.238 & 1.148 \\
\bottomrule
\end{tabular}
\end{table}

\Cref{tab:mean-sim-local-pointwise-meff,fig:local-mean-rates-meff} show the same effect of spatial dependence for the local mean at $u_0$. In the regional comparison in \Cref{tab:mean-sim-local-methods-meff}, in the dense space regime, the spatial-spline estimator becomes more accurate as $m_n$ increases under sparse time sampling, whereas the local-kernel estimator remains more accurate under dense time sampling. These patterns agree with the main experiment, although the errors in the dense space regime change when fewer observations are taken within each site.

\begin{figure}[H]
\centering
\includegraphics[width=0.94\textwidth]{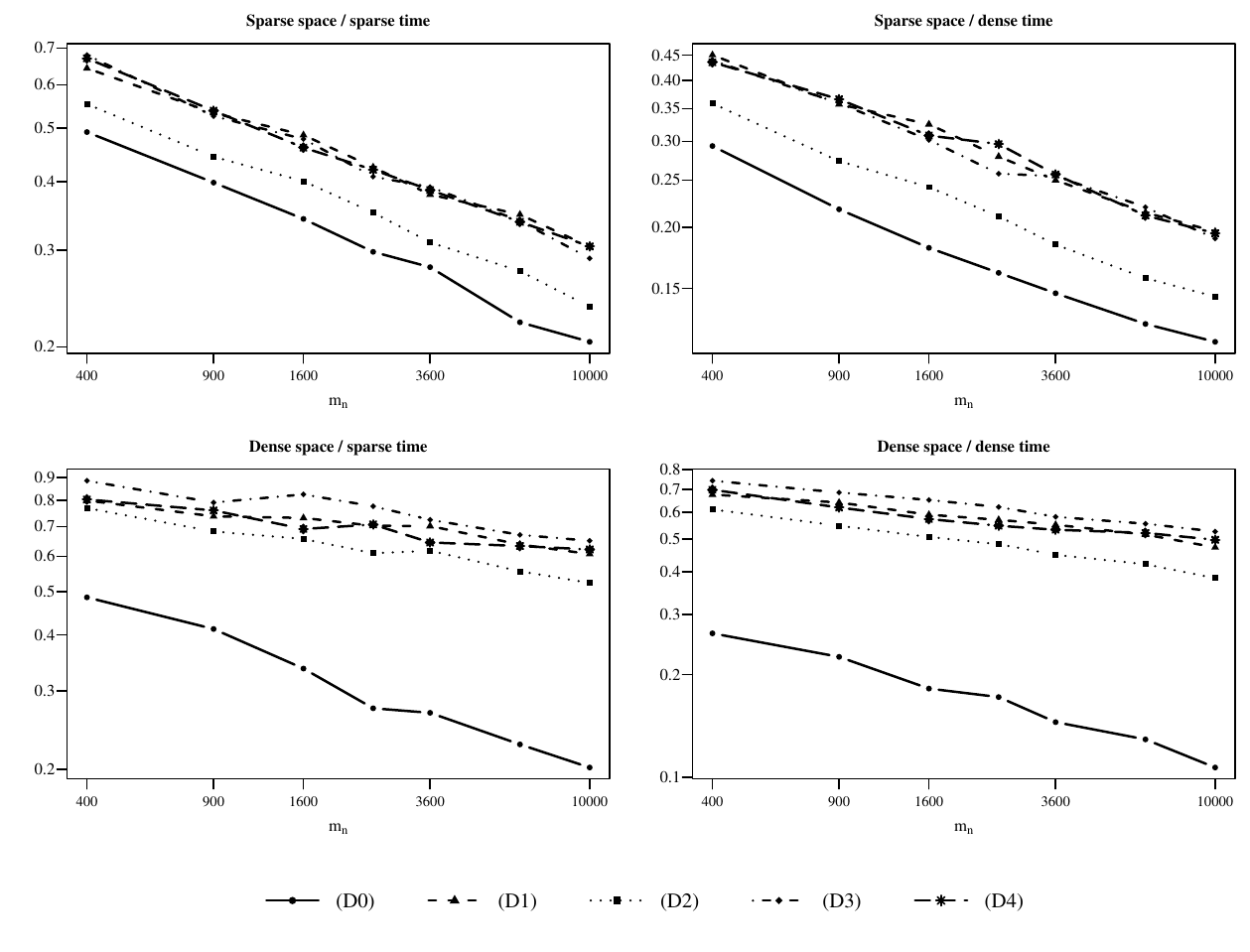}
\caption{Simulated sup-norm errors for the local-kernel estimator~(\ref*{hol3}) at $u_0=(0.5,0.5)^\top$ over all seven values of $m_n$ under (D0)--(D4), with $N_i$ determined by $m_{\mathrm{eff},n}$.}
\label{fig:local-mean-rates-meff}
\end{figure}

The distribution experiments under (D1) use $m_n=10000$ and the alternative ranges for $N_i$ above. We change either the score innovations or the measurement errors, retaining the other part of the model as Gaussian. As in the main experiment, the score comparison uses $\sigma(s,t)=0.3$, and the error comparison uses both the constant and time-varying scales. The three experiments are run separately.

\begin{table}[H]
\centering
\caption{Sup-norm errors for the stationary estimator~(\ref*{eq:mean-est}) and local-kernel estimator~(\ref*{hol3}) at $u_0=(0.5,0.5)^\top$ under (D1) with four score distributions at $m_n=10000$, with $N_i$ determined by $m_{\mathrm{eff},n}$.}
\label{tab:mean-sim-robustness-latent-meff}
\small
\setlength{\tabcolsep}{4.2pt}
\begin{tabular}{l|l|r|r|r|r}
\toprule
Space--time & Target & Gaussian & Uniform & Laplace & $t_5$ \\
\midrule
Sparse--Sparse & Stationary & 0.037 & 0.037 & 0.037 & 0.036 \\
Sparse--Sparse & Local & 0.307 & 0.324 & 0.310 & 0.295 \\
Sparse--Dense & Stationary & 0.027 & 0.028 & 0.027 & 0.026 \\
Sparse--Dense & Local & 0.205 & 0.190 & 0.190 & 0.188 \\
Dense--Sparse & Stationary & 0.255 & 0.252 & 0.244 & 0.236 \\
Dense--Sparse & Local & 0.606 & 0.611 & 0.641 & 0.637 \\
Dense--Dense & Stationary & 0.247 & 0.248 & 0.257 & 0.243 \\
Dense--Dense & Local & 0.499 & 0.491 & 0.512 & 0.474 \\
\bottomrule
\end{tabular}
\end{table}

\begin{table}[H]
\centering
\caption{Sup-norm errors for the stationary estimator~(\ref*{eq:mean-est}) and local-kernel estimator~(\ref*{hol3}) at $u_0=(0.5,0.5)^\top$ under (D1) with four measurement-error distributions and $\sigma(s,t)=0.3$, at $m_n=10000$, with $N_i$ determined by $m_{\mathrm{eff},n}$.}
\label{tab:mean-sim-robustness-error-hom-meff}
\small
\setlength{\tabcolsep}{4.2pt}
\begin{tabular}{l|l|r|r|r|r}
\toprule
Space--time & Target & Gaussian & Uniform & Laplace & $t_5$ \\
\midrule
Sparse--Sparse & Stationary & 0.037 & 0.035 & 0.036 & 0.038 \\
Sparse--Sparse & Local & 0.295 & 0.308 & 0.300 & 0.316 \\
Sparse--Dense & Stationary & 0.027 & 0.025 & 0.027 & 0.026 \\
Sparse--Dense & Local & 0.194 & 0.198 & 0.201 & 0.196 \\
Dense--Sparse & Stationary & 0.261 & 0.254 & 0.248 & 0.258 \\
Dense--Sparse & Local & 0.636 & 0.614 & 0.637 & 0.615 \\
Dense--Dense & Stationary & 0.239 & 0.265 & 0.255 & 0.261 \\
Dense--Dense & Local & 0.514 & 0.491 & 0.484 & 0.516 \\
\bottomrule
\end{tabular}
\end{table}

\begin{table}[H]
\centering
\caption{Sup-norm errors for the stationary estimator~(\ref*{eq:mean-est}) and local-kernel estimator~(\ref*{hol3}) at $u_0=(0.5,0.5)^\top$ under (D1) with four measurement-error distributions and $\sigma(s,t)=0.3\{1+0.5\sin(\pi t)\}$, at $m_n=10000$, with $N_i$ determined by $m_{\mathrm{eff},n}$.}
\label{tab:mean-sim-robustness-error-het-meff}
\small
\setlength{\tabcolsep}{4.2pt}
\begin{tabular}{l|l|r|r|r|r}
\toprule
Space--time & Target & Gaussian & Uniform & Laplace & $t_5$ \\
\midrule
Sparse--Sparse & Stationary & 0.038 & 0.036 & 0.037 & 0.037 \\
Sparse--Sparse & Local & 0.325 & 0.314 & 0.309 & 0.296 \\
Sparse--Dense & Stationary & 0.027 & 0.025 & 0.027 & 0.027 \\
Sparse--Dense & Local & 0.203 & 0.197 & 0.206 & 0.195 \\
Dense--Sparse & Stationary & 0.253 & 0.262 & 0.256 & 0.261 \\
Dense--Sparse & Local & 0.619 & 0.634 & 0.608 & 0.636 \\
Dense--Dense & Stationary & 0.252 & 0.253 & 0.247 & 0.239 \\
Dense--Dense & Local & 0.462 & 0.487 & 0.498 & 0.515 \\
\bottomrule
\end{tabular}
\end{table}
\Cref{tab:mean-sim-robustness-latent-meff,tab:mean-sim-robustness-error-hom-meff,tab:mean-sim-robustness-error-het-meff} show that the differences between spatial sampling regimes remain larger than those between the score or measurement-error distributions. Together with the preceding (D0)--(D4) comparisons, these results indicate that the effect of spatial dependence persists under both choices of the ranges for $N_i$.

\clearpage
\section{Technical lemmas}\label{app:spatial-tools}
\begin{lemma}\label[lemma]{lem:shell}
Under \MainCref{ass:design}, there exists a constant $C>0$ such that, for every site $S_i$ and every integer $u\ge0$,
\begin{equation*}
\#\left\{k:u\le \norm{S_k-S_i}_2<u+1\right\}
\le C\,\delta_n^{-d}\left(1+u\right)^{d-1}.
\end{equation*}
\end{lemma}

\begin{proof}
Fix $S_i$.  If $u\le \norm{S_k-S_i}_2<u+1$, every point in the design cell containing $S_k$ lies in the enlarged annulus
\begin{equation*}
A_{u,n}\left(S_i\right):=
\left\{s\in\R^d:
\left(u-\sqrt d\,\delta_n\right)_+
\le \norm{s-S_i}_2
<u+1+\sqrt d\,\delta_n
\right\}.
\end{equation*}
The relevant design cells are disjoint and each has volume $\delta_n^d$, so their number is at most $|A_{u,n}\left(S_i\right)|/\delta_n^d$.  By \MainCref{ass:design}, $\delta_n=O\left(1\right)$, and hence
\begin{align*}
|A_{u,n}\left(S_i\right)| =v_d\left[ \left(u+1+\sqrt d\,\delta_n\right)^d -\left(u-\sqrt d\,\delta_n\right)_+^d \right] \le C\left(1+u\right)^{d-1},
\end{align*}
where the last bound follows from the mean-value theorem for large $u$ and by enlarging $C$ for bounded $u$.  Dividing by $\delta_n^d$ proves the result.

\end{proof}

\begin{lemma}\label[lemma]{lem:Theta-order}
Under \MainCref{ass:design,ass:beta}, $\Theta_n\le C\delta_n^{-d}$ and, consequently, $m_{\mathrm{eff},n}\ge cV_n$.
\end{lemma}

\begin{proof}
By \Cref{lem:shell},
\begin{equation*}
\sum_{k\neq i}\beta_X\!\left(\norm{S_i-S_k}_2\right)^{1/2}
\le
\sum_{u=0}^\infty
\#\left\{k:u\le\norm{S_k-S_i}_2<u+1\right\}
\sup_{r\in\left[u,u+1\right)}\beta_X\left(r\right)^{1/2}.
\end{equation*}
Applying \Cref{lem:shell} and \MainCref{ass:beta},
\begin{align*}
\sum_{k\neq i}\beta_X\!\left(\norm{S_i-S_k}_2\right)^{1/2}
&\le
C\delta_n^{-d}
\sum_{u=0}^\infty \left(1+u\right)^{d-1}\sup_{r\in\left[u,u+1\right)}\left(1+r\right)^{-\alpha_\beta/2}
\le
C\delta_n^{-d}
\sum_{u=0}^\infty \left(1+u\right)^{d-1-\alpha_\beta/2}.
\end{align*}
Because $\alpha_\beta>2d$, the exponent $d-1-\alpha_\beta/2$ is less than $-1$, and the series converges.
Thus, $\Theta_n\le C\delta_n^{-d}$.
Since $m_n\asymp V_n\delta_n^{-d}$ by \MainCref{ass:design}, it follows that $m_{\mathrm{eff},n}\ge cV_n$.

\end{proof}

\begin{proposition}\label[proposition]{prop:Theta-lower}
Assume \MainCref{ass:design,ass:beta} and that there exists $r_0>0$ such that $\beta_X\left(r_0\right)>0$.  In the dense-space regime $\delta_n\to0$, assume additionally \MainCref{ass:full-dimensional}.  Then $\Theta_n\asymp \delta_n^{-d}$ and $m_{\mathrm{eff},n}\asymp V_n$.
\end{proposition}

\begin{proof}
The upper bound $\Theta_n\le C\delta_n^{-d}$ follows from \Cref{lem:Theta-order}.  If $\delta_n\asymp1$, then $\Theta_n\ge1\asymp\delta_n^{-d}$, so it remains to consider the dense-space regime $\delta_n\to0$.

By \MainCref{ass:design,ass:full-dimensional}, $\min_rL_{r,n}\to\infty$.  Let $x_n=\left(L_{1,n}/2,\ldots,L_{d,n}/2\right)^\top$, and choose the observed site $S_i$ in the design cell containing $x_n$.  For all sufficiently large $n$, $S_i$ is at distance greater than $r_0$ from every face of $\cS_n$.  The ball $B\left(S_i,r_0/2\right)$ therefore contains at least $c\delta_n^{-d}$ complete design cells, each of which contains one observed site.  Every such site $S_k$ satisfies $\|S_i-S_k\|_2\le r_0/2$.  Since $r\mapsto\beta_X\left(r\right)$ is nonincreasing,
\begin{equation*}
\sum_{k\ne i}\beta_X\left(\|S_i-S_k\|_2\right)^{1/2}
\ge c\delta_n^{-d}\beta_X\left(r_0\right)^{1/2}.
\end{equation*}
Thus $\Theta_n\ge c'\delta_n^{-d}$.  Combining the two bounds gives $\Theta_n\asymp\delta_n^{-d}$, and $m_n\asymp V_n\delta_n^{-d}$ from \MainCref{ass:design} then yields $m_{\mathrm{eff},n}\asymp V_n$.
\end{proof}

For the sparse two-dimensional simulation lattice, $\delta_n\asymp1$ and $V_n\asymp m_n$, so \Cref{lem:Theta-order} gives $m_{\mathrm{eff},n}\asymp m_n$.  For the dense lattice, $\delta_n\asymp m_n^{-1/4}, V_n\asymp m_n^{1/2}$. Whenever the simulated score field satisfies \MainCref{ass:beta} and the short-range nonindependence condition of \Cref{prop:Theta-lower}, that proposition gives $\Theta_n\asymp m_n^{1/2}$ and $m_{\mathrm{eff},n}\asymp m_n^{1/2}$.  For D0, $\Theta_n=1$ by independence.

\begin{lemma}\label[lemma]{lem:cov-sum}
Under \MainCref{ass:time}, let $\mathcal G_i:=\sigma\left(T_{ij},\veps_{ij}:1\le j\le N_i\right)$ and let $\left\{Z_i\right\}_{i=1}^{m_n}$ be real-valued random variables satisfying $Z_i\in \cF_X\left(\left\{S_i\right\}\right)\vee\mathcal G_i$.
Assume $\sup_i\norm{Z_i}_4<\infty$.  Then there exists $C>0$ such that, for every $i$,
\begin{equation*}
\sum_{k=1}^{m_n}|\Cov\left(Z_i,Z_k\right)|
\le C\Theta_n \sup_{1\le \ell\le m_n}\norm{Z_\ell}_4^2,
\end{equation*}
and
\begin{equation}\label{eq:var-sum}
\Var\!\left(\sum_{i=1}^{m_n}Z_i\right)
\le C\Theta_n \sum_{i=1}^{m_n}\norm{Z_i}_4^2.
\end{equation}
\end{lemma}

\begin{proof}
By \MainCref{ass:time}, the sigma-fields $\mathcal G_i$ are independent across sites and jointly independent of the latent field.  Adding these independent sigma-fields does not increase the absolute-regularity coefficient between the two point-site field sigma-fields.  Hence, for $i\ne k$, the strong-mixing coefficient between
$\cF_X\left(\left\{S_i\right\}\right)\vee\mathcal G_i$ and
$\cF_X\left(\left\{S_k\right\}\right)\vee\mathcal G_k$ is bounded by
$\beta_X\left(\|S_i-S_k\|_2\right)$.  Rio's covariance inequality \citep[Theorem~1.1]{Rio1993} therefore gives
\begin{equation*}
|\Cov\left(Z_i,Z_k\right)|
\le C\,\beta_X\!\left(\norm{S_i-S_k}_2\right)^{1/2}
\norm{Z_i}_4\norm{Z_k}_4,
\qquad i\ne k.
\end{equation*}
For $i=k$, $\Var\left(Z_i\right)\le\norm{Z_i}_4^2$.  Summing over $k$ and using \MainEqref{eq:Theta} proves the first assertion.  For the second, use $\norm{Z_i}_4\norm{Z_k}_4 \le (1/2)\left\{\norm{Z_i}_4^2+\norm{Z_k}_4^2\right\}$ in the double covariance sum and again apply \MainEqref{eq:Theta}.

\end{proof}

For $\left(t,t'\right)\in\left[0,1\right]^2$, define the tensor-product basis
$\bB_J^{\left(2\right)}\left(t,t'\right):=\bB_J\left(t\right)\otimes\bB_J\left(t'\right)\in\mathbb R^{K_J^2}$.

\begin{lemma}\label[lemma]{lem:Bspline-basic}
There exists a constant $C>0$, depending only on the spline order and the quasi-uniformity constants of the knot sequence, such that, for every $J$:

\begin{enumerate}[label=(\roman*),leftmargin=2.2em]
\item $\sup_{t\in\left[0,1\right]}\norm{\bB_J\left(t\right)}_1=1$ and $\sup_{\left(t,t'\right)\in\left[0,1\right]^2}\norm{\bB_J^{\left(2\right)}\left(t,t'\right)}_1=1$;
\item
For each fixed $t$, at most $p$ entries of $\bB_J\left(t\right)$ are nonzero; for each fixed $\left(t,t'\right)$, at most $p^2$ entries of $\bB_J^{\left(2\right)}\left(t,t'\right)$ are nonzero;

\item For all $t,s\in\left[0,1\right]$, $\norm{\bB_J\left(t\right)-\bB_J\left(s\right)}_1\le C J|t-s|$;
\item For all $\left(t,t'\right),\left(s,s'\right)\in\left[0,1\right]^2$, $\norm{\bB_J^{\left(2\right)}\left(t,t'\right)-\bB_J^{\left(2\right)}\left(s,s'\right)}_1\le C J\left(|t-s|+|t'-s'|\right)$.
\end{enumerate}
\end{lemma}

\begin{proof}
Properties (i) and (ii) follow from the partition-of-unity property and local support of normalized B-splines \citep[Chapters~IX--XI]{deBoor2001}.
On a quasi-uniform mesh, the derivative of each basis function is bounded by $CJ$; hence $|B_{a,J}\left(t\right)-B_{a,J}\left(s\right)|\le CJ|t-s|$.
Summing over the at most $2p$ basis functions that are nonzero at either $t$ or $s$ proves (iii).
For the tensor-product basis,
\begin{equation*}
\bB_J^{\left(2\right)}\left(t,t'\right)-\bB_J^{\left(2\right)}\left(s,s'\right)
=
\left\{\bB_J\left(t\right)-\bB_J\left(s\right)\right\}\otimes \bB_J\left(t'\right)
+
\bB_J\left(s\right)\otimes\left\{\bB_J\left(t'\right)-\bB_J\left(s'\right)\right\},
\end{equation*}
therefore, (iv) follows from (i) and (iii).

\end{proof}

\begin{lemma}\label[lemma]{lem:gram-pop}
Let
\begin{equation*}
Q_{m,J}:=\E\left\{\bB_J\left(T\right)\bB_J\left(T\right)^\top\right\},
\qquad
Q_{G,J}:=\E\left\{\bB_J^{\left(2\right)}\left(T_1,T_2\right)\bB_J^{\left(2\right)}\left(T_1,T_2\right)^\top\right\},
\end{equation*}
where $T,T_1,T_2$ are independent and identically distributed (i.i.d.) with density $f_T$ from \MainCref{ass:time}.
Then there exist constants $0<c<C<\infty$ such that, for every $J$,
\begin{equation*}
cJ^{-1}\le \lambda_{\min}\left(Q_{m,J}\right)\le \lambda_{\max}\left(Q_{m,J}\right)\le CJ^{-1},
\end{equation*}
\begin{equation*}
cJ^{-2}\le \lambda_{\min}\left(Q_{G,J}\right)\le \lambda_{\max}\left(Q_{G,J}\right)\le CJ^{-2}.
\end{equation*}
Moreover, $\oinf{Q_{m,J}^{-1}}\le CJ, \qquad \oinf{Q_{G,J}^{-1}}\le CJ^2$. \end{lemma}

\begin{proof}
For the one-dimensional basis,
\begin{equation*}
x^\top Q_{m,J}x=
\E\left(\sum_{a=1}^{K_J}x_aB_{a,J}\left(T\right)\right)^2
=\int_0^1 \left(\sum_{a=1}^{K_J}x_aB_{a,J}\left(t\right)\right)^2 f_T\left(t\right)\,dt.
\end{equation*}
Because $f_T$ is bounded above and bounded away from zero, this quadratic form is equivalent to the $L^2$ norm of the spline $g_x\left(t\right)=\sum_a x_a B_{a,J}\left(t\right)$.
For normalized B-splines on a quasi-uniform mesh, coefficient stability \citep[Chapter~XI]{deBoor2001} gives
\begin{equation*}
cJ^{-1}\norm{x}_2^2\le \int_0^1 g_x\left(t\right)^2\,dt\le CJ^{-1}\norm{x}_2^2.
\end{equation*}
This proves the eigenvalue bounds for $Q_{m,J}$.
When needed, the same argument on $\left[0,1\right]^2$ gives $cJ^{-2}\norm{x}_2^2\le x^\top Q_{G,J}x\le CJ^{-2}\norm{x}_2^2$. The matrix $Q_{m,J}$ is banded, with bandwidth depending only on $p$, and the inverse of a banded positive-definite matrix has exponentially decaying entries \citep{DemkoMossSmith1984}.  Its eigenvalue scale $J^{-1}$ therefore gives $\oinf{Q_{m,J}^{-1}}\le CJ$.  Independence and the tensor basis give $Q_{G,J}=Q_{m,J}\otimes Q_{m,J}$, so $Q_{G,J}^{-1}=Q_{m,J}^{-1}\otimes Q_{m,J}^{-1}$ and $\oinf{Q_{G,J}^{-1}}=\oinf{Q_{m,J}^{-1}}^2$.  Consequently, $\oinf{Q_{m,J}^{-1}}\le CJ, \qquad \oinf{Q_{G,J}^{-1}}\le CJ^2$. \end{proof}

\begin{lemma}\label[lemma]{lem:spatial-spline-discrete-norm}
Assume \MainCref{ass:balanced,ass:spatial-spline-design}, write
\[
\pi_{i,n}:=\frac{N_i}{M_{1,n}},
\qquad
g_v(u):=\bA_s(u)^\top v,
\]
and let $v\in\mathbb R^{K_s}$.  There are constants $0<c<C<\infty$, independent of $n,J_s$, such that
\begin{equation}\label{eq:spatial-discrete-norm}
c\|g_v\|_{L^2([0,1]^d)}^2
\le
\sum_{i=1}^{m_n}\pi_{i,n}g_v(U_i)^2
\le
C\|g_v\|_{L^2([0,1]^d)}^2,
\end{equation}
and hence
\begin{equation}\label{eq:spatial-coefficient-norm}
cJ_s^{-d}\|v\|_2^2
\le
\sum_{i=1}^{m_n}\pi_{i,n}g_v(U_i)^2
\le
CJ_s^{-d}\|v\|_2^2.
\end{equation}
For every spatial basis coordinate $a$,
\begin{equation}\label{eq:spatial-coordinate-load}
\sum_{i=1}^{m_n}\pi_{i,n}A_a(U_i)\le CJ_s^{-d},
\qquad
\sum_{i=1}^{m_n}\pi_{i,n}A_a(U_i)^2\le CJ_s^{-d}.
\end{equation}
\end{lemma}

\begin{proof}
Let $V_{i,n}$ be the Voronoi cell of $U_i$, intersected with $[0,1]^d$.  The fill-distance bound gives $\operatorname{diam}(V_{i,n})\le2h_{U,n}$.  The separation radius implies that $V_{i,n}$ contains the intersection of $[0,1]^d$ with a ball of radius $q_{U,n}$ centered at $U_i$, whereas the fill-distance bound places it inside a ball of radius $h_{U,n}$.  Since $h_{U,n}/q_{U,n}\le C$ and $h_{U,n}\asymp q_{U,n}\asymp m_n^{-1/d}$, $c m_n^{-1}\le |V_{i,n}|\le C m_n^{-1}$ uniformly, including boundary cells.  Balanced counts give $c/m_n\le\pi_{i,n}\le C/m_n$, and hence $\pi_{i,n}\asymp|V_{i,n}|$.

Define the piecewise-constant sampler $P_ng$ by $P_ng(u)=g(U_i)$ on $V_{i,n}$.  Then
\[
\sum_i\pi_{i,n}g(U_i)^2\asymp\|P_ng\|_{L^2}^2.
\]
Let $\mathcal Q_{J_s}$ be the quasi-uniform tensor-product spline mesh, whose cell diameters are of order $H_s\asymp J_s^{-1}$.  Since $h_{U,n}/H_s=J_sh_{U,n}\to0$, a segment joining $u\in V_{i,n}$ to $U_i$ crosses only a fixed number of neighboring mesh cells for all sufficiently large $n$.  The piecewise mean-value inequality therefore gives
\[
|g(U_i)-g(u)|
\le C h_{U,n}\max_{x\in Q^+}\|\nabla g(x)\|_2
\]
whenever $u$ lies in a mesh cell $Q$, where $Q^+$ is a fixed enlargement of $Q$.  Summing over the Voronoi partition and grouping by mesh cells yields
\[
\|P_ng-g\|_{L^2}^2
\le C h_{U,n}^2
\sum_{Q\in\mathcal Q_{J_s}}|Q|
\|\nabla g\|_{L^\infty(Q^+)}^2.
\]
For fixed-order splines, the local polynomial inverse inequality \citep[Chapters~3 and~12]{Schumaker2007} gives $|Q|\|\nabla g\|_{L^\infty(Q^+)}^2 \le C\|\nabla g\|_{L^2(Q^{++})}^2$, with another fixed enlargement $Q^{++}$.  These enlargements have uniformly bounded overlap, and the global spline inverse inequality gives
\[
\|P_ng-g\|_{L^2}^2
\le C h_{U,n}^2\|\nabla g\|_{L^2}^2
\le C(J_sh_{U,n})^2\|g\|_{L^2}^2.
\]
Thus $\|P_ng-g\|_{L^2}=o(\|g\|_{L^2})$ uniformly over the spline space.  The triangle inequality proves $\|P_ng\|_{L^2}\asymp\|g\|_{L^2}$ and hence \eqref{eq:spatial-discrete-norm}.  Coefficient stability for normalized tensor-product B-splines \citep[Chapter~12]{Schumaker2007} gives $\|g_v\|_{L^2}^2\asymp J_s^{-d}\|v\|_2^2$, proving \eqref{eq:spatial-coefficient-norm}.

Finally, each $A_a$ is bounded by one and supported on a rectangle with volume $O(J_s^{-d})$ and side lengths of order $J_s^{-1}$.  A packing argument using the separation radius shows that the number of sites in this support is at most
\[
C\left(1+\frac{J_s^{-1}}{q_{U,n}}\right)^d
\le C m_nJ_s^{-d},
\]
where the last inequality uses $J_sh_{U,n}\to0$ and $q_{U,n}\asymp m_n^{-1/d}$.  Since $\pi_{i,n}\le C/m_n$ and $0\le A_a\le1$, both bounds in \eqref{eq:spatial-coordinate-load} follow.
\end{proof}

{
\begin{lemma}
\label[lemma]{lem:mean-hilbert-energy}
Let $Z$ be a centered random element of a separable real Hilbert space $\mathbb H$, with $E\|Z\|^2<\infty$ and covariance operator $\Sigma$. Suppose that, for every $h\in\mathbb H$,
\begin{align*}
E\exp\{\langle h,Z\rangle\}
\le
\exp\{K^2\langle h,\Sigma h\rangle/2\}.
\end{align*}
Then, for every $x>0$,
\begin{align*}
P\!\left(
\|Z\|^2>
2K^2\operatorname{tr}(\Sigma)
+
4K^2\|\Sigma\|_{\mathrm{op}}x
\right)
\le e^{-x}.
\end{align*}
\end{lemma}

\begin{proof}
The bound follows from \citet[Proposition~2.2]{MollenhauerSchillings2023} with variance proxy $K^2\Sigma$. We include a direct proof.
If $\|\Sigma\|_{\mathrm{op}}=0$, then $Z=0$ almost surely.
Otherwise, fix an orthonormal basis and let $P_r$ be projection onto
its first $r$ coordinates. Write $\Sigma_r=P_r\Sigma P_r$ on that
finite-dimensional space. For an independent $g\sim N(0,I_r)$,
Tonelli's theorem and the assumed linear moment-generating function (MGF) bound give
\begin{align*}
E\exp\{t\|P_rZ\|^2\}
&=
E_gE_Z\exp\{\sqrt{2t}\langle g,P_rZ\rangle\}\le
E_g\exp\{K^2t\langle g,\Sigma_rg\rangle\}=
\det(I_r-2K^2t\Sigma_r)^{-1/2}.
\end{align*}
For $0\le t\le(4K^2\|\Sigma\|_{\mathrm{op}})^{-1}$, every eigenvalue $u$ of $2K^2t\Sigma_r$ lies in $[0,1/2]$. Using $-\log(1-u)\le2u$, $E\exp\{t\|P_rZ\|^2\} \le \exp\{2K^2t\operatorname{tr}(\Sigma)\}$. Since $\|P_rZ\|^2\uparrow\|Z\|^2$, monotone convergence yields the
same inequality for $Z$.
Apply Markov's inequality with
$t=(4K^2\|\Sigma\|_{\mathrm{op}})^{-1}$ to obtain the result.
\end{proof}

\begin{lemma}
\label[lemma]{lem:mean-direct-max}
Suppress the array index $n$.
Assume \MainCref{ass:time,ass:tails} and the stated measurability convention. For fixed $N_i\ge1$, let $X_i(t)=X(S_i,t)$. Suppose that, for every $i,k,t,t'$, $|\operatorname{Cov}\{X_i(t),X_k(t')\}| \le C_X\kappa_{ik}$, where $(\kappa_{ik})$ is symmetric and nonnegative,
$\kappa_{ii}=1$, and $\sup_i\sum_k\kappa_{ik}\le\Theta$. Let $J\ge1$. For $q\le p$, let $B_q$ be deterministic functions satisfying
\begin{align*}
\sup_q\|B_q\|_\infty\le C_B,\qquad
\int |B_q|f_T+\int B_q^2f_T\le C_B/J.
\end{align*}
Let $w_{iq}$ be deterministic weights and let $d_{iq}$ be deterministic
uniformly bounded functions. Define
\begin{align*}
S_q=
\sum_iw_{iq}\sum_{j=1}^{N_i}B_q(T_{ij})
\{d_{iq}(T_{ij})+X_i(T_{ij})
+\sigma_i(T_{ij})\varepsilon_{ij}\}
-\sum_iw_{iq}N_i\int B_q(t)d_{iq}(t)f_T(t)\,dt.
\end{align*}
Put
\begin{align*}
M_1=M_{1,n},\quad
H_q=\sum_iw_{iq}^2N_i,\quad
F_q=\sum_iw_{iq}^2N_i^2,\quad
E_q=\max_iw_{iq}^2N_i,\quad
w_*=\max_{i,q}|w_{iq}|,
\end{align*}
and
\begin{align*}
V_q=
\frac{H_q}{J}
+
\frac{F_q-H_q+(\Theta-1)F_q}{J^2},
\qquad
\ell=\log(e+p),\qquad
L=\{\log(e+M_1+p)\}^{1/2}.
\end{align*}
For arrays with $M_1+p\to\infty$,
\begin{align*}
\max_{q\le p}|S_q|
=
O_p\!\left(
\max_q\sqrt{V_q\ell}
+
\ell\max_q\sqrt{\Theta E_q/J}
+
w_*\ell L
\right).
\end{align*}
The assertion permits fixed $p$.
\end{lemma}

\begin{proof}
We first justify the use of the joint linear sub-Gaussian condition
for integrated functionals. For deterministic $c_i\in L^2(f_T)$,
consider $I=\sum_i\int c_i(t)X_i(t)f_T(t)\,dt$. Introduce auxiliary independent samples from $f_T$, independent of
the field, and approximate each integral by an empirical average
of size $r$. Conditional on the field these averages are unbiased.
The expected squared approximation error, averaged over the auxiliary
samples and the field, is at most $(C/r)\sum_i\|c_i\|_{L^2(f_T)}^2$. Consequently, for every $r$, deterministic sampling points can be
selected so that the resulting finite-point linear combinations
converge to $I$ in $L^2(P)$.
The original joint variance-relative condition implies a centered
linear MGF bound with a uniform constant. Passing along an almost
surely convergent subsequence and using Fatou's lemma, together with
convergence of variances, proves the same variance-relative MGF bound
for $I$. No sample-path continuity is used.

Decompose the latent part of $S_q$ as $G_q+R_q$, where
\begin{align*}
G_q=\sum_iw_{iq}N_i\int B_q(t)X_i(t)f_T(t)\,dt
\end{align*}
and
\begin{align*}
R_q=\sum_iw_{iq}
\left\{
\sum_{j=1}^{N_i}B_q(T_{ij})X_i(T_{ij})
-
N_i\int B_q(t)X_i(t)f_T(t)\,dt
\right\}.
\end{align*}
The covariance bound, the $L^1$ support bound, and
$2|ab|\le a^2+b^2$ imply $\operatorname{Var}(G_q) \le C\Theta F_q/J^2 \le CV_q$, where the last inequality follows from $J\ge1$.
The integrated linear MGF bound therefore gives, for $x\ge1$,
\begin{align*}
P\left(\max_q|G_q|>C\max_q\sqrt{V_qx}\right)
\le Cp e^{-cx}.
\end{align*}

Conditional on the latent information $\mathcal H_n$, all observation
times remain independent. Moreover,
\begin{align*}
\operatorname{Var}(R_q\mid\mathcal H_n)
\le
Q_q:=
\sum_iw_{iq}^2N_i
\int B_q(t)^2X_i(t)^2f_T(t)\,dt.
\end{align*}
Let $Z_q= \bigl(|w_{iq}|\sqrt{N_i}\,B_qX_i\bigr)_i$ be a random element of
$\mathbb H=\bigoplus_iL^2(f_T)$, and let $\Sigma_q$ denote its covariance
operator. Then $Q_q=\|Z_q\|^2$ and $\operatorname{tr}(\Sigma_q)\le CH_q/J$. For $v=(v_i)_i\in\mathbb H$, Cauchy--Schwarz gives $\int|B_qv_i|f_T\le CJ^{-1/2}\|v_i\|_{L^2(f_T)}$. Thus, by the covariance bound and its row-sum bound,
\begin{align*}
\langle v,\Sigma_qv\rangle \le \frac{C}{J} \sum_{i,k}\kappa_{ik} |w_{iq}|\sqrt{N_i}|w_{kq}|\sqrt{N_k} \|v_i\|\,\|v_k\| \le \frac{C\Theta E_q}{J}\sum_i\|v_i\|^2.
\end{align*}
Hence $\|\Sigma_q\|_{\mathrm{op}}\le C\Theta E_q/J$. The integrated linear MGF argument established above applies to every
$\langle v,Z_q\rangle$. Lemma~\ref{lem:mean-hilbert-energy} and a union
bound therefore yield
\begin{align*}
P\left(
\exists q:\ Q_q>C\{H_q/J+\Theta E_qx/J\}
\right)
\le p e^{-x}.
\end{align*}

Set $L_0=C_0\{\log(e+M_1+p)\}^{1/2}$ and let $X_i^{(L_0)}$ be $X_i$ clipped pointwise to $[-L_0,L_0]$.
Define $R_q^{(L_0)}$ by replacing both the sampled values and the
integral in $R_q$ by their clipped counterparts.
Conditional on $\mathcal H_n$, its centered summands are independent,
bounded by $Cw_*L_0$, and have total variance at most $Q_q$.
Conditional Bernstein's inequality \citep[Theorem~2.10]{BoucheronLugosiMassart2013}, followed by the preceding norm
bound, gives
\begin{align*}
P\left(
\max_q|R_q^{(L_0)}|
>
C\left\{
\max_q\sqrt{H_qx/J}
+x\max_q\sqrt{\Theta E_q/J}
+w_*L_0x
\right\}
\right)
\le Cp e^{-cx}.
\end{align*}

The clipping replacement is uniform over coordinates without a factor
$p$. Indeed, boundedness of every $B_q$ implies
\begin{align*}
\max_q|R_q-R_q^{(L_0)}|
\le Cw_*\Bigg[
&\sum_{i,j}|X_i(T_{ij})-X_i^{(L_0)}(T_{ij})|+\sum_iN_i\int|X_i(t)-X_i^{(L_0)}(t)|f_T(t)\,dt
\Bigg].
\end{align*}
Pointwise sub-Gaussian tails and exogeneity of the times give $E\max_q|R_q-R_q^{(L_0)}| \le Cw_*M_1e^{-cL_0^2}=o(w_*)$ when $C_0$ is sufficiently large. Markov's inequality makes this
replacement negligible. No supremum-norm tail assumption on the
sample paths, and no union bound over a continuum of times, is used.

The centered deterministic-design contribution is a sum of independent
bounded variables with variance at most $CH_q/J$.
For the measurement-error contribution, independence over $(i,j)$,
bounded $\sigma_i$, the support bound on $B_q$, and the sub-Gaussian
error moments imply the Bernstein moment bounds with variance proxy $CH_q/J$ and scale $Cw_*$. A union bound therefore controls these two contributions by $C\left\{\max_q\sqrt{H_qx/J}+w_*x\right\}$ outside an event of probability at most $Cp e^{-cx}$.

Finally, fix $\varepsilon>0$ and take
$x=C_\varepsilon\log(e+p)$ with $C_\varepsilon$ sufficiently large.
The preceding exceptional probabilities are at most $\varepsilon+o(1)$.
Since $H_q/J\le V_q$, combining all contributions proves the asserted
$O_p$ bound, including when $p$ is fixed.
\end{proof}

\begin{corollary}
\label[corollary]{cor:mean-direct-absorption}
Under \MainCref{ass:design,ass:beta,ass:time,ass:tails}, apply \Cref{lem:mean-direct-max} with $p=p_n$, $M_1=M_{1,n}$, and $\kappa_{ik}=\beta_X(\|S_i-S_k\|_2)^{1/2}$ for $i\ne k$, $\kappa_{ii}=1$. Let $D_n=\delta_n^{-d},\quad N_*=\max_iN_i,\quad \Lambda_n=w_*N_*$. Use $b_n$ from \MainCref{sec:theory-stationary} with this value of $p_n$. Then \Cref{lem:mean-direct-max} gives
\begin{align*}
\max_{q\le p_n}|S_q|
=
O_p\!\left(
\max_q\sqrt{V_q\log(e+p_n)}
+
D_n\Lambda_nb_n^d\log(e+p_n)
\right).
\end{align*}
\end{corollary}

\begin{proof}
Rio's covariance inequality \citep[Theorem~1.1]{Rio1993} and \MainCref{ass:tails} give the required covariance bound. \Cref{lem:shell} and $\alpha_\beta>2d$ give $\Theta_n\le CD_n$. In the stated spatial regimes, $D_n\ge c>0$. Since $E_q\le w_*^2N_*$ and $J\ge1$, $\max_q\sqrt{\Theta_nE_q/J} \le Cw_*\sqrt{D_nN_*/J} \le CD_n\Lambda_n$. Also, $M_{1,n}\le m_nN_*=\Theta_nm_{\mathrm{eff},n}N_* \le CD_nN_*m_{\mathrm{eff},n}$. Write $R_n=e+p_n+m_{\mathrm{eff},n}$ and $\epsilon_0=2d/\alpha_\beta>0$, so $b_n^d=R_n^{\epsilon_0}$. Using $\log R\le C_{\epsilon_0}R^{\epsilon_0}$,
\begin{align*}
\log(e+M_{1,n}+p_n)
\le C\{\log R_n+\log(1+D_nN_*)\}
\le CD_nN_*b_n^d.
\end{align*}
Because $D_nN_*\ge c$ and $b_n\ge1$, it follows that $w_*\{\log(e+M_{1,n}+p_n)\}^{1/2} \le CD_n\Lambda_nb_n^d$. Substitution in Lemma~\ref{lem:mean-direct-max} proves the result.

\end{proof}

}

\begin{lemma}\label[lemma]{lem:mean-proj-bias}
Under \MainCref{ass:time,ass:mean-smooth},
\begin{equation*}
\sup_{t\in\left[0,1\right]}|m\left(t\right)-m_{J_m}\left(t\right)|\le C J_m^{-q_m^\ast}.
\end{equation*}
\end{lemma}

\begin{proof}
The spline approximation theorem for quasi-uniform meshes \citep[Chapter~XII]{deBoor2001} gives a vector $\tilde \bbeta_{m,J_m}\in\R^{K_m}$ such that $\tilde m_{J_m}\left(t\right):=\bB_m\left(t\right)^\top \tilde\bbeta_{m,J_m}$ satisfies
\begin{equation}\label{eq:mean-approx0}
\sup_{t\in\left[0,1\right]}|m\left(t\right)-\tilde m_{J_m}\left(t\right)|\le C J_m^{-q_m^\ast}.
\end{equation}
Since $\bbeta_{m,J_m}=Q_{m,J_m}^{-1}r_{m,J_m}$,
\begin{equation*}
\bbeta_{m,J_m}-\tilde\bbeta_{m,J_m}
=
Q_{m,J_m}^{-1}
\E\left[\bB_m\left(T\right)\left\{m\left(T\right)-\tilde m_{J_m}\left(T\right)\right\}\right].
\end{equation*}
Thus, for each $t$,
\begin{align*}
|m_{J_m}\left(t\right)-\tilde m_{J_m}\left(t\right)|
&=
\left|\bB_m\left(t\right)^\top Q_{m,J_m}^{-1}\E\left[\bB_m\left(T\right)\left\{m\left(T\right)-\tilde m_{J_m}\left(T\right)\right\}\right]\right|
\\
&\le
\norm{\bB_m\left(t\right)}_1\,\oinf{Q_{m,J_m}^{-1}}\,
\left\|\E\left[\bB_m\left(T\right)\left\{m\left(T\right)-\tilde m_{J_m}\left(T\right)\right\}\right]\right\|_\infty.
\end{align*}
By \Cref{lem:Bspline-basic,lem:gram-pop}, $\norm{\bB_m\left(t\right)}_1\le1$ and $\oinf{Q_{m,J_m}^{-1}}\le C J_m$.
For each coordinate $a$,
\begin{equation*}
\left|\E\left[B_{a,m}\left(T\right)\left\{m\left(T\right)-\tilde m_{J_m}\left(T\right)\right\}\right]\right|
\le \sup_u|m\left(u\right)-\tilde m_{J_m}\left(u\right)|\,\E B_{a,m}\left(T\right).
\end{equation*}
Because $B_{a,m}$ is supported on an interval of length $O\left(J_m^{-1}\right)$ and $f_T$ is bounded, $\E B_{a,m}\left(T\right)\le C J_m^{-1}$.
Therefore,
\begin{equation*}
\sup_{t\in\left[0,1\right]} |m_{J_m}\left(t\right)-\tilde m_{J_m}\left(t\right)|
\le C J_m\cdot J_m^{-1}\cdot J_m^{-q_m^\ast}
\le C J_m^{-q_m^\ast}.
\end{equation*}
Combining this bound with \eqref{eq:mean-approx0} proves the lemma.

\end{proof}

\begin{lemma}\label[lemma]{lem:mean-gram}
Under \MainCref{ass:time},
\begin{equation*}
\mnorm{\hat Q_{m,J_m}-Q_{m,J_m}}
=
O_p\!\left(
\sqrt{\frac{\log K_m}{M_{1,n}J_m}}+\frac{\log K_m}{M_{1,n}}
\right).
\end{equation*}
Consequently, if
\begin{equation}\label{eq:mean-gram-cond}
J_m\sqrt{\frac{\log K_m}{M_{1,n}J_m}}+J_m\frac{\log K_m}{M_{1,n}}\to0,
\end{equation}
then $\oinf{\hat Q_{m,J_m}^{-1}}=O_p\left(J_m\right)$, and
\begin{equation*}
\oinf{\hat Q_{m,J_m}^{-1}-Q_{m,J_m}^{-1}}
=
O_p\!\left(
J_m^2\sqrt{\frac{\log K_m}{M_{1,n}J_m}}+J_m^2\frac{\log K_m}{M_{1,n}}
\right).
\end{equation*}
\end{lemma}

\begin{proof}
Fix a coordinate pair $\left(a,b\right)$.
\begin{equation*}
\left(\hat Q_{m,J_m}-Q_{m,J_m}\right)_{ab}
=
\frac1{M_{1,n}}\sum_{i=1}^{m_n}\sum_{j=1}^{N_i}
\left[B_{a,m}\left(T_{ij}\right)B_{b,m}\left(T_{ij}\right)-\E\left\{B_{a,m}\left(T\right)B_{b,m}\left(T\right)\right\}\right].
\end{equation*}
The summands are independent over all $\left(i,j\right)$ because the time points are i.i.d., and each is bounded in absolute value by $1$.
If $|a-b|>p-1$, local support makes the product identically zero.
Otherwise,
\begin{equation*}
\Var\left\{B_{a,m}\left(T\right)B_{b,m}\left(T\right)\right\}
\le \E\left\{B_{a,m}\left(T\right)^2B_{b,m}\left(T\right)^2\right\}
\le C J_m^{-1},
\end{equation*}
because the overlap support has length $O\left(J_m^{-1}\right)$.
Thus, Bernstein's inequality \citep[Corollary~2.11]{BoucheronLugosiMassart2013} gives, for every $x>0$,
\begin{equation*}
\PP\left(|\left(\hat Q_{m,J_m}-Q_{m,J_m}\right)_{ab}|>x\right)
\le 2\exp\left\{-\frac{C M_{1,n}x^2}{J_m^{-1}+x}\right\}.
\end{equation*}
There are only $O\left(K_m\right)$ potentially nonzero entries, so take
\begin{equation*}
x=A\left(\sqrt{\frac{\log K_m}{M_{1,n}J_m}}+\frac{\log K_m}{M_{1,n}}\right)
\end{equation*}
and apply a union bound to obtain
\begin{equation*}
\PP\left(\mnorm{\hat Q_{m,J_m}-Q_{m,J_m}}>x\right)
\le {C K_m e^{-c\min(A^2,A)\log K_m}}
\end{equation*}
{
For each $\varepsilon>0$, the last bound is at most $\varepsilon$ for sufficiently large $A$, uniformly for $K_m\ge2$. This proves the max-norm bound.
}
Because the Gram matrix has fixed bandwidth, its $\ell_\infty$ operator norm is bounded by a constant multiple of the max norm; hence $\oinf{\hat Q_{m,J_m}-Q_{m,J_m}}=o_p\left(J_m^{-1}\right)$ under \eqref{eq:mean-gram-cond}.
The Neumann-series inversion identity then gives
\begin{equation*}
\hat Q_{m,J_m}^{-1}
=
Q_{m,J_m}^{-1}\left\{I+\left(\hat Q_{m,J_m}-Q_{m,J_m}\right)Q_{m,J_m}^{-1}\right\}^{-1},
\end{equation*}
which implies $\oinf{\hat Q_{m,J_m}^{-1}}=O_p\left(J_m\right)$.
Finally,
\begin{equation*}
\hat Q_{m,J_m}^{-1}-Q_{m,J_m}^{-1}
=
Q_{m,J_m}^{-1}\left(Q_{m,J_m}-\hat Q_{m,J_m}\right)\hat Q_{m,J_m}^{-1},
\end{equation*}
and the preceding bounds yield the last display.

\end{proof}

\begin{lemma}\label[lemma]{lem:mean-centered}
Under \MainCref{ass:time}, $\E\left(\xi_{i,J_m}\right)=0$ for $1\le i\le m_n$.
\end{lemma}

\begin{proof}
The exogeneity in \MainCref{ass:time}, the pointwise centering of $X$, and the centering and independence of the measurement errors imply $\E\left(Y_{ij}\mid T_{ij}\right)=m\left(T_{ij}\right)$.  Consequently, $\E\left(\xi_{i,J_m}\right)=N_i\E\left[\bB_m\left(T\right)\left\{m\left(T\right)-m_{J_m}\left(T\right)\right\}\right]=0$ by \MainEqref{eq:mean-orthogonality}.

\end{proof}

\begin{lemma}\label[lemma]{lem:mean-score-var}
{Under \MainCref{ass:design,ass:beta,ass:time,ass:tails,ass:mean-smooth}}, for every $1\le a\le K_m$,
\[
\Var(\xi_{i,J_m,a})
\le C\left\{\frac{N_i}{J_m}+\frac{N_i(N_i-1)}{J_m^2}\right\}.
\]
Define
\begin{equation}\label{eq:mean-score-variance-proxy}
v_{m,n}:=
\frac{1}{M_{1,n}J_m}
+
\frac{M_{2,n}+(\Theta_n-1)(M_{1,n}+M_{2,n})}{M_{1,n}^2J_m^2}.
\end{equation}
Then
\[
\Var\left(\frac1{M_{1,n}}\sum_{i=1}^{m_n}\xi_{i,J_m,a}\right)\le Cv_{m,n}.
\]
\end{lemma}

\begin{proof}
Fix $a$, write $B_a(t)=B_{a,m}(t)$, and set $R_{ij}=Y_{ij}-m_{J_m}(T_{ij})$.  By \Cref{lem:mean-centered}, the score is centered.  Exogeneity and the uniform moment bounds give
$\E\{B_a(T_{ij})^2R_{ij}^2\}\le CJ_m^{-1}$.  For $j\ne\ell$, independence of the two observation times and bounded conditional second moments give
\[
\left|\E\{B_a(T_{ij})B_a(T_{i\ell})R_{ij}R_{i\ell}\}\right|
\le C\{\E B_a(T)\}^2\le CJ_m^{-2}.
\]
This proves the same-site variance bound.

For $i\ne k$, projection orthogonality and independence of the observation times and errors from the latent process remove all cross-site terms except the latent-field product.  Rio's covariance inequality \citep[Theorem~1.1]{Rio1993} and \MainCref{ass:tails} yield
\[
|\Cov(\xi_{i,J_m,a},\xi_{k,J_m,a})|
\le C N_iN_kJ_m^{-2}\beta_X(\|S_i-S_k\|_2)^{1/2}.
\]
Using $2N_iN_k\le N_i^2+N_k^2$,
\[
\sum_{i\ne k}N_iN_k\beta_X(\|S_i-S_k\|_2)^{1/2}
\le(\Theta_n-1)\sum_iN_i^2
=(\Theta_n-1)(M_{1,n}+M_{2,n}).
\]
Adding the diagonal and off-diagonal same-site terms gives
\[
\frac{C}{M_{1,n}^2}
\left\{
\frac{M_{1,n}}{J_m}
+
\frac{M_{2,n}+(\Theta_n-1)(M_{1,n}+M_{2,n})}{J_m^2}
\right\}=Cv_{m,n},
\]
{
for the full normalized variance.
}
\end{proof}

\begin{lemma}\label[lemma]{lem:mean-score-max}
{Under \MainCref{ass:design,ass:beta,ass:time,ass:tails,ass:mean-smooth,ass:mean-leverage}}, {
take $p_n=K_m$ in the definition of $b_n$
}.  Then
\begin{equation}\label{eq:mean-score-max}
\left\|\frac1{M_{1,n}}\sum_{i=1}^{m_n}\xi_{i,J_m}\right\|_\infty
=
O_p\left(
\sqrt{v_{m,n}\log K_m}
+
\frac{\delta_n^{-d}b_n^d\log K_m}{m_n}
\right),
\end{equation}
where $v_{m,n}$ is defined in \eqref{eq:mean-score-variance-proxy}.
If $\delta_n^{-d}b_n^dJ_m\left\{\log K_m/\left[m_n\left(\Theta_n+J_m/\bar N_n\right)\right]\right\}^{1/2}\to0$, then \MainCref{ass:mean-leverage} further gives
\begin{equation}\label{eq:mean-score-max-reduced}
J_m\left\|\frac1{M_{1,n}}\sum_{i=1}^{m_n}\xi_{i,J_m}\right\|_\infty
=O_p\left[
\sqrt{\log K_m\left\{
\frac{J_m}{m_n\bar N_n}+\frac{1}{m_{\mathrm{eff},n}}
\right\}}
\right].
\end{equation}
\end{lemma}

\begin{proof}
{
In \Cref{cor:mean-direct-absorption}, take $J=J_m$, $w_{i,a}=M_{1,n}^{-1}$, $B_a=B_{a,m}$ and $d_{i,a}=m-m_{J_m}$. The score is centered by \Cref{lem:mean-centered}, and the residual is uniformly bounded by \Cref{lem:mean-proj-bias}. Here $H_a=M_{1,n}^{-1}$, $F_a=(M_{1,n}+M_{2,n})/M_{1,n}^2$, and $V_a=v_{m,n}$. By \MainCref{ass:mean-leverage}, $w_*\max_iN_i\le C/m_n$, proving \eqref{eq:mean-score-max}.
}
Moreover, \MainCref{ass:mean-leverage} implies $M_{1,n}+M_{2,n}=\sum_iN_i^2\le CM_{1,n}^2/m_n$. Applying this reduction to \eqref{eq:mean-score-variance-proxy}, multiplying \eqref{eq:mean-score-max} by $J_m$, and using the additional condition in the lemma prove \eqref{eq:mean-score-max-reduced}.
\end{proof}

For an interior point $u$, use the weights $W_i(u)$ and local observation mass $M_{1,n}(u)$ from \MainCref{sec:method-local}, and write
\begin{align*}
\bar m_n\left(u,t\right)
&:=\frac{1}{M_{1,n}\left(u\right)}\sum_{i=1}^{m_n}W_i\left(u\right)N_i\,m\left(S_i,t\right),\\
\bbeta_{m,J_m}^{\mathrm{loc}}\left(u\right)
&:=Q_{m,J_m}^{-1}\E\left\{\bB_m\left(T\right)\bar m_n\left(u,T\right)\right\},
\qquad
m_{J_m}^{\mathrm{loc}}\left(u,t\right)
:=\bB_m\left(t\right)^\top\bbeta_{m,J_m}^{\mathrm{loc}}\left(u\right),
\end{align*}
and let
\begin{align*}
\xi_{i,J_m}\left(u\right) :=W_i\left(u\right)\sum_{j=1}^{N_i}\bB_m\left(T_{ij}\right) \left\{Y_{ij}-m_{J_m}^{\mathrm{loc}}\left(u,T_{ij}\right)\right\}, \bar\xi_{i,J_m}\left(u\right) :=\xi_{i,J_m}\left(u\right)-\E\left\{\xi_{i,J_m}\left(u\right)\right\}.
\end{align*}
These quantities are used only in the deterministic and stochastic decompositions below.

\begin{lemma}\label[lemma]{lem:local-mean-det}
Under \MainCref{ass:local-mean,ass:time},
\begin{equation}\label{eq:local-mean-det1}
\sup_{t\in\left[0,1\right]}|\bar m_n\left(u_0,t\right)-M\left(u_0,t\right)|
\le C\left(a_n^{\gamma_u}+L_n^{-1}\right),
\end{equation}
\begin{equation}\label{eq:local-mean-det2}
\sup_{t\in\left[0,1\right]}|m_{J_m}^{\mathrm{loc}}\left(u_0,t\right)-\bar m_n\left(u_0,t\right)|
\le C J_m^{-q_m^\ast},
\end{equation}
and therefore
\begin{equation}\label{eq:local-mean-det3}
\sup_{t\in\left[0,1\right]}|m_{J_m}^{\mathrm{loc}}\left(u_0,t\right)-M\left(u_0,t\right)|
\le C\left(J_m^{-q_m^\ast}+a_n^{\gamma_u}+L_n^{-1}\right).
\end{equation}
\end{lemma}

\begin{proof}
To prove \eqref{eq:local-mean-det1}, write
\begin{align*}
|\bar m_n\left(u_0,t\right)-M\left(u_0,t\right)|
&\le
\frac{1}{M_{1,n}\left(u_0\right)}
\sum_{i=1}^{m_n}W_i\left(u_0\right)N_i\,|m\left(S_i,t\right)-M\left(U_i,t\right)|
\\
&\quad+
\frac{1}{M_{1,n}\left(u_0\right)}
\sum_{i=1}^{m_n}W_i\left(u_0\right)N_i\,|M\left(U_i,t\right)-M\left(u_0,t\right)|.
\end{align*}
The first term is bounded by $CL_n^{-1}$ by \MainCref{ass:local-mean}(iii).
For the second term, $W_i\left(u_0\right)\neq0$ implies $\norm{U_i-u_0}_2\le C a_n$, and \MainCref{ass:local-mean}(i) bounds it by $Ca_n^{\gamma_u}$.
Thus, \eqref{eq:local-mean-det1} follows.

For \eqref{eq:local-mean-det2}, \MainCref{ass:local-mean}(ii),(iv) imply that $\bar m_n\left(u_0,\cdot\right)$ is uniformly bounded in $C^{q_m,\nu_m}\left[0,1\right]$, with time-spline order $p\ge q_m^\ast+1$.
The deterministic spline approximation theorem (again \citet[Chapter~12]{deBoor2001}) gives a spline $\widetilde m_{J_m}$ such that $\sup_{t\in\left[0,1\right]}|\bar m_n\left(u_0,t\right)-\widetilde m_{J_m}\left(t\right)|\le C J_m^{-q_m^\ast}$.
Projection onto the spline space under the population Gram inner product preserves this order by the argument used for \Cref{lem:mean-proj-bias}; therefore, \eqref{eq:local-mean-det2} follows.
Finally, \eqref{eq:local-mean-det3} follows by combining \eqref{eq:local-mean-det1} and \eqref{eq:local-mean-det2}.

\end{proof}

\begin{lemma}\label[lemma]{lem:local-mean-gram}
Under \MainCref{ass:time},
\begin{equation*}
\mnorm{\hat Q_{m,J_m}\left(u_0\right)-Q_{m,J_m}}
=
O_p\!\left(
\sqrt{\frac{\log K_m}{M_{1,n}\left(u_0\right)J_m}}
+
\frac{\log K_m}{M_{1,n}\left(u_0\right)}
\right).
\end{equation*}
If
\begin{equation}\label{eq:local-mean-gram-cond}
J_m\left(
\sqrt{\frac{\log K_m}{M_{1,n}\left(u_0\right)J_m}}
+
\frac{\log K_m}{M_{1,n}\left(u_0\right)}
\right)\to0,
\end{equation}
then
\begin{align*}
\oinf{\hat Q_{m,J_m}\left(u_0\right)^{-1}}
&=O_p\left(J_m\right),\\
\oinf{\hat Q_{m,J_m}\left(u_0\right)^{-1}-Q_{m,J_m}^{-1}}
&=
O_p\!\left(
J_m^2\sqrt{\frac{\log K_m}{M_{1,n}\left(u_0\right)J_m}}
+
J_m^2\frac{\log K_m}{M_{1,n}\left(u_0\right)}
\right).
\end{align*}
\end{lemma}

\begin{proof}
Fix coordinates $\left(a,b\right)$.
Then
\begin{align*}
\left\{\hat Q_{m,J_m}\left(u_0\right)-Q_{m,J_m}\right\}_{ab}
&=
\frac{1}{M_{1,n}\left(u_0\right)}
\sum_{i=1}^{m_n}\sum_{j=1}^{N_i}
W_i\left(u_0\right)\\
&\quad{}
\left[
B_{a,m}\left(T_{ij}\right)B_{b,m}\left(T_{ij}\right)-\E\left\{B_{a,m}\left(T\right)B_{b,m}\left(T\right)\right\}
\right].
\end{align*}
The summands are independent over $\left(i,j\right)$, uniformly bounded, and have variance bounded by $C W_i\left(u_0\right)^2J_m^{-1}$ when the supports overlap and by $0$ otherwise.
Since $\sum_i W_i\left(u_0\right)^2N_i\le C M_{1,n}\left(u_0\right)$, Bernstein's inequality \citep[Corollary~2.11]{BoucheronLugosiMassart2013} and a union bound over the $O\left(K_m\right)$ nonzero entries yield the displayed max-norm rate.
The inverse bound follows from the same Neumann-series argument as in \Cref{lem:mean-gram}.

\end{proof}

\begin{lemma}\label[lemma]{lem:local-mean-score}
Assume \MainCref{ass:design,ass:beta,ass:time,ass:tails,ass:balanced,ass:local-mean,ass:local-mass}.  Define
\begin{align*}
A_{1,n}(u_0)&:=\sum_iW_i(u_0)^2N_i,\\
A_{2,n}(u_0)&:=\sum_iW_i(u_0)^2N_i(N_i-1),\\
A_{0,n}(u_0)&:=\sum_iW_i(u_0)^2N_i^2=A_{1,n}(u_0)+A_{2,n}(u_0),
\end{align*}
and
\begin{equation}\label{eq:local-score-variance-proxy}
v_{\mathrm{loc},n}(u_0):=
\frac{1}{M_{1,n}(u_0)^2}
\left\{
\frac{A_{1,n}(u_0)}{J_m}
+
\frac{A_{2,n}(u_0)+(\Theta_n-1)A_{0,n}(u_0)}{J_m^2}
\right\}.
\end{equation}
For every coordinate $a$,
\[
\Var\left\{\frac{1}{M_{1,n}(u_0)}\sum_i\bar\xi_{i,J_m,a}(u_0)\right\}
\le Cv_{\mathrm{loc},n}(u_0),
\]
{
Moreover,
}
\begin{equation}\label{eq:local-score-variance-simple}
v_{\mathrm{loc},n}(u_0)
\le C\left\{
\frac{1}{m_na_n^d\bar N_nJ_m}
+
\frac{1}{m_{\mathrm{eff},n}a_n^dJ_m^2}
\right\}.
\end{equation}
{
With $p_n=K_m$ in the definition of $b_n$,
}
\begin{equation}\label{eq:local-mean-score-max}
\left\|
\frac{1}{M_{1,n}(u_0)}\sum_i\bar\xi_{i,J_m}(u_0)
\right\|_\infty
=
O_p\left[
\sqrt{\log K_m\left\{
\frac{1}{m_na_n^d\bar N_nJ_m}
+
\frac{1}{m_{\mathrm{eff},n}a_n^dJ_m^2}
\right\}}
+
\frac{\delta_n^{-d}b_n^d\log K_m}{m_na_n^d}
\right].
\end{equation}
If $\delta_n^{-d}b_n^dJ_m\left\{a_n^{-d}\log K_m/\left[m_n\left(\Theta_n+J_m/\bar N_n\right)\right]\right\}^{1/2}\to0$, then
\begin{equation}\label{eq:local-mean-score-max-reduced}
J_m\left\|
\frac{1}{M_{1,n}(u_0)}\sum_i\bar\xi_{i,J_m}(u_0)
\right\|_\infty
=O_p\left[
\sqrt{\log K_m\left\{
\frac{J_m}{m_na_n^d\bar N_n}+\frac{1}{m_{\mathrm{eff},n}a_n^d}
\right\}}
\right].
\end{equation}
\end{lemma}

\begin{proof}
Fix $a$.  The deterministic residual $m(S_i,t)-m_{J_m}^{\mathrm{loc}}(u_0,t)$ is uniformly bounded by \MainCref{ass:local-mean}(iv) and \Cref{lem:local-mean-det}.  The same-site diagonal/off-diagonal calculation gives
\[
\Var\{\bar\xi_{i,J_m,a}(u_0)\}
\le CW_i(u_0)^2\left\{\frac{N_i}{J_m}+\frac{N_i(N_i-1)}{J_m^2}\right\}.
\]
For $i\ne k$, centering does not change covariance, and only the latent-field product remains.  Thus
\[
|\Cov\{\bar\xi_{i,J_m,a}(u_0),\bar\xi_{k,J_m,a}(u_0)\}|
\le CW_i(u_0)W_k(u_0)N_iN_kJ_m^{-2}
\beta_X(\|S_i-S_k\|_2)^{1/2}.
\]
Symmetrization gives
\[
\sum_{i\ne k}W_iW_kN_iN_k\beta_X(\|S_i-S_k\|_2)^{1/2}
\le(\Theta_n-1)\sum_iW_i^2N_i^2.
\]
{
This proves the variance bound with \eqref{eq:local-score-variance-proxy}.
}

Balanced counts and \MainCref{ass:local-mass} imply
\[
M_{1,n}(u_0)\asymp m_na_n^d\bar N_n,
\quad
A_{1,n}(u_0)\le Cm_na_n^d\bar N_n,
\quad
A_{0,n}(u_0)\le Cm_na_n^d\bar N_n^2.
\]
Since $A_{2,n}\le A_{0,n}$ and $m_n=\Theta_nm_{\mathrm{eff},n}$, these bounds give \eqref{eq:local-score-variance-simple}.

{
For the maximal inequality, apply \Cref{cor:mean-direct-absorption} with $J=J_m$ and
}
\[
w_{i,a}=\frac{W_i(u_0)}{M_{1,n}(u_0)},\qquad
{B_a}=B_{a,m},\qquad
d_{i,a}(t)=m(S_i,t)-m_{J_m}^{\mathrm{loc}}(u_0,t).
\]
{
Here $H_a=A_{1,n}(u_0)/M_{1,n}(u_0)^2$, $F_a=A_{0,n}(u_0)/M_{1,n}(u_0)^2$, and $V_a=v_{\mathrm{loc},n}(u_0)$. Balanced counts and local mass give $w_*\max_iN_i\le C(m_na_n^d)^{-1}$. Substitution gives \eqref{eq:local-mean-score-max}.
}  Multiplication by $J_m$ and the additional condition in the lemma give \eqref{eq:local-mean-score-max-reduced}.
\end{proof}

\begin{lemma}\label[lemma]{lem:spatial-spline-mean-bias-gram}
Under \MainCref{ass:time,ass:local-mean,ass:spatial-spline-mean,ass:spatial-spline-design} {
and \MainCref{ass:balanced}
}, use $Q_{M,J_s,J_m}$, $r_{M,J_s,J_m}$, $\btheta_{M,J_s,J_m}^{\circ}$, and $M_{J_s,J_m}^{\circ}$ from \MainCref{sec:method-spatial-spline}. Then
\begin{equation}\label{eq:spatial-spline-mean-bias}
\sup_{u\in[0,1]^d}\sup_{t\in[0,1]}
|M_{J_s,J_m}^{\circ}(u,t)-M(u,t)|
\le C\{J_s^{-q_u^\ast}+J_m^{-q_m^\ast}+L_n^{-1}\}.
\end{equation}
Furthermore, $\|Q_{M,J_s,J_m}^{-1}\|_\infty\le CJ_s^dJ_m$, and the population normal equation is
\begin{equation}\label{eq:spatial-spline-mean-pop-orth}
\sum_iN_i\int_0^1\bZ_{J_s,J_m}(U_i,x)
\{m(S_i,x)-M_{J_s,J_m}^{\circ}(U_i,x)\}f_T(x)\,dx=0.
\end{equation}
\end{lemma}

\begin{proof}
Tensor-product spline approximation \citep[Chapter~12]{Schumaker2007} gives a spline $\widetilde M_{J_s,J_m}(u,t)=\bZ_{J_s,J_m}(u,t)^\top\widetilde\btheta$ satisfying
\[
\sup_{u,t}|\widetilde M_{J_s,J_m}(u,t)-M(u,t)|
\le C\{J_s^{-q_u^\ast}+J_m^{-q_m^\ast}\}.
\]
Together with \MainCref{ass:local-mean}(iii),
\[
\max_i\sup_t|m(S_i,t)-\widetilde M_{J_s,J_m}(U_i,t)|
\le CR_n,
\qquad
R_n:=J_s^{-q_u^\ast}+J_m^{-q_m^\ast}+L_n^{-1}.
\]
The normal equations imply
\[
\btheta_{M,J_s,J_m}^{\circ}-\widetilde\btheta
=Q_{M,J_s,J_m}^{-1}\frac1{M_{1,n}}\sum_iN_i
\int\bZ_{J_s,J_m}(U_i,x)
\{m(S_i,x)-\widetilde M_{J_s,J_m}(U_i,x)\}f_T(x)\,dx.
\]
By \Cref{lem:spatial-spline-discrete-norm}, $M_{1,n}^{-1}\sum_iN_iA_a(U_i)\le CJ_s^{-d}$, while $\E B_{b,m}(T)\le CJ_m^{-1}$.  Hence the moment vector on the right is bounded coordinatewise by $CJ_s^{-d}J_m^{-1}R_n$.

Let
\[
Q_{s,J_s,n}:=\frac1{M_{1,n}}\sum_iN_i\bA_s(U_i)\bA_s(U_i)^\top.
\]
Equation \eqref{eq:spatial-coefficient-norm} gives $cJ_s^{-d}\|v\|_2^2\le v^\top Q_{s,J_s,n}v\le CJ_s^{-d}\|v\|_2^2$. Set $H_{s,n}=J_s^dQ_{s,J_s,n}$.  Its spectrum lies in a fixed interval $[c,C]\subset(0,\infty)$.  Let $G_{s,n}$ be its sparsity graph.  Local support of fixed-order tensor-product B-splines implies that every graph edge changes each tensor index by at most a fixed number; consequently, every graph ball of radius $r$ contains at most $C(1+r)^d$ vertices.

We use the polynomial approximation argument for sparse matrix functions in \citet{BenziRazouk2007}. Because $x\mapsto x^{-1}$ is analytic on $[c,C]$, there are constants $C_0<\infty$, $\rho\in(0,1)$, and polynomials $P_k$ of degree at most $k$ such that $\sup_{x\in[c,C]}|x^{-1}-P_k(x)|\le C_0\rho^k$. If the graph distance between basis indices $a$ and $a'$ is $r>k$, then $\{P_k(H_{s,n})\}_{aa'}=0$, because every power of degree at most $k$ connects only vertices at graph distance at most $k$.  Taking $k=r-1$ and applying the spectral calculus gives $|(H_{s,n}^{-1})_{aa'}|\le C_0\rho^{r-1}, \quad r\ge1$. The diagonal entries are bounded by $\|H_{s,n}^{-1}\|_2\le c^{-1}$, and entries between distinct connected components are zero.  Summing over graph-distance shells therefore yields
\[
\|H_{s,n}^{-1}\|_\infty
\le C\sum_{r=0}^\infty(1+r)^d\rho^r\le C,
\qquad
\|Q_{s,J_s,n}^{-1}\|_\infty\le CJ_s^d.
\]
Independence of $T$ from the site design gives
$Q_{M,J_s,J_m}=Q_{s,J_s,n}\otimes Q_{m,J_m}$.  Thus, by \Cref{lem:gram-pop},
\[
\|Q_{M,J_s,J_m}^{-1}\|_\infty
=\|Q_{s,J_s,n}^{-1}\|_\infty\|Q_{m,J_m}^{-1}\|_\infty
\le CJ_s^dJ_m.
\]
Combining this inverse bound with the coordinatewise moment bound proves \eqref{eq:spatial-spline-mean-bias}.  Equation \eqref{eq:spatial-spline-mean-pop-orth} is the defining normal equation.
\end{proof}

\begin{lemma}\label[lemma]{lem:spatial-spline-mean-score}
Assume \MainCref{ass:design,ass:beta,ass:time,ass:tails,ass:local-mean,ass:spatial-spline-mean,ass:spatial-spline-design} {
and \MainCref{ass:balanced}
}.  Let
\[
p_{M,n}:=K_sK_m,
\qquad
r_{M,Q,n}:=
\sqrt{\frac{J_s^{-d}\log p_{M,n}}{M_{1,n}J_m}}
+
\frac{\log p_{M,n}}{M_{1,n}},
\]
{
and take $p_n=p_{M,n}$ in the definition of $b_n$.
}  Then
$\|\hat Q_{M,J_s,J_m}-Q_{M,J_s,J_m}\|_{\max}=O_p(r_{M,Q,n})$.  If $J_s^dJ_mr_{M,Q,n}\to0$, then
\[
\|\hat Q_{M,J_s,J_m}^{-1}\|_\infty=O_p(J_s^dJ_m),
\qquad
\|\hat Q_{M,J_s,J_m}^{-1}-Q_{M,J_s,J_m}^{-1}\|_\infty
=O_p\{(J_s^dJ_m)^2r_{M,Q,n}\}=o_p(J_s^dJ_m).
\]
Use the centered score $\omega_{i,J_s,J_m}$ from \MainCref{sec:method-spatial-spline}. For $q=(a,b)$, set
\begin{align}
v_{\mathrm{sp},n,q}:={}&\frac{1}{M_{1,n}^2}\left[
\frac{1}{J_m}\sum_iA_a(U_i)^2N_i
+
\frac{1}{J_m^2}\left\{
\sum_iA_a(U_i)^2N_i(N_i-1)
\right.\right.\nonumber\\
&\left.\left.\hspace{4em}+(\Theta_n-1)\sum_iA_a(U_i)^2N_i^2
\right\}
\right].
\label{eq:spatial-score-variance-proxy}
\end{align}
{
The full normalized score variance is bounded by $Cv_{\mathrm{sp},n,q}$, and
}
\begin{equation}\label{eq:spatial-score-variance-simple}
v_{\mathrm{sp},n,q}
\le CJ_s^{-d}\left\{
\frac{1}{m_n\bar N_nJ_m}
+
\frac{1}{m_{\mathrm{eff},n}J_m^2}
\right\}.
\end{equation}
Moreover,
\begin{equation}\label{eq:spatial-spline-mean-score}
\left\|\frac1{M_{1,n}}\sum_i\omega_{i,J_s,J_m}\right\|_\infty
=
O_p\left[
\sqrt{J_s^{-d}\log p_{M,n}\left\{
\frac{1}{m_n\bar N_nJ_m}
+
\frac{1}{m_{\mathrm{eff},n}J_m^2}
\right\}}
+
\frac{\delta_n^{-d}b_n^d\log p_{M,n}}{m_n}
\right].
\end{equation}
If $\delta_n^{-d}b_n^dJ_m\left\{J_s^d\log\left(K_sK_m\right)/\left[m_n\left(\Theta_n+J_m/\bar N_n\right)\right]\right\}^{1/2}\to0$, then
\begin{equation}\label{eq:spatial-spline-mean-score-reduced}
J_s^dJ_m\left\|\frac1{M_{1,n}}\sum_i\omega_{i,J_s,J_m}\right\|_\infty
=O_p\left[
\sqrt{J_s^d\log p_{M,n}\left\{
\frac{J_m}{m_n\bar N_n}+\frac{1}{m_{\mathrm{eff},n}}
\right\}}
\right].
\end{equation}
\end{lemma}

\begin{proof}
For a nonzero Gram coordinate, the summand is bounded and its time-factor variance is $O(J_m^{-1})$. By \eqref{eq:spatial-coordinate-load}, the weighted sum of the spatial basis values is $O(J_s^{-d})$. Hence a normalized Gram coordinate has variance at most $CJ_s^{-d}/(M_{1,n}J_m)$. Bernstein's inequality \citep[Corollary~2.11]{BoucheronLugosiMassart2013} and a union bound over the $O(p_{M,n})$ nonzero entries give the stated max-norm rate. The fixed number of nonzero entries in each row, \Cref{lem:spatial-spline-mean-bias-gram}, the Neumann-series identity, and the resolvent identity give the inverse bounds.

For one score coordinate $q=(a,b)$, the deterministic residual is uniformly bounded by \Cref{lem:spatial-spline-mean-bias-gram}.  The same-site calculation gives
\[
\Var(\omega_{i,q})
\le CA_a(U_i)^2\left\{\frac{N_i}{J_m}+\frac{N_i(N_i-1)}{J_m^2}\right\}.
\]
For $i\ne k$, only the latent-field cross product remains. Rio's covariance inequality \citep[Theorem~1.1]{Rio1993} gives
\[
|\Cov(\omega_{i,q},\omega_{k,q})|
\le C|A_a(U_i)A_a(U_k)|N_iN_kJ_m^{-2}
\beta_X(\|S_i-S_k\|_2)^{1/2}.
\]
Symmetrization bounds the cross-site sum by
$(\Theta_n-1)\sum_iA_a(U_i)^2N_i^2$.  {
This proves the variance bound with \eqref{eq:spatial-score-variance-proxy}.
}

By \eqref{eq:spatial-coordinate-load},
$\sum_iA_a(U_i)^2N_i\le CM_{1,n}J_s^{-d}$.  Balanced counts then give
$\sum_iA_a(U_i)^2N_i^2\le C\bar N_nM_{1,n}J_s^{-d}$, and the same bound applies to the $N_i(N_i-1)$ sum.  Since $M_{1,n}\asymp m_n\bar N_n$ and $m_n=\Theta_nm_{\mathrm{eff},n}$, \eqref{eq:spatial-score-variance-simple} follows.

{
Finally, apply \Cref{cor:mean-direct-absorption} with $J=J_m$ and
}
\[
w_{i,q}=\frac{A_a(U_i)}{M_{1,n}},\qquad
{B_q}=B_{b,m},\qquad
d_{i,q}(t)=m(S_i,t)-M_{J_s,J_m}^{\circ}(U_i,t).
\]
{
The preceding spatial-load bounds give $H_q\le CJ_s^{-d}/(m_n\bar N_n)$, $F_q\le CJ_s^{-d}/m_n$, and $V_q=v_{\mathrm{sp},n,q}$. Bounded spatial B-splines and balanced counts give $w_*\max_iN_i\le C/m_n$. These bounds yield \eqref{eq:spatial-spline-mean-score}.
}  Multiplication by $J_s^dJ_m$ and the additional condition in the lemma give \eqref{eq:spatial-spline-mean-score-reduced}.
\end{proof}

\clearpage
\section{Proofs of main theorems}

\begin{proof}[Proof of \MainCref{thm:mean-rate}]

By \Cref{lem:mean-proj-bias},
\[
\sup_t|\hat m_{J_m}(t)-m(t)|
\le\sup_t|\hat m_{J_m}(t)-m_{J_m}(t)|+CJ_m^{-q_m^\ast}.
\]
The normal equations and \MainEqref{eq:mean-score-avg} give
\[
\hat\bbeta_{m,J_m}-\bbeta_{m,J_m}
=Q_{m,J_m}^{-1}\frac1{M_{1,n}}\sum_i\xi_{i,J_m}
+(\hat Q_{m,J_m}^{-1}-Q_{m,J_m}^{-1})\frac1{M_{1,n}}\sum_i\xi_{i,J_m}.
\]
Because $\sup_t\|\bB_m(t)\|_1=1$ and $\|Q_{m,J_m}^{-1}\|_\infty\le CJ_m$, \Cref{lem:mean-score-max} yields the leading stochastic bound
\[
O_p\left(
J_m\sqrt{v_{m,n}\log K_m}
+
\frac{J_m\delta_n^{-d}b_n^d\log K_m}{m_n}
\right).
\]
Using \eqref{eq:mean-score-variance-proxy},
\[
J_m^2v_{m,n}
=
\frac{J_m}{M_{1,n}}
+
\frac{M_{2,n}+(\Theta_n-1)(M_{1,n}+M_{2,n})}{M_{1,n}^2},
\]
which is the stochastic term in \MainEqref{eq:mean-main-rate}.  By \Cref{lem:mean-gram} and \eqref{eq:mean-gram-cond},
$\|\hat Q_{m,J_m}^{-1}-Q_{m,J_m}^{-1}\|_\infty=o_p(J_m)$; hence its product with the score is lower order.  Adding the approximation bias proves \MainEqref{eq:mean-main-rate}.

\end{proof}

\begin{proof}[Proof of \MainCref{cor:mean-phase}]

Under the additional condition in \MainCref{cor:mean-phase}, use the score bound in \eqref{eq:mean-score-max-reduced} and the bias bound in \Cref{lem:mean-proj-bias}. The inverse-Gram argument used for \MainCref{thm:mean-rate} then gives \MainEqref{eq:mean-phase-rate}. Under \MainCref{ass:balanced}, suppress logarithmic and lower-order terms. In the sparse time regime, balance
\[
J_m^{-q_m^\ast}\asymp\left(\frac{J_m}{m_n\bar N_n}\right)^{1/2},
\]
which gives $J_m\asymp(m_n\bar N_n)^{1/(2q_m^\ast+1)}$ and rate $(m_n\bar N_n)^{-q_m^\ast/(2q_m^\ast+1)}$.  This component meets the curve-level term $m_{\mathrm{eff},n}^{-1/2}$ when
\[
(m_n\bar N_n)^{-q_m^\ast/(2q_m^\ast+1)}
\asymp m_{\mathrm{eff},n}^{-1/2}.
\]
Since $m_n=\Theta_nm_{\mathrm{eff},n}$, this is equivalent to
$\bar N_n\asymp m_{\mathrm{eff},n}^{1/(2q_m^\ast)}/\Theta_n$.  Above the phase transition boundary, balancing $J_m^{-q_m^\ast}$ with $m_{\mathrm{eff},n}^{-1/2}$ gives $J_m\asymp m_{\mathrm{eff},n}^{1/(2q_m^\ast)}$ and the stated dense-time rate.

\end{proof}

\begin{proof}[Derivation of the local effective-sample-size equivalence]

Under \MainCref{ass:balanced,ass:local-mass},
\begin{align*}
M_{1,n}\left(u_0\right) \asymp \bar N_n\sum_iW_i\left(u_0\right) \asymp \bar N_nm_na_n^d, \sum_iW_i\left(u_0\right)^2N_i^2 \asymp \bar N_n^2\sum_iW_i\left(u_0\right)^2 \asymp \bar N_n^2m_na_n^d.
\end{align*}
Consequently, the definition gives $m_{u_0,\mathrm{eff},n}\asymp m_na_n^d/\Theta_n=m_{\mathrm{eff},n}a_n^d$.

\end{proof}

\begin{proof}[Proof of \MainCref{thm:local-mean-rate}]

By \Cref{lem:local-mean-det},
\[
\sup_t|m_{J_m}^{\mathrm{loc}}(u_0,t)-M(u_0,t)|
\le C\{J_m^{-q_m^\ast}+a_n^{\gamma_u}+L_n^{-1}\}.
\]
The local normal equations give
\begin{align*}
\hat\bbeta_{m,J_m}(u_0)-\bbeta_{m,J_m}^{\mathrm{loc}}(u_0)
&=Q_{m,J_m}^{-1}\frac1{M_{1,n}(u_0)}\sum_i\bar\xi_{i,J_m}(u_0)\\
&\quad+\{\hat Q_{m,J_m}(u_0)^{-1}-Q_{m,J_m}^{-1}\}
\frac1{M_{1,n}(u_0)}\sum_i\bar\xi_{i,J_m}(u_0).
\end{align*}
Since $\sup_t\|\bB_m(t)\|_1=1$ and $\|Q_{m,J_m}^{-1}\|_\infty\le CJ_m$, \eqref{eq:local-mean-score-max} gives
\[
O_p\left[
\sqrt{\log K_m\left\{
\frac{J_m}{m_na_n^d\bar N_n}
+
\frac{1}{m_{\mathrm{eff},n}a_n^d}
\right\}}
+
\frac{J_m\delta_n^{-d}b_n^d\log K_m}{m_na_n^d}
\right].
\]
By \Cref{lem:local-mean-gram} and \eqref{eq:local-mean-gram-cond}, the inverse-Gram difference is $o_p(J_m)$, so its score product is lower order.  Adding the deterministic bias proves \MainEqref{eq:local-mean-rate}; under the additional condition in \MainCref{cor:local-mean-phase}, \eqref{eq:local-mean-score-max-reduced} gives \MainEqref{eq:local-mean-rate-simple}.

\end{proof}

\begin{proof}[Proof of \MainCref{thm:local-mean-uniform}]
Use $\ell_n$ and $\mathcal U_n$ from the theorem, and write $\eta_n=a_n^{2d+2}/J_m$ for the mesh size. Then
\begin{equation*}
|\mathcal U_n|\le C\eta_n^{-d},
\qquad
\log\left\{e+K_m|\mathcal U_n|\right\}\le C\ell_n.
\end{equation*}
The uniform local-mass and balanced-count assumptions imply
$\inf_{v\in\mathcal U_n}m_{v,\mathrm{eff},n} \asymp m_{\mathrm{eff},n}a_n^d$.

{
Apply \Cref{cor:mean-direct-absorption} to the joint coordinates $q=(v,a)$ with $J=J_m$ and
}
\[
w_{i,q}=\frac{W_i(v)}{M_{1,n}(v)},\qquad
{B_q}=B_{a,m},\qquad
d_{i,q}(t)=m(S_i,t)-m_{J_m}^{\mathrm{loc}}(v,t).
\]
The uniform first- and second-order kernel-mass bounds in \MainCref{thm:local-mean-uniform}, together with balanced counts, give
{
$w_*\max_iN_i\le C(m_na_n^d)^{-1}$ uniformly over the grid. The bounds for $H_q$, $F_q$ and $V_q$ in \Cref{lem:local-mean-score} hold uniformly as well. Applying the corollary once with $p_n=K_m|\mathcal U_n|$ gives
}
\begin{align*}
\max_{v\in\mathcal U_n}
\left\|\frac{1}{M_{1,n}(v)}\sum_i\bar\xi_{i,J_m}(v)\right\|_\infty
=O_p\left[
\sqrt{\ell_n\left\{
\frac{1}{m_na_n^d\bar N_nJ_m}
+
\frac{1}{m_{\mathrm{eff},n}a_n^dJ_m^2}
\right\}}
+
\frac{\delta_n^{-d}b_n^d\ell_n}{m_na_n^d}
\right].
\end{align*}

Bernstein's inequality \citep[Corollary~2.11]{BoucheronLugosiMassart2013} and a union bound for the empirical Gram entries over the same grid give
\begin{equation*}
\max_{v\in\mathcal U_n}
\mnorm{\hat Q_{m,J_m}\left(v\right)-Q_{m,J_m}}
=
O_p\!\left(
\sqrt{\frac{\ell_n}{\left\{\inf_{u\in\mathcal U_\circ}M_{1,n}\left(u\right)\right\}J_m}}
+
\frac{\ell_n}{\inf_{u\in\mathcal U_\circ}M_{1,n}\left(u\right)}
\right).
\end{equation*}
Because these Gram matrices have fixed bandwidth, the uniform Gram condition in \MainCref{thm:local-mean-uniform} and the Neumann-series argument imply $\max_{v\in\mathcal U_n} \oinf{\hat Q_{m,J_m}\left(v\right)^{-1}}=O_p\left(J_m\right)$, with the inverse-Gram product lower order uniformly on the grid.  Combining this fact, the preceding score bound, and \Cref{lem:local-mean-det} proves the asserted rate on $\mathcal U_n$.

It remains to pass from the grid to all $u\in\mathcal U_\circ$.  For each $u$, choose $u^\sharp\in\mathcal U_n$ with $\|u-u^\sharp\|_2\le C\eta_n$.  Uniform local mass and balanced counts give $\inf_{u\in\mathcal U_\circ}M_{1,n}\left(u\right)\ge ca_n^dM_{1,n}$, while the Lipschitz property of $K$ yields $\max_i|W_i\left(u\right)-W_i\left(u^\sharp\right)|\le C\eta_n/a_n$.  Moreover, \MainCref{ass:tails} and Markov's inequality give $M_{1,n}^{-1}\sum_{i=1}^{m_n}\sum_{j=1}^{N_i}|Y_{ij}|=O_p\left(1\right)$.  Applying these bounds to the unnormalized Gram and moment sums, and also to the change in their normalizing local masses, gives
\begin{equation*}
\sup_{u\in\mathcal U_\circ}
\mnorm{\hat Q_{m,J_m}\left(u\right)-\hat Q_{m,J_m}\left(u^\sharp\right)}
=O_p\left(a_n/J_m\right),
\end{equation*}
\begin{equation*}
\sup_{u\in\mathcal U_\circ}
\left\|\hat r_{m,J_m}\left(u\right)-\hat r_{m,J_m}\left(u^\sharp\right)\right\|_\infty
=O_p\left(a_n/J_m\right).
\end{equation*}
Because every time-spline Gram matrix has fixed bandwidth, the first display also holds in the $\ell_\infty$ operator norm.  Indeed, the first normalization term is bounded by $C\eta_n/a_n^{d+1}$, and the largest term for the moment-vector normalization is $O_p\left(\eta_n/a_n^{2d+1}\right)$; the definition $\eta_n=a_n^{2d+2}/J_m$ makes both $O_p\left(a_n/J_m\right)$ or smaller.

The Gram interpolation bound and the grid inverse bound imply
\begin{equation*}
\sup_{u\in\mathcal U_\circ}
\oinf{\hat Q_{m,J_m}\left(u\right)^{-1}}=O_p\left(J_m\right).
\end{equation*}
Let $\rho_n$ denote the {stochastic terms} on the right-hand side of \MainEqref{eq:local-mean-uniform}.  The grid normal equations give $\max_{v\in\mathcal U_n}\|\hat\bbeta_{m,J_m}\left(v\right)\|_\infty =O_p\left(1+\rho_n\right)$. Indeed, the population local-projection coefficients are uniformly bounded because their moment vectors are $O\left(J_m^{-1}\right)$ coordinatewise and $\oinf{Q_{m,J_m}^{-1}}=O\left(J_m\right)$; the grid score bound controls the remaining coefficient error.  The identity
\begin{align*}
\hat\bbeta_{m,J_m}\left(u\right)-\hat\bbeta_{m,J_m}\left(u^\sharp\right)
&=
\hat Q_{m,J_m}\left(u\right)^{-1}\left[
\hat r_{m,J_m}\left(u\right)-\hat r_{m,J_m}\left(u^\sharp\right)
\right.
\\
&\left.\quad+
\left\{\hat Q_{m,J_m}\left(u^\sharp\right)-\hat Q_{m,J_m}\left(u\right)\right\}
\hat\bbeta_{m,J_m}\left(u^\sharp\right)
\right]
\end{align*}
therefore yields
\begin{equation*}
\sup_{u\in\mathcal U_\circ}\sup_{t\in\left[0,1\right]}
|\hat M_{J_m}\left(u,t\right)-\hat M_{J_m}\left(u^\sharp,t\right)|
=O_p\left\{a_n\left(1+\rho_n\right)\right\}.
\end{equation*}
Since $a_n\le a_n^{\gamma_u}$ and $a_n\rho_n\le\rho_n$, this interpolation term is absorbed by the displayed rate.  Finally, $\sup_{t\in\left[0,1\right]}|M\left(u,t\right)-M\left(u^\sharp,t\right)| \le C \eta_n^{\gamma_u}\le Ca_n^{\gamma_u}$ by \MainCref{ass:local-mean}(i).  Combining the grid and interpolation bounds proves \MainEqref{eq:local-mean-uniform}.

It remains to verify the inverse-Gram assertion away from the net.  The grid Gram bound, fixed bandwidth, and the interpolation bound above imply
\begin{align*}
\sup_{u\in\mathcal U_\circ}
\oinf{\hat Q_{m,J_m}\left(u\right)-Q_{m,J_m}}
&=o_p\left(J_m^{-1}\right),
\end{align*}
because the uniform Gram condition controls the grid deviation and $a_n/J_m=o\left(J_m^{-1}\right)$.  The resolvent identity and the uniform inverse bound therefore give
\begin{align*}
\sup_{u\in\mathcal U_\circ}
\oinf{\hat Q_{m,J_m}\left(u\right)^{-1}-Q_{m,J_m}^{-1}}
&=o_p\left(J_m\right).
\end{align*}
Define $S_n\left(u\right):=\hat r_{m,J_m}\left(u\right)-\hat Q_{m,J_m}\left(u\right)\bbeta_{m,J_m}^{\mathrm{loc}}\left(u\right)$.  This is the normalized local score.  For $\lambda_i\left(u\right):=W_i\left(u\right)N_i/M_{1,n}\left(u\right)$, the normalized-weight calculation used above gives $\sup_u\sum_i|\lambda_i\left(u\right)-\lambda_i\left(u^\sharp\right)|\le C\eta_n/a_n^{d+1}=O\left(a_n/J_m\right)$.  Since the population moment coordinates are $O\left(J_m^{-1}\right)$ and $\oinf{Q_{m,J_m}^{-1}}=O\left(J_m\right)$, it follows that
\begin{align*}
\sup_{u\in\mathcal U_\circ}
\left\|\bbeta_{m,J_m}^{\mathrm{loc}}\left(u\right)-
\bbeta_{m,J_m}^{\mathrm{loc}}\left(u^\sharp\right)\right\|_\infty
&=O\left(a_n/J_m\right).
\end{align*}
Together with those interpolation bounds and the grid score bound, this yields
\begin{align*}
\sup_{u\in\mathcal U_\circ}\left\|S_n\left(u\right)\right\|_\infty
&=O_p\left\{\left(\rho_n+a_n\right)/J_m\right\}.
\end{align*}
Since $\sup_t\|\bB_m\left(t\right)\|_1=1$, the off-net inverse-Gram product is consequently $o_p\left(\rho_n+a_n\right)$, which is absorbed by \MainEqref{eq:local-mean-uniform} because $a_n\le a_n^{\gamma_u}$.

\end{proof}

\begin{proof}[Proof of \MainCref{cor:local-mean-phase}]

Suppress logarithmic and lower-order terms.  The observation-limited balance
\[
J_m^{-q_m^\ast}\asymp\left(\frac{J_m}{m_na_n^d\bar N_n}\right)^{1/2}
\]
gives $J_m\asymp(m_na_n^d\bar N_n)^{1/(2q_m^\ast+1)}$ and rate $(m_na_n^d\bar N_n)^{-q_m^\ast/(2q_m^\ast+1)}$.  Equating this rate with the curve-level term $(m_{\mathrm{eff},n}a_n^d)^{-1/2}$ and using $m_n=\Theta_nm_{\mathrm{eff},n}$ gives the boundary in \MainEqref{eq:local-mean-phase-boundary}.  Above it, balancing the time-spline bias with the curve-level term gives $J_m\asymp(m_{\mathrm{eff},n}a_n^d)^{1/(2q_m^\ast)}$.

\end{proof}

\begin{proof}[Proof of \MainCref{thm:spatial-spline-mean-rate}]

By \Cref{lem:spatial-spline-mean-bias-gram},
\[
\sup_{u,t}|M_{J_s,J_m}^{\circ}(u,t)-M(u,t)|
\le C\{J_s^{-q_u^\ast}+J_m^{-q_m^\ast}+L_n^{-1}\}.
\]
The population normal equation and the least-squares normal equations give
\[
\hat\btheta-\btheta^\circ
=Q^{-1}\frac1{M_{1,n}}\sum_i\omega_i
+(\hat Q^{-1}-Q^{-1})\frac1{M_{1,n}}\sum_i\omega_i.
\]
Since tensor-product B-splines form a partition of unity and $\|Q^{-1}\|_\infty\le CJ_s^dJ_m$, \eqref{eq:spatial-spline-mean-score} gives the leading stochastic term
\[
O_p\left[
\sqrt{J_s^d\log(K_sK_m)\left\{
\frac{J_m}{m_n\bar N_n}+\frac{1}{m_{\mathrm{eff},n}}
\right\}}
+
\frac{J_s^dJ_m\delta_n^{-d}b_n^d\log(K_sK_m)}{m_n}
\right].
\]
The inverse-difference bound in \Cref{lem:spatial-spline-mean-score} is $o_p(J_s^dJ_m)$ under the displayed Gram condition, so its score product is lower order.  Adding the deterministic approximation proves \MainEqref{eq:spatial-spline-mean-rate}.  Under the additional condition in \MainCref{cor:spatial-spline-mean-phase}, \eqref{eq:spatial-spline-mean-score-reduced} gives the same bound without the final {remainder term}.  Boundary knots on the full cube avoid any interior restriction.

\end{proof}

\begin{proof}[Proof of \MainCref{cor:spatial-spline-mean-phase}]

Fix $J_s$ and suppress logarithmic and lower-order terms.  In the observation-limited regime, balance
\[
J_m^{-q_m^\ast}\asymp
\left(\frac{J_s^dJ_m}{m_n\bar N_n}\right)^{1/2},
\]
which gives $J_m\asymp(m_n\bar N_n/J_s^d)^{1/(2q_m^\ast+1)}$ and the stated rate.  Equating it with the curve-level term $(J_s^d/m_{\mathrm{eff},n})^{1/2}$ yields
$\bar N_n\asymp(m_{\mathrm{eff},n}/J_s^d)^{1/(2q_m^\ast)}/\Theta_n$.  Above the boundary, balancing the time bias with the curve-level term gives $J_m\asymp(m_{\mathrm{eff},n}/J_s^d)^{1/(2q_m^\ast)}$.

For joint sparse-time minimization, balance
\[
J_s^{-q_u^\ast},\qquad J_m^{-q_m^\ast},\qquad
\left(\frac{J_s^dJ_m}{m_n\bar N_n}\right)^{1/2}.
\]
Writing the common rate as $R$ gives $J_s\asymp R^{-1/q_u^\ast}$ and $J_m\asymp R^{-1/q_m^\ast}$, hence
$R^{2+d/q_u^\ast+1/q_m^\ast}\asymp(m_n\bar N_n)^{-1}$, which is the stated anisotropic exponent.  In the dense-time regime, balancing $J_s^{-q_u^\ast}$ with $(J_s^d/m_{\mathrm{eff},n})^{1/2}$ gives the stated spatial rate.

\end{proof}

\begin{proof}[Derivation of the local-kernel and spatial-spline comparison]

Set $a_n\asymp J_s^{-1}$.  Apart from logarithmic factors, the squared stochastic terms in \MainEqref{eq:local-mean-rate} become
\begin{align*}
\frac{J_m}{m_na_n^d\bar N_n}+\frac{1}{m_{\mathrm{eff},n}a_n^d}
&\asymp
J_s^d\left(\frac{J_m}{m_n\bar N_n}+\frac{1}{m_{\mathrm{eff},n}}\right).
\end{align*}
For the uniform local rate, $\ell_n=\log K_m+\log\left(a_n^{-1}\right)\asymp\log\left(K_sK_m\right)$ because $K_s\asymp J_s^d$.  {
The comparison concerns the leading stochastic terms after the respective remainder conditions are imposed; $b_n$ depends on the coordinate dimension and need not have the same order for the two procedures.
}  The local phase transition boundary likewise becomes
$\bar N_n\asymp\left(m_{\mathrm{eff},n}/J_s^d\right)^{1/(2q_m^\ast)}/\Theta_n$, the spatial-spline phase transition boundary.  The {
leading stochastic terms
} therefore agree at matched spatial resolution, while the spatial approximation terms remain $a_n^{\gamma_u}$ and $J_s^{-q_u^\ast}$, respectively.

\end{proof}

\end{document}